\documentclass[leqno,letterpaper]{amsart}
\usepackage{graphicx, amsmath, amssymb, amsthm, hyperref, verbatim, multicol, mathrsfs, tikz, caption, subcaption}
\usepackage{stmaryrd}
\usepackage{mathtools}  
\usepackage{cleveref}   
\usepackage{todonotes}
\usepackage{tikz-cd}
\usepackage{physics}
\usepackage{enumitem}   

\newtheorem{thm}{Theorem}[section]
\newtheorem{lemm}[thm]{Lemma}
\newtheorem{cor}[thm]{Corollary}

\newtheorem{prop}[thm]{Proposition}

\theoremstyle{definition}

\newtheorem{rem}[thm]{Remark}

\theoremstyle{remark} 
\newtheorem{case}{Case}

\makeatletter
\newcommand*\rel@kern[1]{\kern#1\dimexpr\macc@kerna}
\newcommand*\widebar[1]{%
  \begingroup
  \def\mathaccent##1##2{%
    \rel@kern{0.8}%
    \overline{\rel@kern{-0.8}\macc@nucleus\rel@kern{0.2}}%
    \rel@kern{-0.2}%
  }%
  \macc@depth\@ne
  \let\math@bgroup\@empty \let\math@egroup\macc@set@skewchar
  \mathsurround\z@ \frozen@everymath{\mathgroup\macc@group\relax}%
  \macc@set@skewchar\relax
  \let\mathaccentV\macc@nested@a
  \macc@nested@a\relax111{#1}%
  \endgroup
}
\makeatother

\DeclareMathOperator{\diam}{diam}
\DeclareMathOperator{\dist}{dist}
\DeclareMathOperator{\Int}{int}
\DeclareMathOperator{\Mod}{mod}
\DeclareMathOperator{\loc}{loc}

\DeclareMathOperator{\mesh}{mesh}

\newcommand{\N}{\mathbb{N}}

\newcommand{\R}{\mathbb{R}}
\newcommand{\C}{\mathbb{C}}

\newcommand{\Z}{\mathbb{Z}}

\newcommand{\br}[1]{\overline{#1}}
\makeatletter

\newcommand{\Rom}[1]{\expandafter\@slowromancap\romannumeral #1@}

\numberwithin{equation}{section}

\begin{document}

\title[Polyhedral approximation and uniformization]{Polyhedral approximation of metric surfaces and applications to uniformization}

\author{Dimitrios Ntalampekos}
\author{Matthew Romney}
\address{Mathematics Department, Stony Brook University, Stony Brook NY, 11794, USA.}
\email{dimitrios.ntalampekos@stonybrook.edu}
\email{matthew.romney@stonybrook.edu}

\date{\today}
\thanks{The first author is partially supported by NSF Grant DMS-2000096}

\maketitle

\begin{abstract}
 We prove that any length metric space homeomorphic to a $2$-manifold with boundary, also called a length surface, is the Gromov--Hausdorff limit of polyhedral surfaces with controlled geometry. As an application, using the classical uniformization theorem for Riemann surfaces and a limiting argument, we establish a general ``one-sided'' quasiconformal uniformization theorem for length surfaces with locally  finite Hausdorff $2$-measure. Our approach yields a new proof of the Bonk--Kleiner theorem characterizing Ahlfors $2$-regular quasispheres. 
\end{abstract} 

\setcounter{tocdepth}{1}
\tableofcontents

\section{Introduction}

\subsection{Polyhedral approximation}

\textit{Surfaces of bounded curvature} were introduced in the 1940s by A.\ D.\ Alexandrov as a generalization of Riemannian 2-manifolds. They provide a natural setting to develop the intrinsic geometry of surfaces. See monographs by Alexandrov--Zalgaller \cite{AZ:67} and Reshetnyak \cite{Res:93} for overviews of this subject. A foundational result, due to Alexandrov \cite{Ale:48},  is that any surface of bounded curvature is the uniform limit of polyhedral surfaces of uniformly bounded curvature.

The objective of this paper is an analogous theorem on polyhedral approximation for arbitrary length surfaces. Instead of the property of bounded curvature, we find it most useful to consider the behavior of the Hausdorff 2-measure $\mathcal{H}^2$. By \textit{length surface}, we mean a length metric space homeomorphic to a 2-manifold, with or without boundary. A \textit{polyhedral surface} is a surface formed by gluing locally finitely many planar polygonal faces isometrically along edges, equipped with the induced length metric. Such a surface is locally flat except in a discrete set of vertex points. We say that a sequence of maps $f_n \colon X_n \to Y_n$, $n \in \mathbb{N}$, between metric spaces is an \textit{approximately isometric sequence} if $f_n$ is a $\varepsilon_n$-isometry for some $\varepsilon_n>0$ for all $n \in \mathbb{N}$, where $\varepsilon_n \to 0$ as $n \to \infty$. We refer to \Cref{sec:prelim} for more detailed definitions. 
\begin{thm} \label{thm:main}
Let $X$ be a length surface. There exists a sequence of polyhedral surfaces $\{X_n\}_{n=1}^\infty$ each homeomorphic to $X$ such that the following properties hold for an absolute constant $K \geq 1$.  
\begin{enumerate}[label=\normalfont(\arabic*)]
    \item \label{item:main_1} There exists an approximately isometric sequence of maps $f_n \colon X_n \to X$, $n \in \mathbb{N}$. Moreover, each $f_n$ is a topological embedding.
    \item \label{item:main_2} For each {compact} set $A \subset X$, \[\limsup_{n \to \infty} \mathcal{H}^2(f_n^{-1}(A)) \leq K \mathcal{H}^2(A).\] 
\end{enumerate} 
\end{thm}

In particular, the sequence $\{X_n\}_{n=1}^\infty$ converges in the Gromov--Hausdorff sense to $X$ \cite[Corollary 7.3.28]{BBI:01}. This theorem should be compared to the standard fact that any length surface is the Gromov--Hausdorff limit of locally finite embedded graphs; see Proposition 7.5.5 and the following exercise in \cite{BBI:01}. By filling in such a graph with polyhedral surfaces so that the length metric on the graph remains unchanged, one obtains a sequence of polyhedral surfaces also converging in the Gromov--Hausdorff sense to the original surface. The point of \Cref{thm:main} is to find approximating surfaces whose geometry is controlled by that of the original space. Compare also \Cref{thm:main} to the classical theorem of Bing \cite[Theorem 7]{Bin:57} that any topological surface in a 3-manifold $M$ may be uniformly approximated by homeomorphic polyhedral surfaces in the ambient space $M$.

There are three conceptual ingredients in the proof of \Cref{thm:main}. The first is a recent result on the existence of {decompositions of arbitrary length surfaces into non-overlapping convex triangular regions} by Creutz and the second-named author in \cite{CR:21}. 
The second is the following fact about bi-Lipschitz embedding metric triangles into the {Euclidean} plane, {denoted here by $\mathbb{C}$,} which we state in more generality than what is needed for \Cref{thm:main}. Despite its simple statement and proof, it appears to be new. By \textit{metric triangle}, we mean a metric space consisting of three points, called \textit{vertices}, and three closed arcs, called \textit{edges}, each isometric to an interval, connecting these vertices pairwise.  Note that this definition allows the edges to intersect at interior points.

\begin{prop} \label{prop:BL_embedding}
Every metric triangle is $L$-bi-Lipschitz embeddable in $\mathbb{C}$ with $L=4$.
\end{prop} 

The third ingredient needed to prove \Cref{thm:main} is a variant of the Besicovitch inequality. See \cite[Section 13.2]{Pet:20} for a statement of this result, including a version for metric spaces (Exercise 13.22). The classical Besicovitch inequality states that the minimal Riemannian filling of a planar Jordan curve is the Jordan domain that it bounds. {The precise result we need} is given as \Cref{thm:besicovitch} below. 

The outline of the proof of \Cref{thm:main} is as follows. We start with a sufficiently fine {triangular decomposition} $\mathcal{T}$ of the surface $X$ and corresponding edge graph $\mathcal{E}(\mathcal{T})$, equipped with the induced length metric. For each {triangular region} $T \in \mathcal{T}$, we {use \Cref{prop:BL_embedding} to obtain a bi-Lipschitz embedding $F\colon \partial T \to \mathbb{C}$.} Using this embedding, we build a polyhedral surface $\widetilde{T}$ of Hausdorff $2$-measure {comparable to the area of the region bounded by $F(\partial T)$} with the property that the length metric on {$\partial \widetilde {T}$ is no smaller than the metric on $\partial T$.} \Cref{prop:BL_embedding} together with the Besicovitch inequality imply that the Hausdorff 2-measure of $\widetilde{T}$ is not too much larger than that of $T$. The polyhedral surfaces $\widetilde{T}$ are then glued together according to the edge graph $\mathcal{E}(\mathcal{T})$ {to form the surface $X_n$. In other words, we build $X_n$ by replacing each triangular region $T \subset X$ with the corresponding polyhedral surface $\widetilde{T}$. Our construction guarantees that $X$ and $X_n$ are approximately isometric.}

\subsection{Uniformization of surfaces}

In the second part of this paper, we give applications to the \textit{uniformization problem} for surfaces. This asks for the existence of geometrically well-behaved parametrizations of metric surfaces in the spirit of the classical uniformization theorem for Riemann surfaces. The classical uniformization theorem states that any simply connected Riemann surface can be mapped conformally onto either the complex plane, the open unit disk or the $2$-sphere. In the setting of metric spaces, conformality is a restrictive requirement, and it is more appropriate to consider instead some notion of quasiconformal mapping. 

Any orientable polyhedral surface can be given the structure of a Riemann surface compatible with its metric. As a result, \Cref{thm:main} {gives a new approach to proving} uniformization-type theorems for metric surfaces by invoking the classical uniformization theorem together with a limiting argument. Our main result on this topic, Theorem \ref{thm:one-sided_qc}, gives the existence of ``one-sided'' quasiconformal parametrizations in great generality.  

For $K\geq 1$, we say that a mapping $h\colon X\to Y$ between two metric surfaces of locally finite Hausdorff $2$-measure is \textit{weakly $K$-quasiconformal} if it is continuous, surjective, and monotone and if it satisfies the modulus inequality 
\begin{equation} \label{equ:qc_inequality}
    \Mod \Gamma \leq K \Mod h(\Gamma)
\end{equation}
for every path family $\Gamma$ in $X$; here $\Mod$ refers to the $2$-modulus. Recall that a continuous map between topological spaces is \textit{monotone} if the preimage of each point is connected. By a result of Youngs \cite{You:48}, monotone mappings between $2$-manifolds are precisely the uniform limits of homeomorphisms. Inequality \eqref{equ:qc_inequality} is commonly referred to as the \textit{$K_O$-inequality}, and a map $h$ satisfying \eqref{equ:qc_inequality} is said to have \textit{bounded outer dilatation}.

\begin{thm} \label{thm:one-sided_qc}
Let $X$ be a length surface of locally finite Hausdorff 2-measure homeomorphic to $\widehat{\C}$, $\overline{\mathbb{D}}$, or $\C$. Then there is a weakly $K$-quasi\-conformal mapping $h\colon \Omega\to X$ for $K = 4/\pi$, where $\Omega$ is either $\widehat{\C}$, $\overline{\mathbb{D}}$, or $\mathbb{D}$ or $\mathbb{C}$, respectively.
\end{thm}

Here, $\mathbb D$ denotes the open unit disk in the complex plane $\C$, and $\widehat{\C}$ is the Riemann sphere, with the spherical metric and measure. To prove \Cref{thm:one-sided_qc}, it is enough to find a weakly $K$-quasiconformal mapping $h$ for some $K \geq 1$. This value can be improved to the constant $K = 4/\pi$ using the argument in \cite[Section 14]{Raj:17} or \cite{Rom:19b}. The constant $4/\pi$ is sharp, as can be shown using the example of the $\ell^\infty$-metric on $\mathbb{R}^2$; see Example 2.2 in \cite{Raj:17}.

{\Cref{thm:one-sided_qc} is motivated by the question of finding minimal assumptions required for producing a uniformizing parametrization of a metric surface. In particular, it gives an affirmative answer to Question 1.1 in \cite{IR:20}, attributed to Rajala and Wenger, under the mild assumption that the metric on $X$ is a length metric. We discuss the relation between \Cref{thm:one-sided_qc} and previous results on the uniformization problem later in this section.} 

Equivalently, we can replace \eqref{equ:qc_inequality} in the definition of weak quasiconformality  by the statement that $h\in N^{1,2}_{\loc}(X,Y)$ and the pointwise distortion inequality $g_h(x)^2 \leq K J_h(x)$ holds for almost every $x \in X$. Here, $g_h$ is the minimal weak upper gradient of $h$ and $J_h$ is the Jacobian of $h$, that is, the Radon--Nikodym derivative of the measure $\mathcal H^2\circ h$ with respect to $\mathcal H^2$. 

\begin{thm}\label{theorem:definitions_qc_1}
Let $X,Y$ be metric surfaces of locally finite Hausdorff $2$-measure and $K\geq 1$. A continuous, surjective, and monotone mapping $h\colon X\to Y$ is weakly $K$-quasiconformal if and only if $h\in N^{1,2}_{\loc}(X,Y)$ and 
$$g_h(x)^2 \leq K J_h(x)$$
for a.e.\ $x\in X$. 
\end{thm}

In the case that $h$ is a homeomorphism, this result follows from {a theorem of Williams} \cite{Wil:12}. We prove the equivalence in the case of monotone mappings in Section \ref{sec:further_properties}. One of the technicalities here is to justify existence of the Jacobian.

We note that in the case that $X$ is homeomorphic to $\C$ in Theorem \ref{thm:one-sided_qc}, there is no clear distinction between the situations where $\Omega = \mathbb D$ and where $\Omega=\mathbb C$, as the following example shows. 

\begin{prop}\label{prop:example}
There exists a length surface $X$ of locally finite Hausdorff $2$-measure, homeomorphic to $\C$, admitting weakly quasiconformal parametrizations by both $\mathbb{D}$ and $\C$.
\end{prop}

This contrasts with uniformization by quasiconformal mappings, since $\C$ is not quasiconformally equivalent to any proper subdomain. We present this example in \Cref{sec:examples}, where we also discuss other examples. 

As a corollary to \Cref{thm:one-sided_qc}, we obtain a result on the existence of \textit{minimal disks} or \textit{solutions to Plateau's problem} in metric spaces. This topic has been studied in great depth by Lytchak--Wenger and collaborators in \cite{FW:21,GW:20,LW:17,LW:17b,LW:18a}. Following \cite{LW:17}, for a given metric space and Jordan curve $\Gamma \subset X$ we let $\Lambda(\Gamma, X)$ denote the family of maps in the Sobolev space $N^{1,2}(\mathbb{D}, X)$ whose trace is a monotone parametrization of $\Gamma$. A \textit{solution to Plateau's problem} for $\Gamma$ is a map in $\Lambda(\Gamma, X)$ having minimal {parametrized} area and minimal (Reshetnyak) energy among area minimizers. {See the references above for more complete definitions.} It is shown in \cite{LW:17} that every Jordan curve in a complete proper metric space $X$ can be spanned by a minimal disk provided that $\Lambda(\Gamma, X)$ is non-empty. Lytchak and Wenger rely on the assumption that $X$ satisfies a quadratic isoperimetric inequality to guarantee that $\Lambda(\Gamma, X)$ is indeed non-empty for any rectifiable curve $\Gamma$. 

In the case where $X$ is a length surface and $\Gamma$ bounds a closed disk, \Cref{thm:one-sided_qc} allows us to remove this dependency on the quadratic isoperimetric inequality. Instead, we require only that the Hausdorff 2-measure is finite. Note as well that we do not require $\Gamma$ to be rectifiable.

\begin{cor} \label{cor:plateau}
Let $X$ be a length surface of finite Hausdorff 2-measure homeomorphic to a closed disk and let $\Gamma = \partial X$. The family $\Lambda(\Gamma,X)$ is non-empty. Consequently, Plateau's problem for $\Gamma$ has a solution.
\end{cor}

Finally, we use \Cref{thm:one-sided_qc} to give a new proof of the well-known Bonk--Kleiner theorem characterizing Ahlfors 2-regular quasispheres, i.e., metric spaces quasisymmetrically equivalent to the standard $2$-sphere. {See \Cref{sec:qs} for definitions of the terms here.}
\begin{cor}[Bonk--Kleiner theorem] \label{cor:bonk-kleiner}
Let $X$ be a metric space homeomorphic to $\widehat{\C}$ that is Ahlfors 2-regular. Then there is a quasisymmetric homeomorphism from $X$ onto $\widehat{\C}$ if and only if $X$ is linearly locally connected. 
\end{cor}
Since this result was originally proved by Bonk and Kleiner in \cite{BK:02}, alternative proofs have been given by Rajala \cite{Raj:17} and Lytchak--Wenger \cite{LW:20}. We now give a brief summary of the three approaches. The basic common step to all these proofs is to produce a mapping (or sequence of mappings) and to use the geometric assumptions to show that the mapping is indeed a quasisymmetric homeomorphism (or that the sequence subconverges to a quasisymmetric homeomorphism). Thus the main difference is how such a mapping is produced.  

In the original proof \cite{BK:02}, Bonk--Kleiner use the geometric assumptions to find an embedded graph that approximates the original space $X$ at a given scale. They apply the Andreev--Koebe--Thurston circle packing theorem to produce a map from the vertex set of this graph into $\widehat{\C}$. These maps subconverge to a quasisymmetric homeomorphism from the whole space to $\widehat{\mathbb{C}}$. Next, in \cite{Raj:17}, Rajala obtains the Bonk--Kleiner theorem as a consequence a general uniformization theorem for quasiconformal mappings. The proof is based on the construction of a harmonic function and corresponding conjugate function on an arbitrary quadrilateral. Pairing these functions gives a quasiconformal homeomorphism from this quadrilateral onto a rectangle in the plane. Rajala's proof is especially notable in that he carries out this construction essentially from scratch. Finally, in the Lytchak--Wenger proof \cite{LW:20}, the existence of the required mapping is provided by the authors' solution to Plateau's problem in metric spaces satisfying a quadratic isoperimetric inequality in \cite{LW:17}, \cite{LW:18a}. 

Our approach, in turn, establishes the Bonk--Kleiner theorem as a consequence of the classical uniformization theorem for Riemann surfaces. In particular, our proof gives a direct connection between the classical uniformization theorem and contemporary work on the uniformization of metric surfaces. That \Cref{thm:one-sided_qc} implies the Bonk--Kleiner theorem is standard; see Theorem 4.9 in \cite{HK:98} and Section 16 in \cite{Raj:17}. The idea is that the assumption that $X$ is Ahlfors 2-regular and linearly locally connected allows one to promote the map $h$ in \Cref{thm:one-sided_qc} to a quasisymmetric homeomorphism.

In addition to the results already mentioned, the uniformization problem has also been studied for metric surfaces of other topological type \cite{GW:18,Iko:22,Wil:08,Wil:10}. One ingredient in \cite{GW:18} and \cite{Iko:22} is the use of the classical uniformization theorem to pass from local quasiconformal or quasisymmetric charts to a globally defined mapping. In contrast, our proof uses the classical uniformization theorem to handle both the local and global aspects of the problem.

Finally, a version of \Cref{thm:one-sided_qc} has been proved concurrently and independently by Meier and Wenger in \cite{MW:21} using a different method, building on the machinery for studying Plateau's problem in \cite{LW:17} and related papers. They also derive the Bonk--Kleiner theorem as a consequence, {along with additional applications}.

\subsection{Outline of the paper} 

In \Cref{sec:prelim}, we review terminology and background related to metric geometry and analysis in metric spaces. Next, \Cref{sec:bl_embedding_triangles} contains the proof of \Cref{prop:BL_embedding} on bi-Lipschitz embeddings of metric triangles in the plane. In \Cref{sec:filling}, we give the construction of polyhedral fillings for any simple metric triangle. The proof of \Cref{thm:main} is then presented in \Cref{sec:approximating_surfaces}. Next, in \Cref{sec:uniformization}, we prove \Cref{thm:one-sided_qc}, \Cref{cor:plateau}, and \Cref{cor:bonk-kleiner} giving our applications to the uniformization problem. In \Cref{sec:further_properties}, we investigate further the regularity properties of the parametrizations in \Cref{thm:one-sided_qc} and prove Theorem \ref{theorem:definitions_qc_1}. Finally, Section \ref{sec:examples} contains several examples, including the example used to prove \Cref{prop:example}.

\subsection*{Acknowledgments} We are especially thankful to Paul Creutz, Alexander Lytchak and Kai Rajala for various conversations about the ideas in this paper. We also thank the anonymous referees for their valuable comments, which improved the exposition.
\bigskip

\section{Preliminaries} \label{sec:prelim}

\subsection{Metric geometry}

We assume that the reader is familiar with the basics of metric geometry as presented, for example, in \cite{BH:99} and \cite{BBI:01}. We recall a few definitions of particular interest. In the following, $X$ and $Y$ will denote metric spaces, with $d_X$ and $d_Y$ the respective metrics. We assume that all metrics are finite valued.  The Euclidean norm on the plane $\mathbb{C}$ is denoted by $|\cdot|$. 

A \textit{path} or \textit{curve} is a continuous map $\gamma\colon I \to X$, where $I$ is a compact interval. The length of the path $\gamma$ is denoted by $\ell(\gamma)$, or by $\ell_{d_X}(\gamma)$ to clarify the metric being used. The trace of $\gamma$, i.e., the set $\gamma(I)$, is denoted by $|\gamma|$. The metric space $X$ is a \textit{length space} if $d_X(x,y) = \inf_\gamma \ell(\gamma)$ for all $x,y \in X$, the infimum taken over all paths $\gamma$ whose trace contains $x$ and $y$. The metric space $X$ is  \textit{quasiconvex} if there exists $C\geq 1$ such that any two points $x,y \in X$ are in the image of a path $\gamma\colon I \to X$ satisfying $\ell(\gamma)\leq Cd_X(x,y)$. A path $\gamma$ between points $x,y \in X$ is a \emph{geodesic} if $\ell(\gamma) = d_X(x,y)$. {A subset $A \subset X$ is \textit{convex} if any two points in $A$ can be joined by a geodesic in $A$. In this case, $A$ is a length space with the restriction of the metric on $X$ and the inclusion map from $A$ to $X$ is an isometric embedding.} The diameter of a set $A \subset X$ is denoted by $\diam (A)$, or by $\diam_{d_X} (A)$ to specify the metric being used.

For any metric space and $s>0$, the \textit{Hausdorff $s$-measure} of a set $A \subset X$ is defined by
\[\mathcal{H}^s(A) = \lim_{\delta\to 0} \mathcal H^s_{\delta}(A),\]
where
\[\mathcal H^s_\delta(A) =\inf \left\{ \sum_{j=1}^\infty C(s) \diam(A_j)^s\right\} \]
and the infimum is taken over all collections of sets $\{A_j\}_{j=1}^\infty$ such that $A \subset \bigcup_{j=1}^\infty A_j$ and $\diam(A_j) < \delta$ for each $j$. Here $C(s)$ is a positive normalization constant, chosen so that the Hausdorff $n$-measure coincides with Lebesgue measure in $\R^n$. The quantity $\mathcal H^s_{\delta}(A)$ is called the \textit{$\delta$-Hausdorff $s$-content of $A$}. If we need to emphasize the metric $d_X$ being used for the Hausdorff $s$-measure, we write $\mathcal{H}_{d_X}^s$ instead of $\mathcal{H}^s$. 

A map $f\colon X \to Y$ between metric spaces is \textit{bi-Lipschitz} if there exists $L\geq 1$ such that
\[L^{-1}d_X(x,y) \leq d_Y(f(x),f(y)) \leq L d_X(x,y)\]
for all $x,y \in X$. In this case, we say that $f$ is \textit{$L$-bi-Lipschitz}. A map $f\colon X \to Y$ is \textit{co-Lipschitz} if the first of these inequalities holds for all $x,y \in X$, and \textit{Lipschitz} if the second of these inequalities holds for all $x,y \in X$. In these cases, we say that $f$ is, respectively, \textit{$L$-co-Lipschitz} and \textit{$L$-Lipschitz}.

We use $\partial X$ to denote the boundary of a manifold $X$ and $\Int(X)$ to denote its interior. Throughout this paper, the terms \textit{boundary} and \textit{interior} refer to manifold boundary and interior rather than topological boundary and interior. The following theorem can be viewed as a consequence of the Besicovitch inequality for metric spaces; see Exercise 13.25 in \cite[Section 13.F]{Pet:20}.
\begin{thm} \label{thm:besicovitch}
Let $X$ be a metric space homeomorphic to a closed topological disk with boundary $\partial X$. If $\Omega\subset \R^2$ is a closed Jordan domain such that for some $L>0$ there exists an $L$-Lipschitz map $f\colon {\partial X}\to \partial \Omega$ of non-zero topological degree, then
\[\mathcal{H}^2(X) \geq  \frac{\pi}{4L^2} \mathcal{H}^2(\Omega).\]
\end{thm}
The inequality is optimal, as one can see by taking $X$ to be the unit square $[0,1]^2$ with the $\ell^\infty$ metric and $\Omega=[0,1]^2$  (with the {Euclidean} metric). 
\begin{proof}
Since the $\ell^\infty$-metric {does not exceed} the Euclidean $\ell^2$-metric on $\R^2$, it follows that $f\colon (\partial X,d)\to (\R^2,\ell^\infty)$ is also an $L$-Lipschitz embedding. By  the McShane--Whitney extension theorem {(see \cite[Theorem 2.3]{Hei:05})}, there exists an $L$-Lipschitz extension $\widetilde f\colon (X,d) \to (\R^2,\ell^\infty)$. Namely, if we write $f=(f_1,f_2)$, then {define} $\widetilde f=(\widetilde f_1,\widetilde f_2)$ {by} 
$$\widetilde f_i(x)= \inf_{y\in \partial X}\{f_i(y)+L d(x,y)\}$$
for $i=1,2$. Since $\widetilde f|_{\partial X}\colon \partial X\to \partial \Omega$ has non-zero degree, it follows that $\widetilde f(X)\supset \Omega$. 
Moreover, since $\widetilde f$  is $L$-Lipschitz, it follows that $\mathcal H^2_{\ell^\infty}(\Omega) \leq L^2 \mathcal H^2 (X)$, directly from the definition of Hausdorff $2$-measure. Finally, we have $\mathcal H^2_{\ell^\infty}= (\pi/4)\mathcal H^2_{\ell^2}$; see \cite[Lemma 6]{Kir:94} or \cite[pp.~2--3]{EBC:21} for a proof of this fact.
\end{proof}

\bigskip

\subsection{Gromov--Hausdorff convergence}
Let $X$ be a metric space and let $E\subset X$ and $\varepsilon>0$. We denote by $N_{\varepsilon}(E)$ the open $\varepsilon$-neighborhood of $E$. We say that $E$ is \textit{$\varepsilon$-dense} (in $X$) if for each $x\in X$ we have $d(x,E)<\varepsilon$ or equivalently $N_{\varepsilon}(E)=X$. A map $f \colon X \to Y$ (not necessarily continuous) between metric spaces is an \textit{$\varepsilon$-isometry} if $f(X)$ is $\varepsilon$-dense in $Y$ and $|d_X(x,y) - d_Y(f(x),f(y))| < \varepsilon$ for each $x,y \in X$.

We define the \textit{Hausdorff distance} of two sets $E,F\subset X$ to be the {infimal value} $r>0$ such that $E\subset N_r(F)$ and $F\subset N_r(E)$. We denote the Hausdorff distance by $d_H(E,F)$. A sequence of sets $E_n\subset X$ \textit{converges in the Hausdorff sense} to a set $E\subset X$ if $d_H(E_n,E)\to 0$ as $n\to\infty$. It is immediate that the diameters of $E_n$ converge to the diameter of $E$.

{The \textit{Gromov--Hausdorff distance} between two metric spaces $X,Y$ is defined as the infimal value $r>0$ such that there is a metric space $Z$ with subsets $\widetilde{X}, \widetilde{Y} \subset Z$ such that $X$ and $Y$ are isometric to $\widetilde{X}$ and $\widetilde{Y}$, respectively, and $d_H(\widetilde{X}, \widetilde{Y}) < r$. This is denoted by $d_{GH}(X,Y)$. We say that a sequence of metric spaces $X_n$ \textit{converges in the Gromov--Hausdorff sense} to a metric space $X$ if $d_{GH}(X_n,X) \to 0$ as $n \to \infty$. By \cite[Corollary 7.3.28]{BBI:01}, this is equivalent to the property that} there exists a sequence of $\varepsilon_n$-isometries $f_n\colon X_n\to X$, where $\varepsilon_n>0$ and $\varepsilon_n\to 0$ as $n\to\infty$. In this case, we say that $f_n$ is an \textit{approximately isometric sequence}.

We collect some immediate properties of Gromov--Hausdorff convergence. 

\begin{prop}\label{prop:gh}
Let $\{X_n\}_{n=1}^\infty$ be a sequence of compact metric spaces converging in the Gromov--Hausdorff sense to a compact metric space $X$, and consider an approximately isometric sequence $f_n\colon X_n\to X$.
\begin{enumerate}[label=\normalfont(\roman*)]
    \item\label{prop:arzela_ascoli} Suppose that $\gamma_n\colon[0,1]\to X_n$ is a sequence of paths, parametrized by rescaled arc length, such that 
    $$\liminf_{n\to\infty} \ell(\gamma_n)<\infty.$$
    Then there is a subsequence of $f_n\circ \gamma_n\colon [0,1]\to X$ that converges uniformly to a path $\gamma\colon[0,1]\to X$ with
    $$\ell(\gamma)\leq \liminf_{n\to\infty} \ell(\gamma_n).$$

    \item\label{prop:lift} Suppose, in addition, that each space $X_n$ is a length space. Then for each path $\gamma\colon[0,1] \to X$ and for each sequences of points $a_n,b_n\in X_n$ with $\lim_{n\to\infty}f_n(a_n)=\gamma(0)$ and $\lim_{n\to\infty}f_n(b_n)=\gamma(1)$ there exists a sequence of paths $\gamma_n\colon [0,1]\to X_n$ such that $\gamma_n(0)=a_n$, $\gamma_n(1)=b_n$, and $f_n\circ \gamma_n$ converges uniformly to $\gamma$. 
    
    \item\label{prop:connected} For each sequence of compact sets $E_n\subset X_n$ there exists a subsequence $E_{k_n}$ such that $f_{k_n}(E_{k_n})$ converges in the Hausdorff sense to a compact set $E\subset X$ and $\diam (E_{k_n})$ converges to $\diam (E)$. Moreover, if each set $E_n$ is connected, then $E$ is also connected.
\end{enumerate}
\end{prop}

The proof of the proposition is elementary, based on the definitions, and the experienced reader can safely skip it. Alternatively, one can prove the statement by embedding isometrically the sequence $\{X_n\}_{n=1}^\infty$ and the space $X$ to a common space $\mathcal X$ and thus reducing Gromov--Hausdorff convergence to Hausdorff convergence in $\mathcal X$; see \cite[Property 5.23]{Pet:20}.

\begin{proof}
{By assumption,} each map $f_n$ is an $\varepsilon_n$-isometry, where $\varepsilon_n\to 0$ as $n\to\infty$.

First we prove \ref{prop:arzela_ascoli}, which follows from a version of the Arzel\`a--Ascoli theorem \cite[Theorem 2.5.14]{BBI:01}. Consider the curves $\gamma_n$ as in the statement. Then for each $p,q\in [0,1]$ we have $d_{X_n}(\gamma_n(p),\gamma_n(q)) \leq \ell(\gamma_n) |p-q|$. Since $f_n$ is a $\varepsilon_n$-isometry, we have
$$d_{X}(f_n(\gamma_n(p)), f_n(\gamma_n(q)) ) < \varepsilon_n+ d_{X_n}(\gamma_n(p),\gamma_n(q)) \leq \varepsilon_n+ \ell(\gamma_n)|p-q|.$$
By passing to a subsequence, we assume that $L=\lim_{n\to\infty}\ell(\gamma_n)<\infty$. This implies that for each $\varepsilon>0$ there exists $\delta>0$ and $N\in \N$ such that for all $n\geq N$ and $|p-q|<\delta$, we have 
$$d_{X}(f_n(\gamma_n(p)), f_n(\gamma_n(q)) )<\varepsilon.$$
Hence, the mappings $f_n\circ \gamma_n\colon [0,1]\to X$ are uniformly equicontinuous. Since $X$ is compact, by the Arzel\`a--Ascoli theorem, there exists a subsequence converging uniformly to a map $\gamma\colon [0,1]\to X$ with the property that 
$$d_X(\gamma(p),\gamma(q)) \leq L|p-q|$$
for every $p,q\in [0,1]$. Hence, $\gamma$ is a rectifiable path with $\ell(\gamma)\leq L$.

Next, we prove \ref{prop:lift}. By the uniform continuity of $\gamma$, for each $n\in \N$ there exists $\delta_n>0$ such that if $|p-q|<\delta_n$, then $d_X(\gamma(p),\gamma(q))<1/n$. We pick a finite set $Q_n\subset [0,1]$ that contains $0$ and $1$ so that each of the complementary intervals of $Q_n$ has length less than $\delta_n$. We define $\gamma_n(0)=a_n$ and $\gamma_n(1)=b_n$. By the definition of an $\varepsilon_n$-isometry, for each $q\in Q_n\setminus \{0,1\}$ there exists a point $\gamma_n(q)\in X_n$ such that $d_X(f_n(\gamma_n(q)), \gamma(q)) <\varepsilon_n$. This defines a map $\gamma_n\colon Q_n \to X_n$. If $(q_1,q_2)$ is a complementary interval of $Q_n$, we define $\gamma_n$ on $[q_1,q_2]$ to be a geodesic in $X_n$ with endpoints $\gamma_n(q_1)$ and $\gamma_n(q_2)$. This gives a path $\gamma_n\colon [0,1]\to X_n$. For each $p\in [0,1]$ there exists a complementary interval $(q_1,q_2)$ of $Q_n$ whose closure contains $p$. If $q_1,q_2 \notin \{0,1\}$, then
\begin{align*} 
    d_X( \gamma(p), f_n(\gamma_n(p))) &\leq d_X(\gamma(p),\gamma(q_1)) + d_X(\gamma(q_1), f_n(\gamma_n(q_1)))\\&\qquad\qquad+ d_X( f_n(\gamma_n(q_1)), f_n(\gamma_n(p))) \\
    &\leq 1/n + \varepsilon_n +\varepsilon_n+d_{X_n} (\gamma_n(q_1),\gamma_n(p)).
\end{align*}
{Since $\gamma_n$ is a geodesic, it follows that}
\begin{align*}
    d_X( \gamma(p), f_n(\gamma_n(p))) &\leq 1/n+2\varepsilon_n+d_{X_n} (\gamma_n(q_1),\gamma_n(q_2))\\
    &\leq 1/n+3\varepsilon_n+d_X( f_n(\gamma_n(q_1)), f_n(\gamma_n(q_2)))\\
    &\leq 1/n+3\varepsilon_n+ 2\varepsilon_n+d_X (\gamma(q_1),\gamma(q_2)) \\
    &\leq 2/n+5\varepsilon_n.
\end{align*}
If $q_1=0$, then in the same way we obtain the estimate
\begin{align*}
    d_X( \gamma(p), f_n(\gamma_n(p)))\leq 2/n+3\varepsilon_n + 2d_{X}(\gamma(0),f_n(\gamma_n(0)))
\end{align*}
and an analogous estimate holds if $q_2=1$. By assumption, the quantities $d_{X}(\gamma(0),f_n(\gamma_n(0)))$, $d_{X}(\gamma(1),f_n(\gamma_n(1)))$, and $\varepsilon_n$ converge to $0$ as $n\to\infty$. Hence, $f_n\circ \gamma_n$ converges uniformly to $\gamma$, as desired.

For part \ref{prop:connected}, the existence of the set $E$ as the Hausdorff limit of a subsequence of $f_n(E_n)$ follows from \cite[Theorem 7.3.8, p.~253]{BBI:01}, which asserts that the space of compact subsets of a compact metric space is compact in the Hausdorff topology. The convergence of the diameters is also immediate from the properties of Hausdorff convergence and the fact that $|\diam(E_n)- \diam(f_n(E_n))|\to 0$, since $f_n$ is an $\varepsilon_n$-isometry. We now show the connectedness of $E$. After passing to a subsequence, we assume that $f_n(E_n)$ converges to $E$. Suppose, on the contrary that $E$ is disconnected. Then there exists a continuous non-constant function $\varphi\colon E \to \{0,1\}$. We define $F_0=\varphi^{-1}(0)$ and $F_1=\varphi^{-1}(1)$. These are non-empty, compact, and disjoint subsets of $X$, so they have a positive distance $\delta>0$. We fix a large $n$ so that $f_n(E_n)\subset N_{\delta/4}(E)$ and $E\subset N_{\delta/4}(f_n(E_n))$. Now, we define a function $\varphi_n\colon E_n\to \{0,1\}$ by
$$\varphi_n(x)= \begin{cases} 1, & f_n(x)\in N_{\delta/4}(F_1) \\ 0, & f_n(x)\in N_{\delta/4}(F_0)\end{cases}.$$
We note that $\varphi_n$ is non-constant, since $F_0,F_1\subset N_{\delta/4}(f_n(E_n))$. Moreover, $\varphi_n$ is continuous for large $n$. Indeed, if $x,y\in E_n$ and $d_{X_n}(x,y)<\varepsilon_n$, then $d_X(f_n(x),f_n(y)) <2\varepsilon_n$. We choose a large $n$ so that  $2\varepsilon_n<\delta/2$. Then both $f_n(x)$ and $f_n(y)$ have to lie in either $N_{\delta/4}(F_0)$ or $N_{\delta/4}(F_1)$. Thus, $\varphi_n(x)=\varphi_n(y)$ and continuity follows. The existence of $\varphi_n$ contradicts the connectedness of $E_n$.  
\end{proof}

\begin{lemm}\label{lemma:boundary_convergence}
Let $X$ be a length space homeomorphic to a closed topological disk and $\{X_n\}_{n=1}^\infty$ be a sequence of length spaces homeomorphic to $X$. Suppose that there exists an approximately isometric sequence $f_n\colon X_n\to X$ of topological embeddings. Then 
\begin{align*}
    \liminf_{n\to\infty}\diam(\partial X_n) \geq \diam( \partial X).
\end{align*}
\end{lemm}
In fact, the result holds without the assumption that $f_n$ is a topological embedding and one actually gets convergence of the diameters, but we do not need this generality here; see \cite[Section 7.5.2]{BBI:01} for such considerations.

\begin{proof}
Suppose that each $f_n$ is an $\varepsilon_n$-isometry, where $\varepsilon_n\to 0$. We claim that $\partial X \subset N_{\varepsilon_n}(f_n(\partial X_n))$, which implies the desired statement. To see this, note that $\partial X\subset N_{\varepsilon_n}(f_n(X_n))$ by the definition of an $\varepsilon_n$-isometry. Thus, if $x\in \partial X$, then there exists $y\in f_n(X_n)$ such that $d(x,y)<\varepsilon_n$. Consider a geodesic in $X$ connecting $x$ and $y$. Then there exists a point $z\in \partial f_n(X_n)$ lying on that geodesic such that $d(x,z)\leq d(x,y) <\varepsilon_n$. Finally, note that $\partial f_n(X_n)=f_n(\partial X_n)$, since $f_n$ is an embedding.
\end{proof} 

\bigskip 

\subsection{Modulus}
Let $X$ be a metric space and $\Gamma$ be a family of curves in $X$. A Borel function $\rho\colon X \to [0,\infty]$ is \textit{admissible} for the path family $\Gamma$ if $\int_{\gamma}\rho\, ds\geq 1$
for all locally rectifiable paths $\gamma\in \Gamma$. We define the \textit{$2$-modulus} of $\Gamma$ as 
$$\Mod \Gamma = \inf_\rho \int_X \rho^2 \, d\mathcal H^2,$$
where the infimum is taken over all admissible functions $\rho$ for $\Gamma$. By convention, $\Mod \Gamma = \infty$ if there are no admissible functions for $\Gamma$. Observe that we consider $X$ to be equipped with the Hausdorff 2-measure. This definition may be generalized by allowing for an exponent different from $2$ or a different measure, though this generality is not needed for this paper.

Let $X$ be a metric space. For each pair of disjoint continua $E,F\subset X$, we define $\Gamma^*(E,F;X)$ to be the family of rectifiable curves in $X\setminus (E\cup F)$ separating $E$ from $F$. That is, for each $\gamma\in \Gamma^*(E,F;X)$, the sets $E$ and $F$ lie in different components of $X\setminus |\gamma|$.

\begin{lemm}\label{lemma:modulus_bound_convergence}Let $\{X_n\}_{n=1}^\infty$ be a sequence of compact length spaces converging in the Gromov--Hausdorff sense to a compact length surface $X$. Moreover, suppose that $\limsup_{n\to\infty} \mathcal H^2(X_n)<\infty$. Then for each $\delta>0$ and for any sequence of pairs of disjoint continua $E_n,F_n \subset X_n$ with $\min\{\diam(E_n),\diam(F_n)\}\geq \delta$ we have
\begin{align*}
    \limsup_{n\to\infty}\Mod \Gamma^*(E_n,F_n;X_n) <\infty.
\end{align*}
\end{lemm}

\begin{proof}
We claim that that there exists $\eta>0$, depending on $\delta$ but not on $n$, such that if $E_n,F_n\subset X_n$ is a pair of disjoint continua satisfying $\min \{\diam(E_n), \diam(F_n)\}\geq \delta$, then $\ell(\gamma)\geq \eta$ for every $\gamma\in \Gamma^*(E_n,F_n;X_n)$. Assuming that this is the case, we see that the function $\rho=\eta^{-1}$ is admissible for $\Gamma^*(E_n,F_n;X_n)$, so 
\begin{align*}
    \Mod \Gamma^*(E_n,F_n;X_n) \leq \eta^{-2} \mathcal H^2(X_n)
\end{align*}
for each $n\in \N$. Passing to the limit gives the desired conclusion.

In order to prove the claim, we argue by contradiction. Let $f_n\colon X_n\to X$ be a sequence of $\varepsilon_n$-isometries, where $\varepsilon_n\to 0$. Suppose that there exist sequences of disjoint continua $E_n,F_n\subset X_n$ with $\min\{\diam(E_n),\diam(F_n)\}\geq \delta$ and a sequence of paths $\gamma_n\in \Gamma^*(E_n,F_n;X_n)$ with $\ell(\gamma_n)\to 0$ as $n\to\infty$.  By Proposition \ref{prop:gh} \ref{prop:arzela_ascoli}, after reparametrizing $\gamma_n$, there exists a subsequence of $f_n\circ \gamma_n$ that converges uniformly to a constant path in $X$, i.e., to a point $x_0\in X$. After passing to a further subsequence, by Proposition \ref{prop:gh} \ref{prop:connected}  the sets $f_n(E_n)$ and $f_n(F_n)$ converge in the Hausdorff sense to continua $E$ and $F$, respectively, with $\min\{\diam(E),\diam(F)\}\geq \delta$. 

Since $X$ is a surface, $X\setminus \{x_0\}$ is path connected. Thus, there exists a path $\eta\colon [0,1]\to X\setminus \{x_0\}$ with $\eta(0)\in E$ and $\eta(1)\in F$. By the Hausdorff convergence of $f_n(E_n)$ and $f_n(F_n)$ to $E$ and $F$, respectively, there exist points $a_n\in E_n$ and $b_n\in F_n$ such that $f_n(a_n)$ converges to $\eta(0)$ and $f_n(b_n)$ converges to $\eta(1)$. By Proposition \ref{prop:gh} \ref{prop:lift}, there exist paths $\eta_n\colon [0,1]\to X_n$ such that $\eta_n(0)=a_n\in E_n$, $\eta_n(1)=b_n\in F_n$, and $f_n\circ \eta_n$ converges uniformly to $\eta$. 

Since $\gamma_n$ separates $E_n$ from $F_n$ and $\eta_n$ connects $E_n$ and $F_n$, the paths $\gamma_n$ and $\eta_n$ intersect each other for each $n\in \N$. The uniform convergence of $f_n\circ \gamma_n$ and $f_n\circ \eta_n$ to $x_0$ and $\eta$, respectively, implies that $\eta$ intersects the point $x_0$. This is a contradiction. 
\end{proof}

\bigskip

\subsection{Metric Sobolev spaces}

Let $h\colon X\to Y$ be a mapping between metric spaces. We say that a Borel function $g\colon X\to [0,\infty]$ is an \textit{upper gradient} of $h$ if 
\begin{align}\label{ineq:upper_gradient}
    d_Y(h(a),h(b)) \leq \int_{\gamma} g \, ds
\end{align}
for all $a,b\in X$ and every locally rectifiable path $\gamma$ in $X$ joining $a$ and $b$. This is called the \textit{upper gradient inequality}. If, instead the above inequality holds for all curves $\gamma$ outside a curve family of $2$-modulus zero, then we say that $g$ is a \textit{weak upper gradient} of $h$. In this case, there exists a curve family $\Gamma_0$ with $\Mod \Gamma_0=0$ such that all paths outside $\Gamma_0$ and all subpaths of such paths satisfy the upper gradient inequality.

{We equip the space $X$ with the Hausdorff $2$-measure $\mathcal{H}^2$.} Let $L^p(X)$ denote the space of $p$-integrable Borel functions from $X$ to the extended real line $\widehat{\mathbb{R}}$, where two functions are identified if they agree $\mathcal{H}^2$-almost everywhere. The Sobolev space $N^{1,p}(X,Y)$ is defined as the space of Borel mappings $h \colon X \to Y$ with a weak upper gradient $g$ in $L^p(X)$ such that the function $x \mapsto d_Y(y,h(x))$ is in $L^p(X)$ for some $y \in Y$, again where two functions are identified if they agree almost everywhere. The spaces $L_{\loc}^p(X)$ and $N_{\loc}^{1,p}(X, Y)$ are defined in the obvious manner. See the monograph \cite{HKST:15} for background on metric Sobolev spaces.

We now restrict to mappings $h\colon X \to Y$, where $X$ and $Y$ are metric surfaces with locally finite Hausdorff $2$-measure. We use the facts that topological surfaces are second countable, separable, and they admit an exhaustion by precompact open sets. Thus, the Hausdorff $2$-measure is $\sigma$-finite if it is locally finite.

\begin{lemm}\label{lemma:weak_upper_gradient_path_integral}
Let $X,Y$ be metric surfaces with locally finite Hausdorff $2$-measure, $h\colon X\to Y$ be a mapping in $N^{1,2}_{\loc}(X,Y)$, and $g\in L^2_{\loc}(X)$ be a weak upper gradient of $h$. 
\begin{enumerate}[label=\normalfont(\roman*)]
    \item\label{lemma:weak_upper_gradient_line_integral}There exists an exceptional family of curves $\Gamma_0$ with $\Mod\Gamma_0=0$ such that for any Borel function $\rho \colon Y\to [0,\infty]$  and for all locally rectifiable curves $\gamma\notin \Gamma_0$ we have
\begin{align*}
    \int_{h\circ \gamma} \rho\, ds \leq \int_{\gamma} (\rho\circ h)g \, ds. 
\end{align*}
    \item\label{lemma:weak_upper_gradient_measure_ineq} Suppose, in addition, that $h$ is continuous and there exists $K>0$ such that for every Borel set $E\subset Y$ we have
    $$\int_{h^{-1}(E)} g^2 \, d\mathcal H^2 \leq K\mathcal H^2(E).$$
    Then, for every curve family $\Gamma$ in $X$ we have
    $$\Mod \Gamma \leq K \Mod h(\Gamma).$$
\end{enumerate}
\end{lemm}

Here, if $h\colon X \to Y$ is continuous and $\Gamma$ is a curve family in $X$, then $h(\Gamma)$ denotes the curve family $\{h\circ \gamma :\gamma\in \Gamma\}$.

\begin{proof}
Part \ref{lemma:weak_upper_gradient_line_integral} follows from \cite[Proposition 6.3.3, p.~157]{HKST:15}, which assumes that $\int_{\gamma}g\, ds<\infty$ and the upper gradient inequality \eqref{ineq:upper_gradient} holds for all subpaths of $\gamma$. Since $g$ lies in $L^2_{\loc}(X)$ and $X$ can be written as a countable union of open sets of finite Hausdorff $2$-measure, we conclude that there exists a curve family $\Gamma_0$ with modulus zero such that the required conditions hold for paths $\gamma\notin \Gamma_0$.

For \ref{lemma:weak_upper_gradient_measure_ineq}, note that the continuity assumption implies that for any Borel function $\rho \colon Y\to [0,\infty]$, the function $\rho\circ h$ is also Borel measurable. Moreover, by monotone convergence we have
$$\int_{X} (\rho\circ h) g^2 \, d\mathcal H^2 \leq K \int_{Y} \rho \, d\mathcal H^2.$$
Let $\rho$ be an admissible function for $h(\Gamma)$. By \ref{lemma:weak_upper_gradient_line_integral}, for $\gamma\in \Gamma\setminus \Gamma_0$ we have
$$1\leq \int_{h\circ \gamma}\rho \, ds \leq \int_{\gamma}(\rho\circ h) g \, ds. $$
Thus, $(\rho\circ h) g$ is a Borel function that is admissible for $\Gamma\setminus \Gamma_0$. It follows that
$$\Mod \Gamma = \Mod(\Gamma\setminus \Gamma_0) \leq \int_X (\rho\circ h)^2 g^2 \, d\mathcal H^2 \leq K \int_Y \rho^2\, d\mathcal H^2.$$
Infimizing over $\rho$ gives the conclusion.
\end{proof}

It is a non-trivial result of  Williams \cite[Theorem 1.1]{Wil:12} that the converse of  Lemma \ref{lemma:weak_upper_gradient_path_integral} \ref{lemma:weak_upper_gradient_measure_ineq} is also true. This result will be used in the proof of Theorem \ref{theorem:definitions_qc_1}.

\begin{thm}[Definitions of quasiconformality]\label{theorem:qc_definitions_williams}
Let $X,Y$ be metric surfaces with locally finite Hausdorff $2$-measure and let $h\colon X\to Y$ be a continuous mapping. The following are equivalent. 
\begin{enumerate}[label=\normalfont(\roman*)]
    \item \label{item:qc_equivalence_i_williams} $h\in N^{1,2}_{\loc}(X,Y)$ and there exists a weak upper gradient $g$ of $h$ such that for every Borel set $E\subset X$ we have $$\int_{h^{-1}(E)}g^2 \, d\mathcal H^2 \leq K \mathcal H^2(E).$$
    \item \label{item:qc_equivalence_ii_williams} For every curve family $\Gamma$ in $X$ we have
    $$\Mod \Gamma \leq K \Mod h(\Gamma).$$
\end{enumerate}
\end{thm}

In fact, the argument of Williams \cite[Proof of Theorem 1.1]{Wil:12} is more general and relies on the local finiteness of the measures and the separability of the spaces. We note that the referenced result is stated for homeomorphisms, but the proof applies identically to the case of continuous mappings.

\bigskip

\subsection{Polyhedral surfaces} \label{sec:polyhedral_surfaces}

A \textit{$1$-dimensional polyhedral space} is a locally finite connected graph, considered as a metric space by assigning a length to each edge and taking the corresponding length metric. Next, we define a {$2$-dimensional polyhedral space} in the following manner. Let $\Gamma$ be a $1$-dimensional polyhedral space and $\mathcal{P}$ a collection of planar polygonal domains homeomorphic to a closed disk. Each $P \in \mathcal{P}$ is equipped with the length metric induced by the Euclidean metric on $\mathbb{C}$, {which we denote by $d_P$}. The boundary $\partial P$ is subdivided into finitely many non-overlapping line segments called \textit{edges}. For each $P \in \mathcal{P}$, let $\psi_P \colon \partial P \to \Gamma$ be an injective mapping such that each edge of $\partial P$ is mapped by arc length onto an edge of $\Gamma$. {Assume that each point in $\Gamma$ is in the image of at least one and finitely many maps $\psi_P$.} We obtain a metric space $S$ by gluing the disjoint union of the sets in $\mathcal{P}$ with $\Gamma$ along the maps $\psi_P$. More precisely, we define $\sim$ to be the equivalence relation on $(\bigsqcup \mathcal{P})\sqcup \Gamma$ generated by declaring $x \sim y$ if $x \in \partial P$ for some $P\in \mathcal{P}$, $y \in \Gamma$,  and $\psi_P(x) = y$. Take $S = (\bigsqcup \mathcal{P})\sqcup \Gamma/\sim$. Define the metric $d$ on $S$ by
\[d(x,y) = \inf \sum_{k=1}^n d_{P_k}(x_k,y_k), \]
the infimum taken over all chains of points $x_1,y_1, \ldots, x_n,y_n$ such that $x_k,y_k$ belong to the same polygonal domain $P_k$ for all $k \in \{1,\ldots,n\}$ and $y_k \sim x_{k+1}$ for all $k \in \{1,\ldots, n-1\}$, and $x = x_1$ and $y = y_n$. It is straightforward to verify that $d$ is indeed a metric. We say that $S$ equipped with the metric $d$ is a \textit{$2$-dimensional polyhedral space} and the metric $d$ is called the \textit{polyhedral metric} on $S$. We identify the graph $\Gamma$ with the subset $\bigsqcup \partial P  / \sim$ of $S$ in the natural way. {Observe that each polygon $P \in \mathcal{P}$ is locally isometric to its image in $S$ at every non-vertex point.}  Each polygon $P \in \mathcal{P}$ is called a \textit{face} of $S$, while the vertices and edges of {each $P$} are called the \textit{vertices} and \textit{edges}, respectively, of $S$. 

A \textit{polyhedral surface} is a $2$-dimensional polyhedral space homeomorphic to a $2$-manifold with boundary. Each point in a polyhedral surface has a neighborhood isometric to a ball in the Euclidean cone over a circle or closed interval; this property can also be taken as a definition of polyhedral surface \cite{LebPet:15}. In particular, a polyhedral surface is locally isometric to a subset of the closed half-plane at each non-vertex point. See \cite[Section I.5.19]{BH:99}, \cite[Section 3.1--3.2]{BBI:01} and \cite{Tro:86} for an overview of polyhedral spaces and the operation of gluing. 

\subsubsection{Complex structure}\label{sec:complex} 
It is known that each orientable polyhedral surface $X$ has a complex structure that agrees with the complex structure of the polygons that constitute it \cite[II.4, pp.~66--67]{Cou:77}. More precisely, suppose that $X=\bigsqcup P_i /\sim$, where each $ P_i$ is a closed planar polygonal domain and the boundaries of the polygons are identified according to some equivalence relation $\sim$ as above. By the orientability of $X$, we may assign an orientation to each $\partial P_i$ so that the orientations of adjacent polygons are compatible; that is, orientations of neighboring edges from different polygons are pointing in opposite directions. Thus, by replacing each $P_i\subset \C$ with a reflected copy if necessary, we may assume that the orientation of $P_i$ as a subset of $X$ is the positive one when $P_i$ is considered as a subset of the plane. 

We let $\varphi_i$ be a homeomorphism, acting as the identity map, that identifies $P_i$ as a subset of $X$ with itself as a subset of $\C$. Then $\varphi_i$ serves as a local chart at each interior point of $P_i$ and at each boundary point of $X$ that is contained in $\partial P_i$ and is not a vertex. If the polygons $P_i,P_j\subset X$ share an open boundary segment $J$, then there exists a local chart $\varphi_J$ in a neighborhood $U\subset X$ of $J$ such that, up to orientation-preserving isometries of the plane, $\varphi_J$ agrees with $\varphi_i$ and $\varphi_j$ in $U\cap \Int(P_i)\subset X$ and $U\cap \Int(P_j)\subset X$, respectively. In particular, the transition from $\varphi_i(\Int(P_i))$ and $\varphi_j(\Int(P_j))$ to $\varphi_J(U)$ is conformal. Finally, at each vertex  $v\in X$ consider a small $r>0$ such that each face $P_i$ that has $v$ as a vertex contains a circular sector $S_i$ of radius $r$, centered at $v$, and whose two sides are contained in two edges of $P_i$. Let $\theta=\theta(v)$ be the sum of the angles of these sectors. To each of these sectors, we apply a map of the form $z\mapsto z^{\alpha}$, where $\alpha= \theta/2\pi$ if $v$ is an interior point of $X$ and $\alpha=\theta/\pi$ if $v$ is a boundary point of $X$; more precisely, consider the maps $(\varphi_i -\varphi_i(v))^{\alpha}$ mapping the sector $S_i\subset X$ onto a sector $S_i'$ centered at $0$ in the plane. Then the sectors $S_i'$ may be rotated and fitted together to form a disk of radius $r^{\alpha}$ if $v\in \Int(X)$ and a semidisk if $v\in \partial X$. In this way we can also define conformal coordinates at the vertices.  In summary, every orientable polyhedral surface $X$ is a Riemann surface with the described natural conformal structure. Thus we call $X$ a polyhedral Riemann surface.

A homeomorphism $h\colon X\to Y$ between Riemann surfaces is conformal if it is complex differentiable in local coordinates. Specifically, at each $x\in  X$ we require that if $\varphi$ is a conformal chart from a neighborhood of $x$ in $X$ into $\C$ and $\psi$ is a conformal chart from a neighborhood of $h(x)$ in $Y$ into $\C$, then $\psi\circ h\circ \varphi^{-1}$  is a conformal map defined on a neighborhood of $\varphi(x)$ in $\C$. If $x\in \partial X$, this definition entails the requirement that $\psi\circ h\circ \varphi^{-1}$ has a conformal extension in a neighborhood of $\varphi(x)$.

Let $Y$ be a polyhedral Riemann surface. If $Y$ homeomorphic to a topological 2-sphere, then, by the uniformization theorem \cite[Theorem 15.12, p.~242]{Mar:19}, there exists a conformal homeomorphism $h$ from the Riemann sphere $\widehat{\C}$ to $Y$. If $Y$ is a closed topological disk, then we obtain a conformal homeomorphism from $Y$ to $\overline{\mathbb{D}}$ in the following way. Glue $Y$ to an isometric copy of itself along the boundary to obtain a polyhedral sphere $\widetilde{Y}$. By the uniformization theorem, there is a conformal homeomorphism $h \colon \widetilde{Y} \to \widehat{\mathbb{C}}$. Define the involution $\varphi \colon \widetilde{Y} \to \widetilde {Y}$ by mapping each point in $Y$ to the same point in its isometric copy. Then $g = h \circ \varphi \circ h^{-1}$ is an anti-conformal homeomorphism of $\widehat{\mathbb{C}}$ and thus is an anti-M\"obius transformation with fixed set $h(\partial Y)$. This implies that $h(\partial Y)$ is a circle. By normalizing $h$, we ensure that $h(\partial Y)$ is the equator of $\widehat{\mathbb{C}}$. Thus $h$ restricts to a conformal homeomorphism from $Y$ to the upper hemisphere.  We summarize these facts below.

\begin{thm}[Uniformization theorem]\label{theorem:uniformization}
Let $\Omega=\widehat\C$ or $\Omega=\overline{\mathbb D}$. If $Y$ is a polyhedral Riemann surface homeomorphic to $\Omega$, then there exists a conformal homeomorphism from $\Omega$ onto $Y$.
\end{thm}

A polyhedral surface with its polyhedral metric becomes a surface of locally finite Hausdorff $2$-measure. We endow the Riemann sphere $\widehat \C$ with the spherical metric $\sigma$, which is given by {the length element} {$2(1+|z|^2)^{-1}\, ds$} in planar coordinates {$z = x+iy$} through stereographic projection. We also consider the spherical measure given by the density ${d\sigma =} 4(1+|z|^2)^{-2}dx\,dy$, which agrees with the Hausdorff $2$-measure on $\widehat\C$ arising from the spherical metric. Similarly, we endow the closed unit disk $\overline{\mathbb D}$ with the planar Euclidean metric {and the Lebesgue measure, which agrees with the Hausdorff $2$-measure}. 

\begin{lemm}\label{lemma:polyhedral_conformal}
Let $\Omega=\widehat\C$ or $\Omega=\overline{\mathbb D}$. Suppose that $Y$ is a polyhedral Riemann surface homeomorphic to $\Omega$ and $h\colon \Omega\to Y$ is a conformal homeomorphism. {There exist Borel measurable functions $|Dh|\colon \Omega\to [0,\infty)$ and $|Dh^{-1}| \colon Y\to [0,\infty)$ such that the following hold.} 
\begin{enumerate}[label=\normalfont(\roman*)]
    \item\label{lemma:polyhedral_conformal_upper}  {$|Dh|$ and $|Dh^{-1}|$ are} upper gradients of $h$ and $h^{-1}$, respectively.
    \item\label{lemma:polyhedral_conformal_change} For all Borel sets $E\subset \Omega$ and $F\subset Y$ we have
\begin{align*}
    \int_E |Dh|^2 \, d\mathcal H^2 = \mathcal H^2(h(E)) \quad \textrm{and}\quad \int_F |Dh^{-1}|^2 \, d\mathcal H^2 = \mathcal H^2(h^{-1}(F)).
\end{align*}
    \item\label{corollary:polyhedral_conformal} For every curve family $\Gamma$ in $\Omega$ we have
\begin{align*}
    \Mod \Gamma=\Mod h(\Gamma).
\end{align*}
\end{enumerate}
\end{lemm}

\begin{proof}
We show the existence of the upper gradient $|Dh|$ of $h$ that satisfies the change of coordinates formula in \ref{lemma:polyhedral_conformal_change}. Then, from Lemma \ref{lemma:weak_upper_gradient_path_integral}, it follows that $\Mod \Gamma \leq  \Mod h(\Gamma)$ for all curve families $\Gamma$ in $\Omega$. The claims for $h^{-1}$ are proved similarly. 

We write $Y=\bigsqcup P_i/\sim$, where $P_i$  are polygonal domains in the plane, and denote by $\varphi_i$   the complex chart identifying $P_i\subset Y$   with itself as as subset of $\C$; see the discussion in the beginning of Section \ref{sec:complex}. Let $V$ denote the set of vertices of $X$ and note that $V$ is finite. 

On $h^{-1}(V)$ we define $|Dh|=0$. On $\Omega\setminus h^{-1}(V)$ we define $|Dh|$ as follows. Let $x\in \Omega \setminus h^{-1}(V)$ and consider a polygon $P_i$ with $h(x)\in P_i$. We define $|Dh|$ to be the absolute value of the derivative of $\varphi_i\circ h$ as a holomorphic map from a subset of $\Omega$ to the planar polygon $P_i$. (If $\Omega=\widehat\C$, using the coordinates of the stereographic projection gives
$|Dh|(z)= 2^{-1}(1+|z|^2)|(\varphi_i\circ h)'(z)|$, although we do not need this formula.) If $h(x)$ does not lie on any polygon $P_j$ for $j\neq i$ then $|Dh|(x)$ is clearly well-defined. Suppose that $h(x)$ lies in the interior of a common edge $J$ of $P_i$ and $P_j$. There exists a local chart $\varphi_J$ in a neighborhood $U\subset Y$ of $J$ such that, up to isometries of the plane, $\varphi_J$ agrees with $\varphi_i$ and $\varphi_j$ in $U\cap \Int(P_i)\subset Y$ and $U\cap \Int(P_j)\subset Y$, respectively. This shows that the absolute values of the derivatives of $\varphi_i\circ h$ and $\varphi_j\circ h$ agree on $h^{-1}(J)$, so $|Dh|(x)$ is also well-defined in this case.

With this definition of $|Dh|$, we claim that if $\gamma$ is a locally rectifiable path in $\Omega$ connecting points $a$ and $b$, then 
\begin{align*}
    d_Y(h(a),h(b))\leq \ell(h\circ \gamma)= \int_{\gamma} |Dh| \, ds, 
\end{align*}
so \ref{lemma:polyhedral_conformal_upper} is true. We only have to justify the equality. The statement holds for paths avoiding the finite set $h^{-1}(V)$ because $|Dh|$ is the absolute value of the derivative of appropriate conformal maps and the metric of $Y$ is locally isometric to the Euclidean metrics of the polygons away from the vertices. The general statement is proved by partitioning a path $\gamma\colon [0,1]\to \Omega$ into possibly infinitely many subpaths $\gamma_i\colon I_i\to \Omega$, $i\in \N$, where $I_i$, $i\in \N$, are the components of $[0,1]\setminus \gamma^{-1}(h^{-1}(V))$. Each path $\gamma_i$ satisfies the claimed equality. Since $V$ is a finite set (in fact, it suffices that $\mathcal H^1(V)=0$; see Proposition \ref{prop:cantor} \ref{prop:cantor:curve}), one can show that 
\begin{align*}
    \ell(h\circ \gamma)= \sum_{i\in \N} \ell(h\circ \gamma_i)
\end{align*}
This completes the proof of the claim. 

The change of coordinates formula in \ref{lemma:polyhedral_conformal_change} is true for Borel sets $E\subset \Omega \setminus h^{-1}(V)$, since $|Dh|^2$ is the Jacobian of appropriate conformal maps and the metric of $Y$ is locally isometric to the Euclidean metrics of the polygons away from the vertices. On the other hand,  the vertices have measure zero both in $\Omega$ and in $Y$.  Thus, the change of coordinates holds for all Borel sets $E\subset \Omega$.
\end{proof}

\bigskip

\section{Bi-Lipschitz embedding triangles into the plane} \label{sec:bl_embedding_triangles}

In this section, we prove \Cref{prop:BL_embedding}, stating that every metric triangle can be bi-Lipschitz embedded into the plane with a uniform bi-Lipschitz constant. Recall that a metric triangle is a metric space consisting of three closed arcs, called {edges}, each isometric to an interval that connect pairwise a set of three points, called {vertices}. More precisely, we can define a metric triangle as the quotient metric space induced by equipping $\mathbb{S}^1$ with a pseudometric such that $\mathbb{S}^1$ is the union of three non-overlapping closed arcs each isometric to an interval. {Here, two sets are \textit{non-overlapping} if their interiors are disjoint.} This definition allows the possibility that the edges intersect in interior points. We say that a metric triangle is \textit{simple} if it is homeomorphic to $\mathbb S^1$. A \textit{tripod} is a length metric space consisting of three closed arcs glued at a common endpoint but otherwise disjoint. Note that a tripod is also a metric triangle {with vertices the non-glued endpoints of the original closed arcs}.

For any triple of points $p,q,r$ in a metric space $(X,d)$, the \emph{Gromov product} $(p\cdot q)_r$ is defined by
\[(p \cdot q)_r = \frac{1}{2}(d(p,r) + d(q,r) - d(p,q)). \] 
To such a triple $p,q,r \in X$, we can associate a tripod $\bar{\Delta}$ with outer vertices $\bar{p}, \bar{q}, \bar{r}$ and central vertex $\bar{o}$, where $\ell([\bar{o},\bar{p}]) = (q\cdot r)_p$, $\ell([\bar{o},\bar{q}]) = (r\cdot p)_q$, and $\ell([\bar{o},\bar{r}]) = (p\cdot q)_r$. Observe that 
\begin{align}\label{gromov}
    d(p,q) = (q \cdot r)_p + (r \cdot p)_q = \ell([\bar{p},\bar{o}]) + \ell([\bar{o},\bar{q}]) = \ell([\bar{p},\bar{q}]),
\end{align}
and similarly for $d(p,r)$ and $d(r,q)$. For more background on the Gromov product, see \cite[Chapter III.H.1]{BH:99}. {We denote the metric on $\bar{\Delta}$ by $D$.}

If $p,q,r$ are the vertices of a metric triangle $\Delta$, then there is a natural projection $\Phi \colon \Delta \to \bar{\Delta}$ such that $\Phi(p) = \bar{p}$, $\Phi(q) = \bar{q}$, $\Phi(r) = \bar{r}$, and $\Phi$ is an isometry on each edge of $\Delta$. {For a point $x \in \Delta$, we write $\bar{x}$ to denote $\Phi(x)$.}

\begin{lemm}\label{lemma:lipschitz}
The natural projection $\Phi\colon \Delta\to \bar \Delta$ is $1$-Lipschitz. More specifically, we have
\begin{equation} \label{equ:tripod_projection}
    D(\bar{x},\bar{y}) \leq d(x,y)
\end{equation}
for all $x,y \in \Delta$, with equality whenever $x,y$ lie on the same edge of $\Delta$.
\end{lemm}
\begin{proof}
By definition, we have $D(\bar x, \bar y)=d(x,y)$ whenever $x,y$ lie on the same edge of $\Delta$. Without loss of generality, we assume that $x$ lies on the edge $[p,q]$ and $y$ lies on $[p,r]$ and that $d(y,p)\geq d(x,p)$. We consider two cases.

Suppose first that $d(x,p)\leq \ell([ \bar o,\bar p])$. Then there exists a point $x'\in [p,r]$ such that $d(x',p)=d(x,p)$ and $\bar x' =\bar x$. Using the fact that $[p,q]$ and $[p,r]$ are geodesics, we have
\begin{align*}
    D(\bar x,\bar y)&= D(\bar x',\bar y) =d(x',y)= d(y,p)-d(x',p)\\
    &=d(y,p)-d(x,p)\leq d(x,y).
\end{align*}

Next, suppose that $d(x,p)> \ell([ \bar o,\bar p])$. Then $d(x,q)< \ell([ \bar o,\bar q])$ by \eqref{gromov}. In this case, there exists a point $x' \in [q,r]$ such that $d(x',q)=d(x,q)$ and $\bar x'=\bar x$. Moreover, since $d(y,p)\geq d(x,p)>\ell([ \bar o,\bar p])$, we have $d(y,r)< \ell([ \bar o,\bar r])$. Hence, there exists a point $y'\in [q,r]$ such that $d(y',r)=d(y,r)$ and $\bar y'=\bar y$. We have
\begin{align*}
    D(\bar x,\bar y)&= D( \bar x', \bar y')=d(x',y') =d(q,r)-d(y',r)-d(x',q)\\
    &=d(q,r)-d(y,q)-d(x,r) \leq d(x,y).
\end{align*}
This completes the proof.
\end{proof}

We consider such a tripod $\bar{\Delta}$ as being embedded in $\mathbb{C}$, with the central vertex $\bar{o}$ at the origin and $p = (q\cdot r)_p$, $q = (r \cdot p)_q e^{2\pi i/3}$, and $r = (p \cdot q)_r e^{4\pi i/3}$. {Here and throughout this section, we use complex notation for points in $\mathbb{C}$.} We call such $\bar{\Delta}$ the \textit{canonical tripod} determined by $\Delta$. Our strategy for proving \Cref{prop:BL_embedding} is to project the metric triangle $\Delta$ onto the corresponding tripod $\bar{\Delta}$ and then add a transverse component whose magnitude is the distance from a given point to the union of the other two sides. A typical embedding is illustrated in \Cref{fig:tripod}.

\bigskip

\subsection{Tripodal metric on $\mathbb{C}$} \label{sec:tripodal_metric}
It is convenient to introduce a new metric on $\mathbb{C}$ that is compatible with embedded tripods. Let $u_1 = 1$, $u_2 = e^{2 \pi i/3}$, $u_3 = e^{4\pi i/3}$, $v_1 = e^{\pi i /3}$, $v_2 = -1$, and $v_3 = e^{5\pi i/3}$. For each $j \in \{1,2,3\}$, let $Z_j = \{tu_j: t \geq 0\}$, and let $Z = \bigcup_{j=1}^3 Z_j$. Then $\mathbb{C} \setminus Z$ consists of three components $U_1, U_2, U_3$, indexed so that $v_j \in U_j$ for each $j \in \{1,2,3\}$. Observe that each point $x \in \overline{U_j}$ can be written uniquely as $x = \bar{x} + t_xv_j$ for some $\bar{x} \in \partial U_j$ and $t_x \geq 0$. We employ this notation for a given point $x \in \mathbb{C}$.

We define the metric $D$ on $\mathbb{C}$ in the following way. First, $D|_{Z \times Z}$ is the intrinsic metric on $Z$. Next, for $x,y \in \overline{U_1}$, let $D(x,y) = |t_x - t_y| + D(\bar{x}, \bar{y})$. Define $D$ similarly on $\overline{U_2} \times \overline{U_2}$ and $\overline{U_3} \times \overline{U_3}$. Finally, for $x\in \br{U_j}$ and $y\in \br{U_i}$, where $i\neq j$, define 
\begin{equation} \label{equ:tripodal_metric}
  D(x,y)= t_x+t_y+D(\bar x, \bar y).  
\end{equation}
We note that $D(x,y)$ is the Euclidean length of a certain polygonal path joining $x$ and $y$. Observe that each set $\br{U_i}$ is convex with respect to $D$, and in particular that $D$ is a length metric on $\mathbb{C}$. 

We observe that $D$ is bi-Lipschitz equivalent to the Euclidean metric. In fact, a straightforward argument shows that
\begin{equation} \label{equ:bl_comparison}
    |x-y| \leq D(x,y) \leq 2|x-y| 
\end{equation}
for all $x,y \in \mathbb{C}$. The right inequality is sharp, as seen by taking $x=1$ and $y = e^{i\pi/3}$. 

\bigskip

\subsection{Proof of \Cref{prop:BL_embedding}}

We first restate \Cref{prop:BL_embedding} in a more precise form. For each $x \in \Delta$, let $I(x)$ denote an edge of $\Delta$ containing $x$ and $\widehat{I}(x)$ the union of the other two edges of $\Delta$. For a point $x \in \Delta$, denote by $\bar{x}$ the natural projection of $x$ in the canonical tripod $\bar{\Delta}$. Recall from the previous section the notation $u_j = e^{(2j-2)\pi i/3}$ and $v_j= e^{(2j-1)\pi i/3}$ for $j\in \{1,2,3\}$.
\begin{prop} \label{prop:BL_embedding2}
Let $\Delta$ be a metric triangle with vertices $p,q,r$ and edges $I_1 = [p,q]$, $I_2 = [q,r]$, $I_3 = [r,p]$. Let $\bar{\Delta}$ denote the canonical tripod determined by $\Delta$. Define the mapping $F\colon \Delta \to \mathbb{C}$ by 
\[F(x) = \bar{x} + \dist (x, \widehat{I}(x))v_j \hspace{.2in} \text{if } x \in I_j,\, j=1,2,3.\]
Then $F$ is $L$-bi-Lipschitz for $L=4$. 
\end{prop}

\begin{proof}

Recall that $\bar{\Delta}$ is the tripod $[\bar{o},\bar{p}] \cup [\bar{o},\bar{q}] \cup [\bar{o},\bar{r}]$, where $\bar{p} = (q\cdot r)_p u_1$, $\bar{q} = (r \cdot p)_qu_2$, $\bar{r} = (p \cdot q)_ru_3$, and $\bar{o} = 0$. We use $D$ to denote the tripodal metric on $\mathbb{C}$ defined in \Cref{sec:tripodal_metric}, which agrees with the length metric on $\bar{\Delta}$ as a tripod.

Observe that $F$ is well-defined, and in particular that $F(x) = \bar{x}$ in the case that $x \in I_j \cap I_k$ for $j \neq k$. Let $x,y \in \Delta$. By symmetry, it suffices to consider the case that $x \in I_1$.

\begin{figure}
    \centering
    \begin{tikzpicture}[scale=1.4]
        \draw[->] (0,0) to (2,0);
        \draw[->] (0,0) to (-1,1.73);
        \draw[->] (0,0) to (-1,-1.73);
        \draw[blue, line width = 2pt] (0,0) to (1.8,0);
        \draw[blue, line width = 2pt] (0,0) to (-.8,1.39);
        \draw[blue, line width = 2pt] (0,0) to (-.6,-1.039);
        \draw[red,line width = 2pt] (1.8,0) .. controls (1.5,.03) and (1.3,.06) .. (1.,.2) .. controls (.7,.3) and (.4,.85) .. (.1,1.05) .. controls (-.2,1.3) and (-.5,1.35) .. (-.8,1.39) .. controls (-.95,1.1) and (-.95,.9) .. (-.85,.6) .. controls (-.7,.3) and (-.45,0) .. (-.45,-.3) .. controls (-.45,-.6) and (-.55,-.8) .. (-.6,-1.039) .. controls (-.4,-.8) and (-.2,-.55) .. (.1,-.4) .. controls (.3,-.3) and (.6,-.2) .. (.8,-.15) .. controls (1.1,-.1) and (1.3,-.05) .. (1.8,0);
        \filldraw[red] (-.8,1.39) circle (.04);
        \filldraw[red] (-.6,-1.039) circle (.04);
        \filldraw[red] (1.8,0) circle (.04);
        \node[red] at (1.8,-.2) {$\bar{p}$};
        \node[red] at (-1.0,1.4) {$\bar{q}$};
        \node[red] at (-.75,-1) {$\bar{r}$};
        \node[blue] at (.2,.2) {$\bar{\Delta}$};
        \node[red] at (.8,.8) {$F(\Delta)$};
    \end{tikzpicture}
    \caption{}
    \label{fig:tripod}
\end{figure}

If $y \in I_1$ as well, then 
$$D(F(x),F(y))= \abs{\dist(x,\widehat{I}(x)) - \dist(y, \widehat{I}(y))}+D(\br x,\br y). $$
Thus, we have $D(F(x),F(y)) \geq D(\bar{x},\bar{y}) = d(x,y)$ by \eqref{equ:tripod_projection}. Moreover, since \[\abs{\dist(x,\widehat{I}(x)) - \dist(y, \widehat{I}(y))} \leq d(x,y),\] we have $D(F(x),F(y)) \leq 2d(x,y)$. Summarizing, in this case we have
$$d(x,y)\leq D(F(x),F(y))\leq 2d(x,y).$$

Next, we suppose that $y \notin I_1$. The Lipschitz inequality follows immediately, since by \eqref{equ:tripodal_metric} and \eqref{equ:tripod_projection} we have
\[D(F(x),F(y)) = \dist(x, \widehat{I}(x)) + \dist(y, \widehat{I}(y))+D(\bar{x}, \bar{y})  \leq 3d(x,y).\]

For the co-Lipschitz inequality, let $z \in \widehat{I}(x)$ be such that $d(x,z) = \dist(x,\widehat{I}(x))$, and let $w \in \widehat{I}(y)$ be such that $d(y,w) =\dist(y,\widehat{I}(y))$. We split into cases.

\begin{case} \label{case:bl_1}
Suppose the points $y,z$ lie on the same edge. Then, $d(y,z) = D(\bar{y},\bar{z})$ by \eqref{equ:tripod_projection}. Moreover, also applying \eqref{equ:tripod_projection}, we have
\begin{align*}
    d(x,y)&\leq d(x,z)+d(y,z)=d(x,z)+D(\bar y, \bar z)\leq d(x,z)+D(\bar x,\bar y)+D(\bar x, \bar z)\\
    &\leq 2d(x,z)+ D(\bar x,\bar y).
\end{align*}
Therefore, \eqref{equ:tripodal_metric} gives
$$D(F(x),F(y))= D(\bar x,\bar y) +d(x,z)+d(y,w)\geq D(\bar x,\bar y) +d(x,z) \geq \frac{1}{2}d(x,y).$$
\end{case}

\begin{case} \label{case:bl_2}
Suppose the points $x,w$ lie on the same edge. This follows from \Cref{case:bl_1} by reversing the roles of $x$ and $y$.
\end{case}

\begin{case} \label{case:bl_3}
Suppose the points $z,w$ lie on the same edge. Then, by \eqref{equ:tripod_projection}, we have
\begin{align*}
    d(z,w) & = D(\bar{z},\bar{w}) \leq D(\bar{z},\bar{x}) + D(\bar{x},\bar{y}) + D(\bar{y},\bar{w}) \\
    & \leq d(x,z) + D(\bar{x},\bar{y}) + d(y,w) =D(F(x),F(y)).
\end{align*}
Therefore,
\begin{align*}
    D(F(x),F(y))&= \frac{1}{2}( D(F(x),F(y))+ D(\bar{x},\bar{y})+d(x,z)+d(y,w))\\
    &\geq \frac{1}{2}(d(z,w)+d(x,z)+d(y,w))
    \geq \frac{1}{2}d(x,y).
\end{align*}
\end{case}

We conclude that
\[\frac{1}{2} d(x,y) \leq D(F(x),F(y)) \leq 3d(x,y) \]
for all $x,y \in \Delta$. Combining this with \eqref{equ:bl_comparison}, we have
\[\frac{1}{4} d(x,y) \leq |F(x) - F(y)| \leq 3d(x,y)\]
for all $x,y \in \Delta$.
\end{proof}

\begin{rem}\label{remark:polygon}
    By the definition of $F$, it is clear that every line parallel to the vector $v_j$ intersects $F(I_j)$ in at most one point, for each $j\in \{1,2,3\}$. Suppose now that $\Delta$ is a simple metric triangle. If we partition $F(\Delta)$ into arcs $[x_{i-1},x_i]$, $i\in \{1,\dots,n\}$, where $x_0=x_n$, such that the collection $\{x_i\}_{i=1}^n$ contains the vertices of $F(\Delta)$, then the polygonal curve formed by joining $x_{i-1}$ with $x_i$ for each $i$ is a simple closed curve.
\end{rem}

\bigskip

\section{Fillings of simple metric triangles} 
\label{sec:filling}

A \textit{polygonal metric disk} or \textit{polygonal disk} is a metric space homeomorphic to a closed disk whose boundary can be represented as the union of finitely many non-overlapping geodesics, each of which is called an \textit{edge}. The endpoints of the edges are called \textit{vertices}. If a polygonal disk has three edges, we call it a \textit{triangular disk}. Observe that the boundary of a triangular disk is a {simple} metric triangle. A polygonal disk is \textit{planar} if it is a subset of $\mathbb{C}$, {equipped with the length metric induced by the Euclidean metric, and its} boundary consists of finitely many non-overlapping line segments. Thus the boundary of a planar polygonal disk is a polygon in the ordinary sense of the word.
In this section, we construct polyhedral fillings of {simple} metric triangles based on the bi-Lipschitz embedding of the previous section. We first give a preliminary lemma. 

\begin{lemm}\label{lemma:polygon_triang}
For every {planar polygonal disk} $P\subset \mathbb{C}$ and each $\varepsilon>0$ there exists a decomposition $\{P_k\}_{k\in K}$ of $P$ into non-overlapping polygonal disks satisfying the following.
\begin{enumerate}[label=\normalfont(\roman*)]
    \item  $\ell(\partial P_k)<\varepsilon$ for each $k\in K$. \label{item:decomposition_i}
    \item Any points $x,y\in \bigcup_{k\in K} \partial P_k$ can be joined by a path in $\bigcup_{k\in K} \partial P_k$ with length {at most} $\ell(\partial P)$. \label{item:decomposition_ii}
    \item $\displaystyle{\sum_{k\in K} \ell(\partial P_k)^2} \leq 17\mathcal H^2(P)$. \label{item:decomposition_iii}
\end{enumerate}
\end{lemm}
\begin{proof}
Let $\varepsilon>0$. Choose $\varepsilon' \in (0,\varepsilon)$ sufficiently small that any square of side length $\varepsilon'$ intersects at most two edges of $\partial P$. The square grid $\varepsilon' \Z^2$ divides $\C$ into non-overlapping square regions $Q$ of the form $[\varepsilon'j_1,\varepsilon'(j_1+1)]\times[\varepsilon'j_2,\varepsilon'(j_2+1)]$ for some $j_1,j_2 \in \mathbb{Z}$. Note that $Q \cap P$ can be written as the union of at most two polygonal regions for each square region $Q$. Enumerate by $P_k$, $k\in K_1$, the square regions $Q$ that are contained in $P$ and by $P_k$, $k\in K_2$, the polygonal regions arising as the intersection of $P$ with those $Q$ whose interior intersects $\partial P$. We set {$K=K_1\sqcup K_2$}. {By taking $\varepsilon'$ sufficiently small, we also ensure that} $\ell(\partial P_k) < {\varepsilon}$ for each $k\in K$.

Each $x \in  \bigcup_{k\in K} \partial P_k $ belongs to a horizontal line segment in $\bigcup_{k\in K} \partial P_k$ or a vertical line segment in $\bigcup_{k\in K} \partial P_k$. In the first case, let $L_x$ denote the maximal horizontal line segment contained in $P$ passing through $x$. Otherwise, let $L_x$ denote the maximal vertical line segment contained in $P$ passing through $x$.  Let $p_x$ denote a point in $L_x \cap \partial P$ nearest to $x$; then the line segment $A_x \subset L_x$ from $x$ to $p_x$ has length at most $\ell(\partial P)/4$. Given two points $x,y \in \bigcup_{k\in K} \partial P_k$, we can join $p_x$ to $p_y$ by a subarc $C_{xy} \subset \partial P$ of length at most $\ell(\partial P)/2$. Joining $A_x$, $C_{xy}$, and $A_y$ gives a path in $\bigcup_{k\in K} \partial P_k$ with length at most $\ell(\partial P)$. 

For all $k \in K_1$, $P_k$ is a square region and we have $\ell(\partial P_k)^2=16\mathcal H^2(P_k)$. Thus,
\begin{align*}
    \sum_{k\in K_1} \ell(\partial P_k)^2\leq 16 \mathcal H^2(P).
\end{align*}
{For each $k \in K_2$, let $Q_k$ denote the square region above used to define $P_k$, and observe that the correspondence $P_k\mapsto Q_k$ is at most two-to-one. Note that $Q_k$ has diameter less than $2\varepsilon'$ and that $\ell(\partial P_k) \leq  {2} \ell(\partial Q_k)$.} Thus, 
\begin{align*}
    \sum_{k\in K_2} \ell(\partial P_k)^2 & \leq 4 \sum_{k \in K_2} \ell(\partial Q_k)^2 = {64} \sum_{k\in K_2} \mathcal H^2(Q_k) \\ & \leq {64\cdot 2\cdot \mathcal H^2 (N_{2\varepsilon'}(\partial P))} \leq {64\cdot 2\cdot  2(2\varepsilon')} \ell(\partial P),
\end{align*}
where the last inequality follows from \cite[Theorem 10--41, p.~285]{Apo:57}. Therefore
\begin{align*}
    \sum_{k\in K} \ell(\partial P_k)^2 \leq 17\mathcal H^2(P)
\end{align*}
upon choosing $\varepsilon'$ to be sufficiently small.
\end{proof}

We continue with the main result of this section, giving a polyhedral filling of an arbitrary {simple} metric triangle with controlled Hausdorff $2$-measure.

\begin{thm} \label{thm:triangular_surface}
Let $(T,d)$ be a triangular metric disk with edges $\alpha_j$, $j\in \{1,2,3\}$. There exists a polyhedral surface $(S,d_S)$ {that is a triangular metric disk with edges} $\beta_j$, $j\in \{1,2,3\}$, and a homeomorphism  {$\varphi\colon S \to T$} such that the following hold for an absolute constant $L>0$ independent of $T$.
\begin{enumerate}[label=\normalfont(\arabic*)]
    \item $\diam_{d_S}(S)\leq L \diam_d(T)$. \label{item:filling_1}
    \item $\mathcal H_{d_S}^2(S)\leq L \mathcal  H_d^2(T)$. \label{item:filling_2}
    \item $\varphi|_{|\beta_j|}$ maps $|\beta_j|$ isometrically onto $|\alpha_j|$ for each $j \in \{1,2,3\}$. In particular, $\varphi|_{\partial S}$ is length-preserving. \label{item:triangular_disk_3} 
    \label{item:filling_3}
    \item For all $x,y \in \partial S$, $d_S(x,y) \geq d(\varphi(x),\varphi(y))$. \label{item:filling_4}
\end{enumerate}
\end{thm}
\begin{proof}
Let $\Delta = \partial T$. Thus $\Delta$ is a metric triangle equipped with the metric $d$. Let $F\colon \Delta \to \mathbb{C}$ be the $4$-bi-Lipschitz embedding in \Cref{prop:BL_embedding2}, and let ${\Omega}$ be the closed region in $\mathbb{C}$ bounded by $F(\Delta)$. Moreover, let $\beta_j = F \circ \alpha_j$ for each $j \in \{1,2,3\}$. {We use the embedded curve $F(\Delta)$ to construct a polyhedral surface $S$ with the desired properties.} We note that as soon as $S$ is homeomorphic to a closed disk, \ref{item:filling_3} and \ref{item:filling_4} imply immediately that $S$ is a triangular metric disk.

Equip $F(\Delta)$ with the pushforward metric of $d$ under $F$, which we also denote by $d$. Given two points $x,y \in F(\Delta)$, let $[x,y]$ denote the positively oriented subarc of $F(\Delta)$ from $x$ to $y$, according to the counterclockwise orientation on the curve $F(\Delta)\subset \C$. For each $\varepsilon>0$ there exists a partition of $F(\Delta)$ into arcs $[x_{m-1},x_m]$, $m\in \{1,\dots,n\}$, where $x_0=x_n$, and $d(x_{m-1},x_m)<\varepsilon$ for each $m$. We also require that the images of the vertices of $T$ are contained in the collection $\{x_m\}_{m=1}^n$. This guarantees that 
\begin{align*}
    \sum_{m=1}^n d(x_{m-1},x_m)  =\sum_{m=1}^n \ell_d([x_{m-1},x_m])= \ell_d( F(\Delta)) .
\end{align*}
Consider the Euclidean polygon formed by joining $x_{m-1}$ with $x_m$ for all $m\in \{1,\dots,n\}$. Our definition of the embedding $F$ ensures that the polygon does not have self-intersections; see Remark \ref{remark:polygon}. Denote the polygonal region bounded by that polygon by $P$. {By taking $\varepsilon$ to be sufficiently small,  we have that the region $P$ is arbitrarily close} to the region $\Omega$ bounded by $F(\Delta)$. In particular, we choose {$\varepsilon$} so that
\begin{align*}
    \mathcal H_{|\cdot|}^2(P) \leq 2\mathcal H_{|\cdot|}^2( \Omega).
\end{align*}
Since $F$ is $4$-bi-Lipschitz, we have 
\begin{align}\label{ineq:bilip_boundary}
    \ell_{|\cdot|}(\partial P) =\sum_{m=1}^n |x_{m}-x_{m-1}| \leq 4  \sum_{m=1}^n d(x_{m-1},x_m) = 4\ell_d( F(\Delta)).
\end{align}
We consider a {polygonal decomposition} $\{P_k\}_{k\in K}$ of the region $P$ satisfying the conclusions of \Cref{lemma:polygon_triang} with the given $\varepsilon$.

We declare the length of each edge of the polygonal decomposition $\{P_k\}_{k\in K}$ to be $4$ times its Euclidean length. Thus the $1$-skeleton of the decomposition is a $1$-dimensional polyhedral space {with the resulting length metric}. We add to this polyhedral space the arcs $[x_{m-1},x_m] \subset F(\Delta)$, each with length $\ell_d( [x_{m-1},x_m])$.  Note that we do not consider $[x_{m-1},x_m]$ as a subset of the plane, which could intersect the interior of some triangles $P_k$, but as an abstract segment. We denote the resulting $1$-dimensional polyhedral space by $S_1$ and its length metric by $d_1$. Denote by $T_1$ the copy of $F(\Delta)$ in $S_1$.

Let $\gamma$ be a path in $S_1\setminus T_1$ joining two points $x_m,x_l \in T_1$ for some $m,l\in \{1,\dots,n\}$. Then $$d(x_m,x_l) \leq 4|x_m-x_l| \leq 4 \ell_{|\cdot |}(\gamma)= \ell_{d_1}(\gamma).$$
From this, it follows that $d(x,y)\leq d_1(x,y)$ for all points $x,y\in T_1$. If $x,y\in |\beta_j|$ for some $j\in \{1,2,3\}$, then there exists a subpath $\gamma$ of $\beta_j$ connecting $x$ and $y$ with $d(x,y)=\ell_d(\gamma)=\ell_{d_1}(\gamma)$. Thus $d(x,y)\geq d_1(x,y)$, and it follows that $d=d_1$ on $|\beta_j|$. Moreover, by property \ref{item:decomposition_ii} of \Cref{lemma:polygon_triang}, \eqref{ineq:bilip_boundary}, and the relation $\ell_d(F(\Delta)) = \ell_d(\Delta) \leq 3\diam(T)$,  we have 
\begin{align}\label{ineq:diameter}
\diam_{d_1}(S_1) \leq 4 \ell_{|\cdot|}(\partial P) + 2\varepsilon \leq 16  \ell_d( F(\Delta)) + 2\varepsilon \leq 48\diam_d(T)+2\varepsilon .    
\end{align}

We wish to fill in the $1$-skeleton $S_1$ with faces so that we obtain a polyhedral surface $S$ with the desired properties. To each Jordan curve $\partial P_k \subset S_1$ we glue a cube $S(P_k)$ with bottom face removed isometrically {along its boundary, where the boundary of $S(P_k)$ necessarily} has length equal to $\ell_{d_1}(\partial P_k)$. Thus 
\begin{equation} \label{equ:cube_area_1}
    \mathcal H^2(S(P_k)) = (5/16)\ell_{d_1}(\partial P_k)^2 = 5 \ell_{|\cdot|}(\partial P_k)^2.
\end{equation}
{Next, consider a Jordan curve formed by {an arc} $[x_{m-1},x_m]$ and a line segment $I \subset \partial P$. We observe first that $\ell_{d_1}(I) {=} 4 \ell_{|\cdot|}(I) \leq 16 d(x_{m-1},x_m)$, since $F$ is $4$-bi-Lipschitz. Glue a cube $S(x_m)$ with bottom face removed isometrically into this Jordan curve along its boundary. Then $\partial S(x_m)$ has length at most $17 d(x_{m-1},x_m)$ and thus} 
\begin{equation} \label{equ:cube_area_2}
    \mathcal{H}^2(S(x_m)) \leq L_0  d(x_{m-1},x_m)^2
\end{equation}  
{for $L_0 = 5\cdot (17/4)^2$.} Denote by $S$ the resulting polyhedral space and by $d_S$ the resulting length metric. By construction, $S$ is a closed topological disk with boundary $T_1$. We define $\varphi\colon S \to T$ to be {an arbitrary homeomorphism such that $\varphi|_{\partial S} = F^{-1}$.} 

It is immediate that $d_S(x,y)=d_1(x,y)$ for all $x,y\in S_1\subset S$. Indeed, any path inside an attached cube with endpoints on the boundary has longer length than the path on the boundary of the cube that has the same endpoints. This is the reason for attaching cubes to $S_1$. Since $d\leq d_1=d_S$ on $T_1$, we immediately obtain \ref{item:filling_4}. Moreover, $d=d_1=d_S$ on $|\beta_j|$ for each $j\in \{1,2,3\}$, so we also obtain \ref{item:filling_3}. For \ref{item:filling_2}, {we use \eqref{equ:cube_area_1} and \eqref{equ:cube_area_2} to get}
\begin{align*}
    \mathcal H^2(S) &=\sum_{k\in K} \mathcal H^2(S(P_k)) +\sum_{m=1}^n \mathcal H^2(S(x_m)) \\ & \leq \sum_{k\in K} 5  \ell_{|\cdot|}(\partial P_k)^2 +  \sum_{m=1}^n L_0  d(x_{m-1},x_m)^2.
\end{align*}
Applying property \ref{item:decomposition_iii} from Lemma \ref{lemma:polygon_triang} and the relationship $d(x_{m-1},x_m)<\varepsilon$, we obtain
\begin{align*}
   \mathcal H^2(S) &\leq  85 \mathcal H^2(P)+ L_0  \varepsilon \sum_{m=1}^n d(x_{m-1},x_m)\\
    &\leq 170 \mathcal H^2(\Omega)+ L_0 \varepsilon  \ell_d(F(\Delta)).
\end{align*}
We choose a sufficiently small $\varepsilon$ so that the second term of the sum is bounded by $\mathcal H^2 (\Omega)$. Then, by \Cref{thm:besicovitch}, we have \[\mathcal H^2(\Omega) { \leq (4/\pi) \cdot 16 \mathcal H^2(T)\leq 32\mathcal H^2(T)}.\] It follows that $\mathcal H^2(S)\leq L\mathcal H^2(T)$  {for $L = 32\cdot 171$.} 

Finally, we verify \ref{item:filling_1}. Since  $\ell_{|\cdot|}(\partial P_k)<\varepsilon$, we have
\[\diam_{d_S} (S(P_k)) \leq \frac{3}{4} \ell_{d_S}(\partial S(P_k)) = \frac{12}{4} \ell_{|\cdot|}(\partial P_k) \leq 3\varepsilon.\]
Moreover,
$$\diam_{d_S}(S(x_m))  \leq \frac{3}{4} \ell_{d_S}(\partial S(x_m))\leq \frac{3}{4}\cdot 17 d(x_{m-1},x_m) \leq 13\varepsilon .$$
Therefore, by \eqref{ineq:diameter}, 
\begin{align*}
    \diam_{d_S}(S) & \leq \diam_{d_S}(S_1) + 2\max_{k\in K} \diam_{d_S}(S(P_k))\\
    &\qquad\qquad+ 2\max_{m\in \{1,\dots,n\}} \diam_{d_S}(S(x_m)) \\ 
    & \leq (48\diam_d(T) +2\varepsilon) + 6\varepsilon + 26\varepsilon.
\end{align*}
Choose $\varepsilon$ so that $34\varepsilon < \diam_d (T)$. Thus $\diam_{d_S}(S) \leq 49 \diam_d(T)$. This completes the proof.
\end{proof}

\bigskip

\section{Building the approximating surfaces} \label{sec:approximating_surfaces}

This section is dedicated to the proof of \Cref{thm:main}. First, we carry out some technical preparations in \Cref{sec:improved_triangulations}. We then give the proof of \Cref{thm:main} in \Cref{sec:1.1_proof}. We conclude this section with a discussion of the case where $X$ is homeomorphic to $\mathbb{C}$, in preparation for the proof of \Cref{thm:one-sided_qc}.

\bigskip

\subsection{Improved triangulations} \label{sec:improved_triangulations}

We start by stating the main result by Creutz and the second-listed author in \cite{CR:21} on the existence of {decompositions} of a length surface {into non-overlapping convex triangles}. Let $X$ be a length surface. We say that a collection $\mathcal T$ of non-overlapping closed Jordan regions $T\subset X$ is a \textit{geometric triangulation} of $X$ if it is locally finite, it covers $X$, and each $T\in \mathcal T$ is a triangular disk, endowed with the restriction of the metric of $X$.  We remark that a {geometric} triangulation is not necessarily a triangulation in the usual topological sense, since we do not require that the edges of triangles match exactly.

We employ the following terminology. Recall that a set $P \subset X$ is \textit{convex} if any two points $x,y\in P$ can be joined by a geodesic contained in $P$, and in this case $P$ is a length space with the restriction of the metric of $X$ and the inclusion map from $P$ to $X$ is an isometric embedding. If all triangular disks in the geometric triangulation $\mathcal T$ of $X$ are convex, then we say that $\mathcal T$ is \textit{convex}. For a {geometric} triangulation $\mathcal T$, we also define $\mesh(\mathcal T)$ to be the supremum of diameters of triangular disks $T\in \mathcal T$. Finally, a surface $X$ has \textit{polygonal boundary} if each boundary component of $X$ consists of non-overlapping geodesics. Note that if a boundary component of $X$ is homeomorphic to $\R$, then it may consist of infinitely many such geodesics.  Moreover, any surface whose boundary is empty necessarily has polygonal boundary.

Now we state the main result of \cite{CR:21}.

\begin{thm} \label{thm:triangulation}
Let $X$ be a length surface with polygonal boundary and $\varepsilon>0$. Then there exists a convex triangulation $\mathcal T$ of $X$ with $\mesh(\mathcal T)<\varepsilon$.
\end{thm}

For the proof of \Cref{thm:main}, we need {to refine the triangulation given by} \Cref{thm:triangulation} to guarantee that the edge graph is approximately isometric to the original space $X$. This is similar to Proposition 7.5.5 of \cite{BBI:01}. 

Given two triangulations ${\mathcal T_1}$ and $\mathcal T_2$ of $X$, we say that $\mathcal T_2$ is a \textit{refinement} of $\mathcal T_1$ if for every triangular disk $T\in \mathcal T_2$ there exists a triangular disk $T'\in \mathcal T_1$ such that $T\subset T'$. For a triangulation $\mathcal{T}$ of $X$, let $\mathcal{E}(\mathcal{T})$ denote the embedded graph in $X$ consisting of the edges of triangles in $\mathcal{T}$. This is equipped with the length metric induced by {$X$}. If $D\subset X$ is {a connected set that is} the union of triangular disks in $\mathcal T$, then we denote by $\mathcal E( \mathcal T| D)$ the set $\mathcal E(\mathcal T)\cap D$, again equipped with the {induced} length metric.

\begin{prop} \label{prop:good_triangulation}
{Let $X$ be a length surface and $\varepsilon>0$. Then for each convex triangulation $\widetilde {\mathcal T}$ of $X$ with $\mesh(\widetilde {\mathcal T})<\varepsilon/8$ there exists a convex triangulation $\mathcal T$ that is a refinement of $\widetilde {\mathcal T}$ with the property that  the inclusion map from $\mathcal{E}(\mathcal{T})$ to $X$ is an $\varepsilon$-isometry. More generally, if $D$ is a connected union of triangular disks $T\in \widetilde{\mathcal T}$, then the inclusion map from $\mathcal E(\mathcal T| D)$ to $D$, equipped with the length metric induced by $X$, is an $\varepsilon$-isometry.}
\end{prop}
\begin{proof}
Let $(X,d)$ be a length surface, $\varepsilon>0$, and $\varepsilon'< \varepsilon/8$. Let $\widetilde{\mathcal{T}}$ be a convex triangulation of $X$ such that $\diam(\widetilde{T}) \leq \varepsilon'$ for every $\widetilde{T} \in \widetilde{\mathcal{T}}$. Enumerate $\widetilde {\mathcal{T}}$ as $\{\widetilde{T}_j\}_{j\in \widetilde{J}}$, where $\widetilde{J} = \mathbb{N}$ or $\widetilde{J} = \{1, \ldots, n\}$ for some $n \in \mathbb{N}$.

Consider a triangle $\widetilde{T}_j$. Then $\partial \widetilde{T}_j$ is the union of three geodesics $\widetilde{\alpha}_j^1, \widetilde{\alpha}_j^2, \widetilde{\alpha}_j^3$. Pick a finite set of points $\mathcal{W}_j=\{w_j^1,\ldots, w_j^{k_j}\}$ in $\partial \widetilde{T}_j$ such that every point $x\in \partial \widetilde{T}_j$ is within distance $2^{-j}\varepsilon'$ of a point in $\mathcal{W}_j$ on the same edge as $x$. We also include the vertices of $\partial \widetilde{T}_j$ in the collection $\mathcal{W}_j$.  For each pair of points in $\mathcal{W}_j$ we add a geodesic in $\widetilde{T}_j$ connecting them. By applying Lemma 4.3 of \cite{CR:21} inductively, we can do this so that the resulting system of geodesics is a finite graph.

These additional geodesics subdivide $\widetilde{T}_j$ into a finite number of smaller polygonal disks. By Lemma 3.3 in \cite{CR:21}, each of the resulting polygonal disks is still convex and thus is a length space with polygonal boundary (with the restriction of the metric of $X$). By Theorem \ref{thm:triangulation},  we can subdivide these polygonal disks further so that we again have triangular disks. This gives a new convex triangulation $\mathcal{T} = \{T_j\}_{j\in J}$, with the same notational conventions that we used for $\widetilde{\mathcal T}$, that refines $\widetilde{\mathcal T}$. 

Let $D$ be a connected union of triangular disks of the original triangulation $\widetilde{\mathcal T}$. We show that the inclusion map from $\mathcal{E}(\mathcal{T}|D)$ to $D$ is an $\varepsilon$-isometry, where $D$ is endowed with the length metric $d_D$ induced by $X$. First, note that $\mathcal{E}(\mathcal{T}|D)$ contains $\mathcal{E}(\widetilde{\mathcal{T}}|D)$, which is within $d_D$-distance $\varepsilon'<\varepsilon$ from every point of $D$. Hence  $\mathcal{E}(\mathcal{T}|D)$ is $\varepsilon$-dense in $D$. 

Let $d_\mathcal{T}$ denote the length metric on $\mathcal{E}(\mathcal{T}|D)$. Clearly we have $d_D \leq d_{\mathcal{T}}$ on the set $\mathcal{E}(\mathcal{T}|D)$, which is a subset of $D$. We claim that $ d_{\mathcal{T}}< d_D+\varepsilon$ on  $\mathcal{E}(\mathcal{T}|D)$ and this will complete the proof. Let $x,y \in \mathcal{E}(\mathcal{T}|D)$. Then, by the construction of $\mathcal T$, there exist points $x',y' \in \mathcal{E}(\widetilde{\mathcal{T}}|D)$ such that ${d_D(x,x')=}d_\mathcal{T}(x,x') \leq \varepsilon'$ and ${d_D(y,y')=}d_\mathcal{T}(y,y') \leq \varepsilon'$. In particular, we have $d_\mathcal{T}(x,y) \leq d_\mathcal{T}(x',y') + 2\varepsilon'$ and $d_D(x',y') \leq d_D(x,y) + 2 \varepsilon'$. Thus it suffices to show that $d_{\mathcal T} (x',y') \leq d_D(x',y')+ 4\varepsilon'$ for every $x',y'\in \mathcal{E}(\widetilde{\mathcal{T}}|D)$. 

Let $x',y' \in \mathcal{E}(\widetilde{\mathcal{T}}|D)$ and $\gamma$ a curve in $D$ joining $x'$ and $y'$. Inductively define curves $\gamma_j$ for each $j \in \widetilde{J}$ in the following way. Take $\gamma_0 = \gamma$. If $|\gamma_{j-1}|$ intersects the interior of $\widetilde{T}_j$, then let $z_1,z_2$ denote the first and last points of intersection with $\widetilde{T}_j$. Let $\gamma_{z_1}^{z_2}$ denote the maximal subcurve of $\gamma$ from $z_1$ to $z_2$. Choose points $w_1,w_2 \in \mathcal{W}_j$ so that $d(w_k,z_k)\leq 2^{-j}\varepsilon'$ and $w_k$ belongs to the same edge as $z_k$ for each $k \in \{1,2\}$. There are geodesics from $z_1$ to $w_1$, from $w_1$ to $w_2$, and from $w_2$ to $z_2$ contained in $\mathcal{E}(\mathcal{T}|D)$. Let $\widetilde{\gamma}_{z_1}^{z_2}$ be the concatenation of these three paths. It is immediate that 
\begin{align*}
    \ell(\widetilde{\gamma}_{z_1}^{z_2})&=d(z_1,w_1)+d(w_1,w_2)+d(w_2,z_2) \\
    &\leq d(z_1,z_2)+ 2d(z_1,w_1)+2d(w_2,z_2) \\
    &\leq  \ell(\gamma_{z_1}^{z_2})+ 4\cdot 2^{-j}\varepsilon'.
\end{align*}
Let $\gamma_j$ be the curve formed by replacing $\gamma_{z_1}^{z_2}$ with $\widetilde{\gamma}_{z_1}^{z_2}$. If $|\gamma_{j-1}|$ does not intersect the interior of $\widetilde{T}_j$, then take $\gamma_j = \gamma_{j-1}$. Note that $\gamma$ intersects only finitely many of the triangles in $\widetilde{\mathcal{T}}$, so this process must terminate after finitely many steps. This yields a curve $\widetilde{\gamma}$ in $D$. It is immediate that $\ell(\widetilde{\gamma}) \leq \ell(\gamma) + 4\varepsilon'$. Since $\gamma$ is arbitrary, we have that $d_\mathcal{T}(x',y') \leq d_D(x',y') + 4\varepsilon'$.
\end{proof}

\Cref{thm:triangulation} requires the surface to have polygonal boundary. Since we do not impose this restriction in \Cref{thm:main}, we give the following additional lemma on polygonal approximation of the boundary.

\begin{lemm} \label{lemm:polygonal_boundary}
Let $X$ be a length surface and $\varepsilon>0$. There exists a convex set $\widetilde{X} \subset X$ homeomorphic to $X$ having polygonal boundary such that the inclusion map from $\widetilde{X}$ to $X$ is an $\varepsilon$-isometry.
\end{lemm}
Here, the distance between two points in $\widetilde X$ is the length of a shortest curve in $X$ that connects them. By convexity, there is such a curve contained in $\widetilde X$, so this implies that the length metric on $\widetilde{X}$ is the same as the metric inherited from $X$.

\begin{proof}
For each component $Y$ of the boundary $\partial X$, apply the following procedure. Note that $Y$ is homeomorphic to either $\mathbb{R}$ or $\mathbb{S}^1$; assume in the following that it is homeomorphic to $\mathbb{R}$. 

As a consequence of the tubular neighborhood theorem \cite[p. 76]{GP:74}, there is a neighborhood $U_Y$ of $Y$ in $X$ that is homeomorphic to the closed upper half-plane in $\mathbb{C}$, denoted by $\mathbb{H}$. By restricting to a smaller neighborhood if needed, we may assume that $U_Y$ is contained in $\bigcup_{y \in Y} B(y,\dist(y,\partial X \setminus Y)/2)$. Thus, for any two distinct components $Y_1,Y_2 \subset \partial X$, the neighborhoods $U_{Y_1}$ and $U_{Y_2}$ are disjoint. 

Choose a sequence of points $(y_j)_{j=-\infty}^\infty$ in $Y$, indexed in increasing order according to the parametrization of $Y$ by $\mathbb{R}$. By adding more points if needed, we may assume that $d(y_j, y_{j+1}) < d(y_j, X \setminus U_Y)$ for each $j \in \mathbb{Z}$. In particular, since the closed ball at $y_j$ of radius $d(y_j, y_{j+1})$ is compact, each point $y_j$ is joined to $y_{j+1}$ by a geodesic $\gamma_j$ contained in $U_Y$. Moreover, we may assume the (possibly empty) open region $W_j \subset X$ enclosed by $Y$ and $\gamma_j$ has diameter at most $\varepsilon$ (given in the statement). Note that each component of $W_j$ is a Jordan region. This follows from  Ker\'ekj\'art\'o's theorem \cite[Chapter IV.16, p.~168]{New:51}.

Finally, by redefining the geodesics if needed, we may assume that for all distinct values $j,k \in \mathbb{Z}$ the sets $W_j$ and $W_k$ are disjoint. To justify this claim, fix a bijection $\varphi\colon \mathbb{N} \to \mathbb{Z}$ and apply the following inductive procedure. For some $n \in \mathbb{N}$, suppose that $\gamma_{\varphi(1)}, \ldots, \gamma_{\varphi(n)}$ are such that the sets $W_{\varphi(1)}, \ldots, W_{\varphi(n)}$ are mutually disjoint. Suppose that $\gamma_{\varphi(n+1)}$ intersects $W_{\varphi(1)}$. For each component $(t_1,t_2)$ of $\gamma_{\varphi(n+1)}^{-1}(W_{\varphi(1)})$, we redefine $\gamma_{\varphi(n+1)}$ on $[t_1,t_2]$ to coincide with the subarc of $\gamma_{\varphi(1)}$ from $\gamma_{\varphi(n+1)}(t_1)$ to $\gamma_{\varphi(n+1)}(t_2)$; note that the interior of this subarc lies inside $W_{\varphi(n+1)}$. The resulting curve, still denoted by $\gamma_{\varphi(n+1)}$, does not intersect $W_{\varphi(1)}$ and is also a geodesic. As a byproduct, we also have that $\gamma_{\varphi(1)}$ does not intersect $W_{\varphi(n+1)}$ and that $W_{\varphi(1)}$ and $W_{\varphi(n+1)}$ are disjoint. Since $W_{\varphi(1)}, \ldots, W_{\varphi(n)}$ are mutually disjoint, we also see that this redefining of $\gamma_{\varphi(n+1)}$ does not introduce new intervals of intersection between $\gamma_{\varphi(n+1)}$ and $W_{\varphi(k)}$ for some $k \in \{1, \ldots, n\}$. Finally, this procedure only makes the set $W_{\varphi(n+1)}$ smaller, so the property that $W_{\varphi(n+1)}$ has diameter at most $\varepsilon$ remains. Apply this same redefining procedure for $W_{\varphi(k)}$ for all $k \in \{2, \ldots, n\}$. This completes the inductive step.

Let $\widetilde{E}_Y = \bigcup_{j=-\infty}^\infty |\gamma_j|$. Consider its image $\widetilde{E}_Y'$ in $\mathbb{C}$ under the homeomorphism from $U_Y$ to $\mathbb{H}$. { Let $V$ denote the unbounded component of $\mathbb{C} \setminus \widetilde{E}_Y'$ in $\mathbb{H}$ and $\widehat{V}$ the same set as a subset of the Riemann sphere $\widehat{\mathbb{C}}$. Then $\partial \widehat{V}$ is connected and locally connected. By Theorem IV.6.7 in \cite{Wil:49}, there is a Jordan curve in $\partial \widehat{V}$ (in fact, $\partial \widehat{V}$ itself) separating $\widehat{V}$ and $\widehat{\mathbb{C}} \setminus \overline{\widehat{V}}$. From this we obtain an arc $E_Y'$ in $\partial V$ separating $V$ and $\mathbb{C} \setminus \overline{V}$.}  The preimage of $E_Y'$ is an open arc $E_Y \subset \widetilde{E}_Y$ in $X$. Let $V_Y$ denote the component of $U_Y \setminus E_Y$ not intersecting $Y$.  Observe that, since $U_Y$ and $V_Y$ are both topological closed half-planes there is a homeomorphism from $U_Y$ to $V_Y$ that is the identity outside of some neighborhood of the closed region bounded by $E_Y$ and $Y$. 

If $Y$ is homeomorphic to $\mathbb{S}^1$, then we apply a similar procedure to obtain a Jordan curve $E_Y$ separating $Y$ and $\partial X \setminus Y$ and sets $U_Y$ and $V_Y$. 

Let $E = \bigcup E_Y$, where the union ranges over all components $Y$ of $\partial X$. Note that $E_{Y_1}$ and $E_{Y_2}$ are disjoint for any distinct components $Y_1, Y_2 \subset \partial X$. Then $X \setminus E$ contains a unique component $X'$ not intersecting $\partial X$. Let $\widetilde{X}$ be the union of the sets $E_Y$ and $X'$. Gluing the respective homeomorphisms for each component $Y \subset \partial X$  and the identity map on a suitable subset of $X'$, we obtain a  homeomorphism from $X$ to $\widetilde{X}$. 

Next, we show that $\widetilde{X}$ is a convex subset of $X$. Consider a path $\gamma \colon [t_1,t_2] \to X$ between two points in $\widetilde{X}$. We wish to find a path $\widetilde{\gamma}$ contained in $\widetilde{X}$ of shorter length connecting the same endpoints. If $\gamma$ is already contained in $\widetilde{X}$, then we take $\widetilde{\gamma} = \gamma$. Otherwise, by restricting to subcurves if needed, we may assume that $\gamma(t_1)$ and $\gamma(t_2)$ are in the same boundary component $E_Y\subset \partial \widetilde{X}$ and $\gamma((t_1,t_2)) \subset X \setminus \widetilde{X}$. Let $\widetilde{\gamma} \colon [t_1,t_2] \to X$ be the path in $E_Y$ from $\gamma(t_1)$ to $\gamma(t_2)$.  

We claim that $\ell(\widetilde{\gamma}) \leq \ell(\gamma)$. To show this, define a path $\zeta$ in the following way: for each component $(t_3,t_4)$ of $\gamma^{-1}(W_j)$ for each $j \in \mathbb{Z}$, define $\zeta$ on $[t_3,t_4]$ to traverse the subarc of $\gamma_j$ homotopic to $\gamma|_{[t_3,t_4]}$ relative to the endpoints.  Since $\gamma_j$ is a geodesic, we must have $\ell(\zeta|_{[t_3,t_4]}) \leq \ell(\gamma|_{[t_3,t_4]})$. Define $\zeta$ to coincide with $\gamma$ otherwise. It follows that $\ell(\zeta) \leq \ell(\gamma)$. 

Consider a point $x = \widetilde{\gamma}(t)$ for some $t \in (t_1,t_2)$. Then $x \in \partial W_j$ for some $j \in \mathbb{Z}$. Since each component of $W_j$ is a Jordan region, we can find a curve $\eta$ from $x$ to $Y \setminus |\gamma_j|$ contained in $W_j$ except for its endpoints. Observe that the curve $\eta$ separates $X\setminus \widetilde{X}$. Now $\gamma(t_1)$ and $\gamma(t_2)$ are in different components of $E_Y \setminus \{x\}$. Consequently, $\gamma^{-1}(W_j)$ must contain a component $(t_3,t_4)$ such that $\gamma(t_3)$ and $\gamma(t_4)$ are in different components of $\partial W_j \setminus |\eta|$. Then $\zeta|_{[t_3,t_4]}$ joins the same endpoints and has image contained in $|\gamma_j|$. It follows that $x \in |\zeta|$. Thus $|\widetilde{\gamma}| \subset |\zeta|$, and we have $\ell(\widetilde{\gamma}) \leq \ell(\zeta) \leq \ell(\gamma)$. 

It remains to show that the inclusion map $\widetilde{X} \to X$ is a $\varepsilon$-isometry. Since $\widetilde X$ is convex, it follows that the inclusion map is an isometric embedding. Finally, each point in $X$ is within distance $\varepsilon$ of a point in $\widetilde{X}$. This completes the proof.
\end{proof}

\bigskip

\subsection{Proof of \Cref{thm:main}} \label{sec:1.1_proof}

Let $(X,d)$ be a length surface. Choose a sequence $(\varepsilon_n)_{n=1}^\infty$ of positive reals satisfying $\varepsilon_n \to 0$ as $n \to \infty$. 

We apply \Cref{lemm:polygonal_boundary} to find a surface $\widetilde{X}_n \subset X$ that is homeomorphic to $X$ and has polygonal boundary such that the inclusion map is a $\varepsilon_n$-isometry. Moreover, $\widetilde{X}_n$ is convex as a subset of $X$, so that the restriction of $d$ to $\widetilde{X}_n$ is still a length metric.

Since the space $\widetilde{X}_n$ has polygonal boundary, we can apply \Cref{thm:triangulation} and \Cref{prop:good_triangulation} with the parameter $\varepsilon_n$ to obtain a decomposition $\widetilde{\mathcal{T}}_n$ of $\widetilde{X}_n$ into convex triangular disks with $\mesh( \widetilde{\mathcal T}_n) <\varepsilon_n$. We consider the edge graph $\widetilde{\mathcal{E}}_n = \mathcal{E}(\widetilde{\mathcal{T}}_n)$ as having the induced length metric $\widetilde{d}_n$. As given by \Cref{prop:good_triangulation}, the inclusion map $id\colon \widetilde{\mathcal{E}}_n \to \widetilde{X}_n$ is a $\varepsilon_n$-isometry. More precisely, $d(x,y) \leq \widetilde{d}_n(x,y) < d(x,y) + \varepsilon_n$ for all $x,y \in \widetilde{\mathcal{E}}_n$.

For each triangular disk $T \in \widetilde{\mathcal{T}}_n$ with metric $d$, consider the polyhedral surface $S$ and the corresponding homeomorphism $\varphi_T\colon S \to T$ given by \Cref{thm:triangular_surface}.  Observe that $\varphi_T|_{\partial S}$ is length-preserving as a map from $\partial S$ into $\widetilde{\mathcal{E}}_n$,  with either the metric $d$ or the metric $\widetilde{d}_n$. Thus we may define a length surface $X_n$ by gluing each disk $S$ into $\widetilde{\mathcal{E}}_n$ along the map $\varphi_T$. Denote the metric on $X_n$ by $d_n$. We obtain a homeomorphism $\Phi_n\colon X_n \to \widetilde{X}_n$ by gluing the maps $\varphi_T$. Let $\mathcal{E}_n = \Phi_n^{-1}(\widetilde{\mathcal{E}}_n)$. For each $x \in \mathcal{E}_n$, put $\widetilde{x} = \Phi_n(x) \in \widetilde{\mathcal{E}}_n$.

Since $\widetilde{d}_n$ is a length metric on $\widetilde{\mathcal{E}}_n$ and $d_n$ is a length metric on $X_n$, it is immediate that $d_n(x,y) \leq \widetilde{d}_n(\widetilde{x},\widetilde{y})$ for all $x,y \in \mathcal{E}_n$. 
On the other hand, let $x,y \in \mathcal{E}_n$ and consider an arbitrary path $\gamma$ in $X_n$ from $x$ to $y$. For each disk $S$, consider each component $(a,b)$ of the set $\gamma^{-1}( \Int(S))$. We deduce from \Cref{thm:triangular_surface} \ref{item:filling_4} that $\ell_{d_n}(\gamma|_{(a,b)}) \geq d_S(\gamma(a),\gamma(b))\geq  d(\Phi_n(\gamma(a)),\Phi_n(\gamma(b)))$,  where $d_S$ denotes the length metric on $S$. This implies that $\ell_{d_n}(\gamma) \geq d(\widetilde{x},\widetilde{y})$. Since $\gamma$ is arbitrary, we have
$d(\widetilde{x},\widetilde{y}) \leq d_n(x,y) $ for all $x,y \in \mathcal{E}_n$. In summary, we have
\begin{equation} \label{equ:isometry_estimate}
    d(\widetilde{x},\widetilde{y}) \leq d_n(x,y) \leq \widetilde{d}_n(\widetilde{x},\widetilde{y}) < d(\widetilde{x},\widetilde{y})+\varepsilon_n
\end{equation}
for all $x,y, \in \mathcal{E}_n$.

Define the map $f_n\colon X_n \to X$ as the composition of $\Phi_n\colon X_n \to \widetilde{X}_n$ and the inclusion map of $\widetilde{X}_n$ in $X$. We claim that $f_n$ is $((2L+3)\varepsilon_n)$-isometric, where $L$ is the constant in \Cref{thm:triangular_surface}. First note that $f_n(X_n)$ is $\varepsilon_n$-dense in $X$, since $\Phi_n$ is surjective and the inclusion map from $\widetilde X_n$ to $X$ is a $\varepsilon_n$-isometry. Next, let $x,y \in X_n$. Then, by \Cref{thm:triangular_surface} \ref{item:filling_1} and the fact that $\mesh(\widetilde{\mathcal T}_n) <\varepsilon_n$, there exist $x', y' \in \mathcal{E}_n$ such that $d_n(x',x) < L\varepsilon_n$, $d_n(y',y) < L\varepsilon_n$,  $d(f_n(x'),f_n(x)) < \varepsilon_n$, and $d(f_n(y'),f_n(y)) < \varepsilon_n$. These properties and \eqref{equ:isometry_estimate} imply that
\begin{align*}
    d(f_n(x),f_n(y)) & < d(f_n(x'),f_n(y')) + 2\varepsilon_n \\
    & \leq d_n(x',y') + 2\varepsilon_n < d_n(x,y) + (2L+2)\varepsilon_n.
\end{align*}
In the other direction, we have
\begin{align*}
    d_n(x,y) & < d_n(x',y') + 2L\varepsilon_n \\ 
    &< d(f_n(x'),f_n(y')) + (2L+1)\varepsilon_n \\
    & < d(f_n(x),f_n(y)) + (2L+3)\varepsilon_n.
\end{align*}
This concludes the proof that $f_n$ is $((2L+3)\varepsilon_n)$-isometric.

Next, we verify the property \ref{item:main_2} regarding the Hausdorff 2-measure. Let $A \subset X$ be a {compact} set and fix $\delta>0$. Let $n$ be sufficiently large so that $\diam(T)<\delta$ for every triangular disk $T\in \widetilde{\mathcal T}_n$. Then $T\subset N_{\delta}(A)$ whenever $T\cap A\neq \emptyset$. The set $f_n^{-1}(A)$ (which could be empty) is covered by the sets $f_n^{-1}(T)$ for which $T\cap A\neq \emptyset$. Moreover, by \Cref{thm:triangular_surface} \ref{item:filling_2} we have $\mathcal H^2(f_n^{-1}(T)) \leq L \mathcal H^2( T)$ for each $T\in \widetilde{\mathcal T}_n$. Since the boundary of each triangle $T$ has Hausdorff $2$-measure zero, we have $$\mathcal H^2(f_n^{-1}(A)) \leq L\mathcal H^2(N_\delta(A)).$$
Hence, letting $n\to\infty$ and then $\delta\to 0$ gives
$$\limsup_{n\to\infty}\mathcal  H^2(f_n^{-1}(A)) \leq L\mathcal H^2(A),$$
as desired.\qed

\bigskip

\subsection{The case where $X$ is homeomorphic to $\C$}
\label{sec:planar_case}
To prepare for the proof of \Cref{thm:one-sided_qc}, we refine the work in this section for the case where $X$ is homeomorphic to $\mathbb{C}$. Note that $X$ has no boundary as a manifold. Choose a decreasing sequence $(\varepsilon_n)_{n=1}^\infty$ of positive real numbers such that $\varepsilon_n \to 0$ as $n \to \infty$. 

First, we describe the existence of a sequence of nested triangulations of $X$. By \Cref{thm:triangulation} and \Cref{prop:good_triangulation}, there exists a triangulation $\mathcal T_1$ of $X$ such that the inclusion from the edge graph $\mathcal E(\mathcal T_1)$ to $X$ is a $\varepsilon_1$-isometry. Now, to each triangular disk $T\in \mathcal T_1$ we apply \Cref{thm:triangulation} to obtain a triangulation of $T$ with mesh less than $\varepsilon_2/8$. Note that the union of these triangulations gives a triangulation of $X$ also with mesh less than $\varepsilon_2/8$. By \Cref{prop:good_triangulation} we may refine this triangulation to obtain a triangulation $\mathcal T_2$ of $X$ such that the inclusion from $\mathcal E(\mathcal T_2)$ to $X$ is a $\varepsilon_2$-isometry. By construction, $\mathcal T_2$ is a refinement of $\mathcal T_1$. We proceed in the same way to obtain triangulations $\mathcal T_n$ such that $\mathcal T_{n+1}$ is a refinement of $\mathcal T_n$ and the inclusion from $\mathcal E(\mathcal T_{n})$ to $X$ is a $\varepsilon_n$-isometry for each $n\in \N$. 

Moreover, according to the last part of \Cref{prop:good_triangulation}, for each connected union $D$ of triangular disks $T\in {\mathcal T_n}$ the inclusion map from $\mathcal E(\mathcal T_{n+1}|D)$ to $D$, endowed with the length metric induced by $X$, is a $\varepsilon_{n+1}$-isometry. Note that for each $k\in \N$, if $D$ is a connected union of triangular disks of $\mathcal T_k$, then $D$ is also a connected union of triangular disks of $\mathcal T_n$ for $n\geq k$. Thus, the inclusion map from $\mathcal E(\mathcal T_{n}|D)$ to $D$ is a $\varepsilon_n$-isometry for every $n\geq k+1$.

Since $X$ is homeomorphic to $\C$, for each $k\in \N$ there exists a closed topological disk $\br{D_k}$ that is the union of triangular disks of $\mathcal T_k$ and such that $\br{D_k} \subset \br{D_{k+1}}$ and $X= \bigcup_{k=1}^\infty \br{D_k}$. We consider $\overline{D_k}$ as being equipped with the length metric induced by $d$, denoted by $d_k$. For each $k,n \in \mathbb{N}$ such that $n \geq k$, let $\mathcal{T}_k^n$ denote the subset of $\mathcal{T}_n$ consisting of triangular disks contained in $\overline{D_k}$. By the above, for $n\geq k+1$ the inclusion map from $\mathcal E (\mathcal T_{k}^n)$ to $\br{D_k}$, considered with the metric $d_k$, is a $\varepsilon_{n}$-isometry. 

Consider the polyhedral surfaces $X_n$ and the maps $f_n\colon X_n\to X$ as constructed in Section \ref{sec:1.1_proof}, corresponding to the triangulations $\mathcal T_n$. Since $X$ has no boundary, the maps $f_n$ as constructed are homeomorphisms (rather than topological embeddings). For each $k,n \in \mathbb{N}$ satisfying $n \geq k$, let $\overline{D_k^n} = f_n^{-1}(\overline{D_k})$. The metric on $X_n$ induces a length metric on $\overline{D_k^n}$, which we denote by $d_k^n$. Because the triangulations are nested, the space $\overline{D_k^n}$ is a polyhedral surface for all $n \geq k$. 

We claim that the conclusions of \Cref{thm:main} are also valid for each fixed $k\in \N$ and for the sequence $f_n \colon ( \overline{D_k^n}, d_k^n)\to  (\overline{D_k}, d_k)$, $n\geq k$. 

\begin{lemm}\label{lemma:convergence_D}
For each $k\in \N$, the sequence of homeomorphisms
$$f_n\colon ( \overline{D_k^n}, d_k^n)\to  (\overline{D_k}, d_k),\quad n\geq k,$$
is approximately isometric. Moreover, for each compact set $A\subset \overline{D_k}$ we have
$$\limsup_{n\to\infty}\mathcal H^2 (f_n^{-1}(A)) \leq K\mathcal H^2(A).$$
\end{lemm}
Here the Hausdorff $2$-measures refer to the metrics of $d_{k}^n$ and $d_k$, respectively, but one can use instead the measures with respect to the metrics of $X_n$ and $X$.

We now justify the lemma. For each $n\geq k$, let $\mathcal{S}_k^n = \{f_n^{-1}(T): T \in \mathcal{T}_k^n\}$. Then $\mathcal{S}_k^n$ covers $\overline{D_k^n}$ and consists of those triangular disks used to construct $X_n$ that are contained in $\overline{D_k^n}$. From here, we follow the same argument as in the second part of \Cref{sec:1.1_proof}, with $\mathcal{E}(\mathcal{S}_k^n)$ and $\mathcal{E}(\mathcal{T}_k^n)$ taking the roles of $\mathcal{E}_n$ and $\widetilde{\mathcal{E}}_n$ in \Cref{sec:1.1_proof}, respectively, to conclude that $f_n|_{\overline{D_k^n}}$ is a $((2L+3)\varepsilon_n)$-isometry from $\overline{D_k^n}$ to $\overline{D_k}$, for $n\geq k+1$. The details are omitted. 

Finally, since the Hausdorff 2-measures on $\overline{D_k}$ with respect to $d$ and $d_k$ coincide, and similarly for $\overline{D_k^n}$, we see directly from the argument in Section \ref{sec:1.1_proof} that the conclusion regarding the Hausdorff 2-measure is also satisfied.

\bigskip

\section{Uniformization of surfaces} \label{sec:uniformization}

In this section, we prove {our main result on the uniformization problem,} Theorem \ref{thm:one-sided_qc}, {as well as} Corollary \ref{cor:plateau} and Corollary \ref{cor:bonk-kleiner}. {In \Cref{section:compact}, we prove the compact case of \Cref{thm:one-sided_qc}, i.e., where $X$ is homeomorphic to $\widehat{\C}$ or $\overline{\mathbb D}$, along with \Cref{cor:plateau}.}  {In \Cref{sec:qs}, we use a standard argument to derive the Bonk--Kleiner theorem (\Cref{cor:bonk-kleiner}) from \Cref{thm:one-sided_qc}.} Finally, in Section \ref{sec:noncompact}, we complete the proof of Theorem \ref{thm:one-sided_qc} in the non-compact case {where $X$ is homeomorphic to $\mathbb{C}$}. 

We start with some definitions that are used throughout this section. Recall that a continuous map $\nu\colon X\to Y$ between topological spaces is {monotone} if the preimage of every {point} under $\nu$ is connected. 

Let $X,Y$ be metric spaces of locally finite Hausdorff $2$-measure. A mapping $h\colon X\to Y$ is \textit{quasiconformal} (according to the geometric definition) if $h$ is a homeomorphism and there exists $K\geq 1$ such that for all curve families $\Gamma$ in $X$ we have
$$K^{-1}\Mod \Gamma\leq \Mod h(\Gamma) \leq K\Mod \Gamma.$$
In this case, we say that $h$ is $K$-quasiconformal.  A mapping $h\colon X\to Y$ is \textit{weakly quasiconformal} if $h$ is continuous, surjective, and monotone and there exists $K\geq 1$ such that for every curve family $\Gamma$ in $X$ we have
$$\Mod \Gamma\leq K \Mod h(\Gamma).$$
In this case, we say that $h$ is weakly $K$-quasiconformal. Recall that if $h\colon X \to Y$ is continuous and $\Gamma$ is a curve family in $X$, then $h(\Gamma)$ denotes the curve family $\{h\circ \gamma :\gamma\in \Gamma\}$.

\bigskip

\subsection{Compact metric surfaces}\label{section:compact}
{Here, we give the proof of} Theorem \ref{thm:one-sided_qc} in the case where {$X$ is homeomorphic to $\widehat{\C}$ or $\Omega=\overline{\mathbb D}$}. This theorem follows readily from the following auxiliary result.

\begin{thm}\label{theorem:uniformization_long}
Let $\Omega=\widehat\C$ or $\Omega=\overline{\mathbb D}$. Suppose that $X$ is a length surface homeomorphic to $\Omega$ with $\mathcal H^2(X)<\infty$ and $\{X_n\}_{n=1}^\infty$ is a sequence of polyhedral Riemann surfaces homeomorphic to $\Omega$ converging in the Gromov--Hausdorff sense to $X$. Let $f_n\colon X_n\to X$ be an approximately isometric sequence such that there exists $K\geq 1$ with 
$$\limsup_{n\to\infty}\mathcal H^2 (f_n^{-1}(A)) \leq K\mathcal H^2(A)$$
for all compact sets $A\subset X$. If $h_n\colon \Omega\to X_n$, $n\in \N$, is a normalized sequence of conformal parametrizations, then $f_n\circ h_n$ has a subsequence that converges uniformly to a weakly $K$-quasiconformal map $h\colon \Omega \to X$ with $h\in N^{1,2}(\Omega,X)$. 
\end{thm}

Here we say that a sequence $h_n\colon X_n\to Y_n$ of homeomorphisms between {compact} metric spaces is \textit{normalized} if there exists a value $\delta>0$ and a sequence of triples $a_n,b_n,c_n\in X_n$ with mutual distances bounded from below by $\delta$ such that the mutual distances between the points $h_n(a_n),h_n(b_n),h_n(c_n)$ are also bounded from below by $\delta$, where $\delta$ is independent of $n\in \N$.

In fact, Theorem \ref{theorem:qc_definitions_williams} implies that the conclusion that $h\in N^{1,2}(\Omega,X)$ is redundant and follows from the weak quasiconformality. Moreover, one may obtain the conclusions of Theorem \ref{theorem:uniformization_long} (with a different constant) under the more general assumptions that $X_n$ are length spaces, rather than polyhedral surfaces, and the mappings $h_n$ are $K'$-quasiconformal for some uniform $K'\geq 1$, rather than conformal. However, we do not need this generality for the proof of Theorem \ref{thm:one-sided_qc}.

Assuming Theorem \ref{theorem:uniformization_long}, we prove Theorem \ref{thm:one-sided_qc} in the compact case.

\begin{proof}[Proof of Theorem \ref{thm:one-sided_qc} for compact $X$]
Let $\Omega=\widehat{\C}$ or $\Omega=\overline{\mathbb D}$. Suppose that $X$ is homeomorphic to $\Omega$ with $\mathcal H^2(X)<\infty$. By Theorem \ref{thm:main}, there exists a sequence $X_n$ of polyhedral surfaces homeomorphic to $\Omega$ that converges in the Gromov--Hausdorff sense to $X$. Moreover, there exists an approximately isometric sequence $f_n\colon X_n\to X$ of topological embeddings such that for some $K\geq 1$ we have
$$\limsup_{n\to\infty}\mathcal H^2 (f_n^{-1}(A)) \leq K\mathcal H^2(A)$$
for all compact sets $A\subset X$. We endow each $X_n$ with the natural complex structure, as in Section \ref{sec:complex}. By the uniformization {theorem for polyhedral surfaces (Theorem \ref{theorem:uniformization})}, there exists {a sequence of} conformal parametrizations $h_n\colon \Omega \to X_n$. In order to apply Theorem \ref{theorem:uniformization_long}, it only remains to normalize the sequence $h_n$.

Suppose that $\Omega=\widehat\C$. Since $X_n$ converges in the Gromov--Hausdorff sense to $X$,  there exists a sequence of triples $a_n',b_n',c_n'\in X_n$ with mutual distances uniformly bounded away from $0$. By precomposing $h_n$ with a M\"obius transformation of $\widehat{\C}$, we may assume that the preimages of these points {under $h_n$} also have the same property. If $\Omega= \overline{\mathbb D}$, then, by  Lemma \ref{lemma:boundary_convergence}, $\diam(\partial X_n)$ is uniformly bounded below away from $0$. Hence, we may find points $a_n',b_n',c_n'\in \partial X_n$ with mutual distances uniformly bounded below. We now precompose $h_n$ with a M\"obius transformation of $\overline {\mathbb D}$, so that the preimages of $a_n',b_n',c_n'$ are the points $1,i,-1\in \partial \mathbb D$.
\end{proof}

{Next, we derive \Cref{cor:plateau} from \Cref{thm:one-sided_qc} and the following lemma.}

\begin{lemm}\label{lemma:monotone_boundary}
Let ${X}$ and ${Y}$ be compact $2$-manifolds with boundary that are homeomorphic and $h\colon  X\to  Y$ be a continuous, surjective, and monotone mapping. Then {$\Int Y \subset h(\Int X)$}, {$\partial Y = h(\partial X)$}, and $h|_{\partial X}\colon \partial X\to \partial Y$ is monotone.  
\end{lemm}
\begin{proof}
This result follows from a theorem of Youngs \cite[p.~92]{You:48}, which asserts that $h$ is the uniform limit of homeomorphisms from $X$ onto $Y$. In particular, this theorem implies that {$\partial Y = h(\partial X)$}. The monotonicity of $h|_{\partial X}\colon \partial X\to \partial Y$ is immediate from the monotonicity of $h$. Since $h$ is surjective, we conclude that {$\Int Y \subset h(\Int X)$}. 
\end{proof}

\begin{proof}[Proof of Corollary \ref{cor:plateau}]
Consider the weakly quasiconformal map $h\colon \overline{\mathbb D} \to {X}$ given by Theorem \ref{thm:one-sided_qc}, which lies in
$N^{1,2}(\overline {\mathbb D},X)$. From Lemma \ref{lemma:monotone_boundary} we conclude that $ h|_{\partial \mathbb D}\colon \partial \mathbb D \to \partial X$ is a monotone parametrization of $\partial X$. 
\end{proof}

We now start the proof of Theorem \ref{theorem:uniformization_long}. The proof is split into two parts: the proof of uniform convergence and the proof of quasiconformality. 

In what follows, we assume, as in the statement of \Cref{theorem:uniformization_long}, that $\Omega=\widehat\C$ or $\Omega=\overline{\mathbb D}$. Moreover, $X$ is a length surface homeomorphic to $\Omega$ with $\mathcal H^2(X)<\infty$ and $\{X_n\}_{n=1}^\infty$ is a sequence of polyhedral surfaces homeomorphic to $\Omega$ converging in the Gromov--Hausdorff sense to $X$. Let $f_n\colon X_n\to X$ be a sequence of $\varepsilon_n$-isometries, where $\varepsilon_n\to 0$ as $n\to \infty$, such that there exists $K\geq 1$ with 
$$\limsup_{n\to\infty}\mathcal H^2 (f_n^{-1}(A)) \leq K\mathcal H^2(A)$$
for all compact sets $A\subset X$.

\subsubsection{Equicontinuity and existence of {the} limiting mapping}\label{section:equicontinuity}
We prove here that if $h_n\colon \Omega\to X_n$ is a normalized sequence of conformal homeomorphisms, then the sequence $f_n\circ h_n\colon \Omega \to X$ is uniformly equicontinuous.  Recall that {a mapping between Riemann surfaces is conformal if it is complex differentiable} in local coordinates; see Section \ref{sec:polyhedral_surfaces}. By Lemma \ref{lemma:polyhedral_conformal} \ref{corollary:polyhedral_conformal}, $h_n$ is also $1$-quasiconformal (according to the geometric definition). The proof of the equicontinuity relies on the fact that that $\limsup_{n\to\infty}\mathcal H^2(X_n)\leq K\mathcal H^2(X)$, which allows us to apply Lemma \ref{lemma:modulus_bound_convergence}.

\begin{lemm}[Equicontinuity]\label{lemma:equicontinuity}
The sequence $f_n\circ h_n\colon \Omega\to X$, $n\in \mathbb N$, is uniformly equicontinuous.
\end{lemm}

\begin{proof}
We claim that for each $\varepsilon>0$ there exists $\delta>0$ and $N\in \mathbb N$ such that if $E$ is a continuum in $\Omega$ with $\diam(E)<\delta$, then $\diam(f_n(h_n(E)))<\varepsilon$ for all $n\geq N$. Since $f_n$ is an $\varepsilon_n$-isometry with $\varepsilon_n\to 0$, it suffices instead to show that $\diam(h_n(E))<\varepsilon$ for all $n\geq N$.

We argue by contradiction. Suppose that there exists $\varepsilon_0>0$ such that for every $n\in \mathbb N$ there exists a continuum $E_n\subset \Omega$ with $\diam(E_n)<1/n$, but $\diam(h_{k_n}(E_n)) \geq \varepsilon_0$ for a subsequence $h_{k_n}$ of $h_n$. To simplify the notation, we write $h_n$ instead of $h_{k_n}$.

Since the sequence $h_n$ is normalized, there exist points $a_n,b_n,c_n\in \Omega$ with mutual distances bounded away from $0$, such that their images under $h_n$ also have mutual distances bounded away from $0$. 

Since $\diam(E_n)\to 0$, by passing to a subsequence we may assume that $E_n$ converges to a point in the Hausdorff sense. Then for each $n\in \N$ there exists a curve in $\Omega$ between a pair of the points $a_n,b_n,c_n$ that does not intersect $E_n$ and whose distance to $E_n$ is bounded below away from $0$, uniformly in $n$. Indeed $E_n$ can be very close to only one of the points $a_n,b_n,c_n$, so the other two can be joined by a curve that is away from $E_n$; this can be justified formally using the \textit{linear local connectivity} of $\Omega$, as defined in Section \ref{sec:qs}.  We define $F_n$ to be {the trace of} that curve. We note that 
$$\Mod \Gamma^*(E_n,F_n; \Omega) \to \infty,$$
since the relative distance between $E_n$ and $F_n$, defined by 
$$\Delta(E_n,F_n)= \frac{\dist(E_n,F_n)}{\min\{\diam(E_n),\diam(F_n)\}},$$
tends to infinity. In fact, since $E_n$ converges to a point and $\dist(E_n,F_n)$ is bounded away from $0$, for all sufficiently large $n\in \N$  one can find an annulus separating $E_n$ and $F_n$ with inner radius $r_n$, where $r_n\to 0$, and fixed outer radius $R>0$ so that  
$$\Mod \Gamma^*(E_n,{F_n}; \Omega)  \geq  c\log( R/r_n),$$
where $c>0$ is a uniform constant. See \cite[Section 7.9]{Hei:01} for similar estimates.

Consider the continua $h_n(E_n),h_n(F_n)\subset X_n$. By Lemma \ref{lemma:polyhedral_conformal} \ref{corollary:polyhedral_conformal}, the mappings $h_n$ are $1$-quasiconformal, so we have
$$\Mod \Gamma^*(h_n(E_n),h_n(F_n); X_n) \to \infty.$$
Note that there exists $\eta>0$ such that $\diam(h_n(F_n))>\eta$ for all $n\in \N$, since $h_n(F_n)$ joins a pair of the points $h(a_n),h(b_n),h(c_n)$. Moreover, by assumption, $\diam(h_n(E_n)) \geq \varepsilon_0$. Now, Lemma \ref{lemma:modulus_bound_convergence} with $\delta=\min\{\eta, \varepsilon_0\}$ implies that $$\limsup_{n\to\infty}\Mod \Gamma^*(h_n(E_n),h_n(F_n); X_n)<\infty.$$This is a contradiction.
\end{proof}

Next, we prove the convergence of {the sequence} $f_n\circ h_n$, $n \in \mathbb{N}$.

\begin{lemm}[Convergence]\label{lemma:convergence_to_monotone}
The sequence $f_n\circ h_n\colon \Omega\to X$, $n\in \N$, has a subsequence that converges uniformly to a continuous, surjective, and monotone mapping $h\colon \Omega\to X$.
\end{lemm}

For the {conclusion regarding} monotonicity, we will use the fact that each space $X_n$ is a length space, which allows us to apply Proposition \ref{prop:gh} \ref{prop:lift}.

\begin{proof}
The proof of uniform convergence to a continuous map follows from the Arzel\`a--Ascoli theorem, applied to the uniformly equicontinuous sequence $f_n\circ h_n \colon \Omega\to X$. The surjectivity follows from the fact that the set  $f_n(h_n(\Omega))=f_n(X_n)$ is $\varepsilon_n$-dense in $X$, i.e., $d_X(f_n(X_n),x)<\varepsilon_n$ for all $x\in X$. Thus, the uniform convergence implies that $h(\Omega)=X$. 

{It remains to show that the mapping $h$ is monotone.} Suppose that for some $x\in X$ the set $h^{-1}(x)$ is a disconnected compact subset of $\Omega$. Consider points $a,b$ lying in distinct components of $h^{-1}(x)$. Then, by planar topology, there exists a simple curve $\gamma$ in $\Omega\setminus h^{-1}(x)$ separating the points $a$ and $b$; see \cite[Corollary 3.11, p.~35]{Wh:64}. Since each $h_n$ is a homeomorphism, $h_n\circ \gamma$ separates the points $h_n(a)$ and $h_n(b)$. The convergence of $f_n\circ h_n$ to $h$ implies that $f_n(h_n(a))$ and $f_n(h_n(b))$ converge to $x$, which we consider as a constant path. By Proposition \ref{prop:gh} \ref{prop:lift}, there exists a sequence of paths $\gamma_n\colon [0,1]\to X_n$ such that $\gamma_n(0)=h_n(a)$, $\gamma_n(1)=h_n(b)$, and $f_n\circ \gamma_n$ converges uniformly to the constant path $x$. Since $h_n\circ \gamma$ separates $h_n(a)$ and $h_n(b)$ and $\gamma_n$ joins the two points, we conclude that the two paths intersect. By the uniform convergence of $f_n\circ h_n\circ \gamma$ to $h\circ \gamma$ and of $f_n\circ \gamma_n$ to $x$, we conclude that $h\circ \gamma$ intersects $x$, a contradiction.
\end{proof}

\subsubsection{Regularity of {the} limiting parametrization}\label{sec:limiting_parametrization}
If $h_n\colon \Omega \to X_n$ is a normalized sequence of conformal parametrizations, then by Lemma \ref{lemma:convergence_to_monotone} the sequence $f_n\circ h_n$ has a subsequence that converges uniformly to a continuous, surjective, and monotone mapping $h\colon \Omega \to X$. By passing to a subsequence, we assume that $f_n\circ h_n$ converges to $h$.  We {now} complete the proof of Theorem \ref{theorem:uniformization_long} by proving that the limiting map $h$ is weakly quasiconformal.

Recall that  $f_n\colon X_n\to X$ is an $\varepsilon_n$-isometry, where $\varepsilon_n\to 0$, with the property that for every compact set $A\subset X$ we have
\begin{align}\label{ineq:measure_open}
    \limsup_{n\to\infty}\mathcal H^2(f_n^{-1}(A))\leq K\mathcal H^2(A)
\end{align}
for some uniform constant $K>0$. By Lemma \ref{lemma:polyhedral_conformal}, each mapping $h_n$ has an upper gradient $|Dh_n|$ with the property that 
$$\int_{E} |Dh_n|^2 \, d\mathcal H^2 = \mathcal H^2(h_n(E))$$
for each Borel set $E\subset \Omega$. We first prove that the upper gradients $|Dh_n|$ converge to a weak upper gradient of $h$.

\begin{lemm}[Upper gradient]\label{lemma:weak_upper_gradient}
The sequence of upper gradients $|Dh_n|$ of $h_n$, $n\in \N$, has a subsequence that converges weakly in $L^2(\Omega)$ to a function $g_h$ that is a weak upper gradient of $h$.
\end{lemm}

The argument is classical for mappings between fixed spaces. See, for example, \cite[Theorem 7.3.9, p.~194]{HKST:15}. Since here we also have the additional complication of Gromov--Hausdorff convergence of spaces, we include the proof. 

\begin{proof}
For each $n\in \N$ and for all locally rectifiable paths $\gamma$ in $\Omega$ connecting  points $a,b$ we have
\begin{align}\label{lemma:weak_upper_gradient_seq}
    d_{X_n}(h_n(a),h_n(b)) \leq \int_{\gamma} |Dh_n| \, ds. 
\end{align}
Moreover, $\|Dh_n\|_{L^2(\Omega)}^2 = \mathcal H^2(X_n)$ and the latter is uniformly bounded from above by \eqref{ineq:measure_open}. By the Banach--Alaoglu theorem (see \cite[Theorem 2.4.1]{HKST:15}), there exists a function $g_h\in L^2(\Omega)$ such that a  subsequence of $|Dh_n|$ converges weakly in $L^2(\Omega)$ to $g_h$. We choose a Borel representative of $g_h$. We claim that $g_h$ is a weak upper gradient of $h$, which is the uniform limit of $f_n\circ h_n$.

Note that since each $f_n$ is an $\varepsilon_n$-isometry with $\varepsilon_n\to 0$, we have
$$\lim_{n\to\infty}d_{X_n}(h_n(a),h_n(b))=\lim_{n\to\infty} d_X( f_n(h_n(a)), f_n(h_n(b)))= d_X(h(a),h(b))$$
for every $a,b\in \Omega$. Also, by Mazur's lemma \cite[p.~19]{HKST:15}, there exists a sequence of convex combinations 
$$g_n=\sum_{i=n}^{M_n}\lambda_{i,n}|Dh_i|, \quad \sum_{i=n}^{M_n}\lambda_{i,n}=1, \quad 0\leq \lambda_{i,n}\leq 1, \,\, i\in \{n,\dots,M_n\}, \,\, n\in \N,$$ 
that converge strongly to $g_h$ in $L^2(\Omega)$. By Fuglede's lemma \cite[p.~131]{HKST:15}, there exists a curve family $\Gamma_0$ in $\Omega$ with $\Mod\Gamma_0=0$ such that
$$\lim_{n\to\infty}\int_{\gamma} |g_n-g_h| \, ds=0$$
for all $\gamma\notin \Gamma_0$. Taking the corresponding convex combinations in \eqref{lemma:weak_upper_gradient_seq} and passing to the limit, shows that $g_h$ is a weak upper gradient of $h$. 
\end{proof}

Thus, by passing to a subsequence, we may assume that $|Dh_n|$ converges weakly in $L^2(\Omega)$ to $g_h$.

\begin{lemm}[Quasiconformality]\label{lemma:jacobian}
For each Borel set $E\subset X$ we have
$$\int_{h^{-1}(E)} g_h^2 \, d\mathcal H^2 \leq K \mathcal H^2(E).$$
\end{lemm}
\begin{proof}
Let $E\subset X$ be a Borel set and $A\subset h^{-1}(E)$ be a compact set. By the inner regularity of Hausdorff $2$-measure in $\Omega$, it suffices to show that
\begin{align*}
    \int_{A} g_h^2 \, d\mathcal H^2 \leq K \mathcal H^2(h(A)) \leq K \mathcal H^2(E).
\end{align*}
Since $h(A)$ is compact and $\mathcal H^2$ is finite on $X$, for each $\varepsilon>0$ there exists an open set $U\supset h(A)$ such that 
\begin{align}\label{lemma:jacobian_u}
    \mathcal H^2(\br U) \leq \mathcal H^2( h(A))+\varepsilon.
\end{align}
By the uniform convergence of $f_n\circ h_n$ to $h$, we conclude that $f_n(h_n(A))$ converges in the Hausdorff sense to $h(A)$ as $n\to\infty$, so $f_n(h_n(A)) \subset U$ and thus $h_n(A)\subset f_n^{-1}(U)$ for all sufficiently large $n\in \N$. Combining this with  Lemma \ref{lemma:polyhedral_conformal} \ref{lemma:polyhedral_conformal_change}, we have
\begin{align*}
    \int_{A} |Dh_n|^2 \, d\mathcal H^2 = \mathcal H^2( h_n(A)) \leq \mathcal H^2( f_n^{-1}(\br U))
\end{align*}
for all sufficiently large $n\in \N$. Passing to the limit and using \eqref{ineq:measure_open} and \eqref{lemma:jacobian_u}, we obtain
\begin{align*}
    \limsup_{n\to\infty}\int_{A} |Dh_n|^2 \, d\mathcal H^2 \leq K \mathcal H^2( \br U) \leq K\mathcal H^2(h(A))+K\varepsilon.
\end{align*}
Next, we let $\varepsilon\to 0$. Finally, since $|Dh_n|$ converges weakly in $L^2(\Omega)$ to $g_h$, we see that $|Dh_n|\chi_A$ also converges weakly to $g_h \chi_A$, which implies that 
$$ \int_A g_h^2 \, d\mathcal H^2 \leq \liminf_{n\to\infty} \int_A |Dh_n|^2 \, d\mathcal H^2.$$
This completes the proof.
\end{proof}

\begin{lemm}\label{lemma:quasiconformality_geometric}
We have $h\in N^{1,2}( \Omega, X)$. Moreover, for every curve family $\Gamma$ in $\Omega$ we have
$$\Mod \Gamma \leq K \Mod h(\Gamma).$$
\end{lemm}
\begin{proof}
By Lemma \ref{lemma:weak_upper_gradient}, $g_h$ is a weak upper gradient of $h$. The conclusions now follow from Lemma \ref{lemma:jacobian} and Lemma \ref{lemma:weak_upper_gradient_path_integral} \ref{lemma:weak_upper_gradient_measure_ineq}.
\end{proof}
With this lemma, the proof of Theorem \ref{theorem:uniformization_long} is complete.

\bigskip 

\subsection{Quasisymmetric uniformization}\label{sec:qs}
In this section, we give an alternative proof of the Bonk--Kleiner theorem, stated as Corollary \ref{cor:bonk-kleiner}. We first recall the necessary definitions. Let $(X,d_X)$ and $(Y,d_Y)$ be metric spaces. A homeomorphism $f \colon X \to Y$ is \textit{quasisymmetric} if there exists a homeomorphism $\eta\colon [0,\infty) \to [0, \infty)$ such that 
\[ \frac{d_Y(f(x),f(y))}{d_Y(f(x),f(z))}  \leq \eta\left( \frac{d_X(x,y)}{d_X(x,z)} \right) \]
for all distinct points $x,y,z \in X$. Next, a metric space $X$ is \emph{Ahlfors $2$-regular} if there exists a constant $C\geq 1$ such that for all $0<r<\diam(X)$ we have
\begin{align*}
    C^{-1}r^2\leq \mathcal H^2(B(x,r)) \leq Cr^2.
\end{align*}
Moreover, we say that $X$ is \emph{linearly locally connected} (abbreviated \textit{LLC}) if there exists $\lambda\geq 1$ such that for any ball $B(a,r)\subset X$ the following conditions hold:
\begin{description}
    \item[LLC(1)]If $x,y\in B(a,r)$, then there exists a continuum $E\subset B(a,\lambda r)$ containing $x$ and $y$.
    
    \item[LLC(2)]If $x,y\in X\setminus B(a,r)$, then there exists a continuum $E\subset X\setminus B(a,r/\lambda)$ containing $x$ and $y$.
\end{description}
In this case, we say that $X$ is $\lambda$-LLC.

Let $X$ be a metric $2$-sphere that is Ahlfors $2$-regular and LLC. By a result of Semmes \cite[Theorem B.6]{Sem:96}, $X$ is {quasiconvex}, quantitatively. That is, there exists a constant $c\geq 1$ depending only on the Ahlfors regularity and linear local connectivity constants such that for any two points $x,y\in X$ there exists a curve $\gamma$ connecting them with 
$$\ell(\gamma)\leq cd(x,y).$$
Alternatively, one can obtain the quasiconvexity from a result of Wildrick \cite[Corollary 4.8]{Wil:10}. This implies that we can {replace the metric on $X$ with a bi-Lipschitz equivalent length metric} that is Ahlfors $2$-regular and LLC, quantitatively. Therefore, in order to prove Corollary \ref{cor:bonk-kleiner}, we may assume in addition that $X$ is a length space.

By Theorem \ref{thm:one-sided_qc}, there exists a weakly quasiconformal mapping $h\colon \widehat{\C}\to X$.  Under the Ahlfors $2$-regularity condition, such a mapping $h$ is necessarily a homeomorphism, as follows from Theorem \ref{theorem:homeomorphism} below. Finally, the following general result implies that $h$ is quasisymmetric, thus completing the proof of Corollary \ref{cor:bonk-kleiner}.

\begin{thm}
Let $X$ be a metric $2$-sphere that is Ahlfors $2$-regular and LLC. Suppose that $g\colon \widehat{\C}\to X$ is homeomorphism such that 
$$\Mod \Gamma\leq K\Mod g(\Gamma)$$
for some $K\geq 1$ and for all curve families $\Gamma$ in $\widehat{\C}$.  Then $g$ is quasisymmetric.
\end{thm}

See \cite[Theorem 2.5]{LW:20} or \cite[Proof of Corollary 1.7]{Raj:17} for a proof.

\bigskip

\subsection{Non-compact metric surfaces}\label{sec:noncompact}
Suppose that $X$ is a non-compact length space homeomorphic to $\C$ that has locally finite Hausdorff $2$-measure. In this subsection, we show that there exists a weakly $K$-quasiconformal map $h\colon \Omega \to X$, where $\Omega=\mathbb D$ or $\Omega=\C$. This {proves} Theorem \ref{thm:one-sided_qc} in the non-compact case.

Consider an approximately isometric sequence $f_n\colon X_n\to X$ of topological  embeddings as in Theorem \ref{thm:main}. By the discussion in Section \ref{sec:planar_case}, $f_n$ can be chosen so that the following additional conditions hold. There exists an exhaustion of $X$ by an increasing sequence of closed topological disks $\overline{D_k}$, $k\in \N$, such that for all $n\geq k$, $\overline{D_{k}^n}=f_n^{-1}(\overline{D_k})$ is a polyhedral closed topological disk and ${f_n|_{\br{D_k^n}}}\colon \br{D_k^n}\to \br{D_k}$ is a homeomorphism. Moreover, $\overline{D_k}$ is a length space with metric
$$d_k(x,y)= \inf_{\gamma} \ell_{d_X}(\gamma),$$
where the infimum is taken over all paths $\gamma\subset \overline{D_k}$ connecting $x$ and $y$ and the length of $\gamma$ is computed with the metric of $X$. Note that $d_X\leq d_k$ and that $d_k$ is locally isometric to $d_X$ in $D_k=\Int(\overline{D_k})$. Similarly,  $\overline{D_{k}^{n}}$ is a length space with metric $d_k^n$ defined analogously. {Finally,} as stated in Lemma \ref{lemma:convergence_D}, the conclusions of Theorem \ref{thm:main} are true for the restriction of $f_n$ to $\overline{D_k^n}$. 

We split the proof of the existence of a weakly $K$-quasiconformal parametrization of $X$ into several parts. 

\vspace{1em}

\noindent
\textit{Step 1: Normalizations in $X_n$ and $X$.}
We fix distinct points $p,q\in D_1$. Since $f_n\colon X_n\to X$ is an approximately isometric sequence, there exist points $p_n,q_n\in D_1^n$, {$n \in \mathbb{N}$}, such that $f_n(p_n) \to p$ and $f_n(q_n)\to q$. Here, $D_k^n = \Int(\overline{D_k^n})$. Since ${f_n|_{\br{D_1^n}}}\colon \br{D_1^n}\to \br{D_1}$ is a homeomorphism, we have $f_n(\partial D_1^n)=\partial D_1$ and  
\begin{align*}
    \liminf_{n\to\infty}\dist_{d_{X_n}}(p_n, \partial D_1^n) \geq \dist_X (p,\partial D_1)>0.
\end{align*}
In particular, the distance from $p_n$ to $\partial D_1^n$ is uniformly bounded away from $0$. Since $D_1^n\subset D_k^n$, $k\in \N$, $n\geq k$, it follows that for each $k\in \N$, the distance from $p_n$ to $\partial D_k^n$ is bounded away from $0$, uniformly in $n\geq k$. The same conclusions hold for the point $q_n$. Finally, the distance from $p_n$ to $q_n$ is uniformly bounded away from $0$. All these conclusions hold for the metric $d_{X_n}$ and {thus} also the metric $d_k^n$, which is larger than $d_{X_n}$.

\vspace{1em}

\noindent
\textit{Step 2: Uniformization by disks and normalizations in the plane.}
By Theorem \ref{theorem:uniformization}, for each $k\in \N$ and for $n\geq k$ there exists a conformal map from {$\overline{\mathbb{D}}$ onto $\br{D_{k}^n}$. By precomposing with a M\"obius transformation, we obtain a conformal map $h_k^n$ from} a disk $\br B(0,r_k^n)\subset \C$ with $r_{k}^{n}>1$ onto $\br{D_{k}^n}$ such that $h_k^n(0)=p_n$ and $h_{k}^n(1)=q_n$. 

We claim that for each fixed $k\in \N$ the sequence $\{r_k^n\}_{n\geq k}$ is bounded above. Let $E$ be the unit interval {$[0,1]$} inside $B(0,r_k^n)$ and $F_n=\partial B(0,r_k^n)$. Consider the continua $h_k^n(E)$ and $h_{k}^n(F_n)=\partial D_k^n$, and recall {from Lemma \ref{lemma:convergence_D}} that the sequence ${f_n|_{\br{D_k^n}}}\colon ( \overline{D_k^n}, d_k^n)\to  (\overline{D_k}, d_k)$, $n\geq k$, is approximately isometric. From Lemma \ref{lemma:boundary_convergence}, $\partial D_k^n$ has diameter uniformly bounded below away from $0$ for $n\geq k$. Since $p_n,q_n\in h_k^n(E)$, that set also has diameter uniformly bounded away from $0$. From Lemma \ref{lemma:modulus_bound_convergence}, we conclude that $\Mod \Gamma^*( h_k^n(E), h_k^n(F_n); \br{D_k^n})$ is uniformly bounded above in $n\geq k$. Since $h_k^n$ is conformal, it follows that $\Mod \Gamma^*( E,F_n; \br B(0,r_k^n))$ is uniformly bounded above. On the other hand, $\Gamma^*( E,F_n; \br B(0,r_k^n))$ contains the circles $\partial B(0,r)$ for all $1<r<r_k^n$, so
$$\frac{1}{2\pi} \log \left(r_k^n\right)\leq  \Mod \Gamma^*( E,F_n; \br B(0,r_k^n)).$$
The boundedness of $r_k^n$ follows.

For fixed $k\in \N$, consider the sequence $g_n(z)= h_k^n(r_k^nz)$, $n\geq k$, from $\br{\mathbb D}$ onto $\br{D_k^n}$. We show that this sequence is normalized in the sense Theorem \ref{theorem:uniformization_long}, using the metric $d_k^n$ in the target. Note that the points $0$, $1/r_k^n$, and $-1$ of $\br{\mathbb D}$ have mutual distances uniformly bounded away from $0$ as $n\to \infty$. Moreover, we have $g_n(0)= p_n$, $g_n(1/r_k^n)=q_n$, and $g_n(-1)\in \partial D_k^n$, and by Step 1 the mutual distances of these points are also bounded away from $0$. Thus, the sequence $g_n$ is normalized, as claimed. 

\vspace{1em}

\noindent
\textit{Step 3: Weakly quasiconformal parametrizations.} By Theorem \ref{theorem:uniformization_long}, for each $k\in \N$, there exists a subsequence of $f_n\circ g_n$, $n\geq k$, that converges uniformly on $\br{\mathbb D}$ to a weakly $K$-quasiconformal map onto $\overline{D_k}$. Since $r_k^n$ is bounded above and below in $n\geq k$, we conclude that there exists a subsequential limit $r_k$ {of} $r_k^n$ such that the sequence $f_n\circ h_k^n$ has a subsequence that converges to a weakly $K$-quasiconformal map $h_k\colon \br B(0,r_k) \to \br{D_k}$. We remark that the modulus of curve families in $\br {D_k}$ is computed with respect to the metric $d_k$ here. Note that  $h_k(B(0,r_k))\supset D_k$ by Lemma \ref{lemma:monotone_boundary}. By passing to a diagonal subsequence, we assume that $r_k^n$ converges to $r_k$ and $f_n\circ h_k^n$ converges to $h_k$ for each $k\in \N$. 

\vspace{1em}

\noindent
\textit{Step 4: Normal families argument.} 
Now we fix $n\geq l\geq k$. In $B(0,r_k^n)$, we have 
$$f_n \circ h_l^n  \circ (h_l^n)^{-1}\circ h_k^n =f_n\circ h_k^n.$$
Note that the conformal embedding $(h_l^n)^{-1}\circ h_k^n\colon B(0,r_k^n)\to B(0,r_l^n)$ fixes $0$ and $1$, and that the balls $B(0,r_l^n)$ are uniformly bounded in $n$. By Montel's theorem \cite[Theorem 10.7, p.~160]{Mar:19}, as $n\to\infty$ these maps {subconverge}  locally uniformly to a conformal homeomorphism ${\varphi_{k,l}} \colon B(0,r_k)\to \Omega_{k,l}$, where ${\Omega_{k,l}} \subset B(0,r_l)$. Moreover, since $D_{k}^n\subset D_{k+1}^n$, we have $\Omega_{k,l} \subset \Omega_{k+1,l}$ for all $l \geq k+1$. By {passing to} a diagonal subsequence, we may assume that $(h_l^n)^{-1}\circ h_k^n$ converges to $\varphi_{k,l}$ for each $l\geq k$ and
$$h_l\circ \varphi_{k,l} =h_k$$
in $B(0,r_k)$ for all $l\geq k$. 

Next, note that for fixed $k\in \N$ and for $l\geq k$ the conformal maps $\varphi_{k,l}\colon B(0,r_k)\to \Omega_{k,l}$ form a normal family since they fix the points $0$ and $1$; see \cite[Exercise 12.29 (v), p.~441]{Bur:79}. Thus they converge along a subsequence to a conformal homeomorphism $\varphi_k\colon B(0,r_k)\to \Omega_k$, where $\Omega_k\subset \C$ is a simply connected domain. Hence, $h_l=h_k\circ \varphi_{k,l}^{-1}$ converges along a subsequence of $l\to \infty $ locally uniformly on $\Omega_k$ to the map $h=h_k\circ \varphi_k^{-1}\colon \Omega_k\to \br{D_k}$. Note that the limiting map $h$ is independent of $k$, since it was obtained as a limit of $h_l$. Moreover, $h(\Omega_k)=h_k(B(0,r_k)) \supset D_k$, and $h\colon \Omega_k\to h(\Omega_k)$ is weakly $K$-quasiconformal, where the modulus in $h(\Omega_k) \subset \br {D_k}$ is computed with the metric $d_k$.

By considering a diagonal sequence, we may obtain a map $h$ that maps $\Omega_k$ onto $h(\Omega_k)$ with $D_k\subset h(\Omega_k)\subset \br{D_k}$ for each $k\in \N$ and is weakly $K$-quasiconformal on $\Omega_k$ (using the metric $d_k$ in the image). Additionally, we have $\Omega_k\subset \Omega_{k+1}$. This is true by Carath\'eodory's kernel convergence theorem \cite[Chapter I, Theorem 1.8]{Pom:92}, since $\Omega_{k,l}$ converges to $\Omega_k$ as $l\to \infty$.
\vspace{1em}

\noindent
\textit{Step 5: The limiting parametrization.} 
Since $\Omega_k\subset \Omega_{k+1}$, the set $\Omega=\bigcup_{k=1}^\infty \Omega_k$ is a simply connected domain in $\C$. The map $h$ is a continuous map from $\Omega$ onto $X=\bigcup_{k=1}^\infty D_k$. Since $h|_{\Omega_k}$ is monotone and $\Omega_k\subset \Omega_{k+1}$, it follows that $h$ is monotone on $\Omega$. 

Finally, we argue {that $h$ is weakly quasiconformal} on $\Omega$. By the monotonicity of modulus, it suffices to show that if $\Gamma$ is a curve family contained in a compact subset of $\Omega$, then $\Mod \Gamma\leq K\Mod h(\Gamma)$. By continuity, $h(\Gamma)$ is contained in a compact subset of $X$, so there exists $k\in \N$ such that $h(\Gamma)\subset D_k$. Recall that $h|_{\Omega_k}$ is weakly $K$-quasiconformal, so we obtain the desired inequality but with the modulus of $h(\Gamma)$ computed in the metric $d_k$ rather than in $d_X$. However, $d_k$ is locally isometric to $d_X$ in $D_k$, so the modulus of $h(\Gamma)$ is the same in both metrics.

As {the final step}, by precomposing $h$ with a conformal map, we may {obtain} that $\Omega= \mathbb D$ or $\Omega=\C$.\qed

\bigskip

\section{Further properties of weakly quasiconformal mappings} \label{sec:further_properties}

In this section we establish further properties of weakly quasiconformal mappings $h\colon X \to Y$, where $X$ and $Y$ are metric surfaces with locally finite Hausdorff $2$-measure. Recall that $h$ is weakly quasiconformal if it is continuous, surjective, monotone, and there exists $K\geq 1$ such that for all curve families $\Gamma$ in $X$ we have
$$\Mod \Gamma\leq K \Mod h(\Gamma).$$

The main result in this section concerns the equivalence of different definitions of weak quasiconformality. In general, if $X$ and $Y$ are metric surfaces with locally finite Hausdorff $2$-measure and $h\in N^{1,2}_{\loc}(X,Y)$, then there exists a minimal weak upper gradient of $h$ that we denote by $g_h$; see \cite[Theorem 6.3.20, p.~162]{HKST:15}. 

\begin{thm}[Definitions of quasiconformality]\label{theorem:qc_definitions}
Let $X,Y$ be metric surfaces with locally finite Hausdorff $2$-measure and let $h\in N^{1,2}_{\loc}(X,Y)$ be a continuous and monotone mapping. The following are equivalent.
\begin{enumerate}[label=\normalfont(\roman*)]
    \item \label{item:qc_equivalence_i} For every Borel set  $ E \subset Y$ we have $$\int_{h^{-1}(E)}g_h^2 \, d\mathcal H^2 \leq K \mathcal H^2(E).$$
    \item \label{item:qc_equivalence_iii} The set function $\nu(E)=\mathcal H^2(h(E))$ is an outer regular, locally finite Borel measure on $X$. Moreover, if $J_h$ is the Radon--Nikodym derivative of $\nu$ with respect to $\mathcal H^2$, then for $\mathcal H^2$-a.e.\ $x\in X$ we have
    $$g_h(x)^2\leq KJ_h(x).$$ 
\end{enumerate}
\end{thm}

Combining this theorem with the result of Williams stated in Theorem \ref{theorem:qc_definitions_williams}, we obtain Theorem \ref{theorem:definitions_qc_1}; that is, $h$ is weakly $K$-quasiconformal if and only if it satisfies \ref{item:qc_equivalence_iii}.

The more intricate implication is from \ref{item:qc_equivalence_i} to  \ref{item:qc_equivalence_iii}, because $h$ is not  assumed to be a homeomorphism. One needs to make sense of the Jacobian of $h$ first. We note that it is not immediate that $E\mapsto \mathcal H^2(h(E))$ is a measure on $\Omega$, since $h$ is not a homeomorphism. Instead, we use the weak quasiconformality of $h$ to derive this.

\begin{rem}
We note that if one uses the Borel measure (see \cite[Theorem 2.10.10, p.~176]{Fed:69})
\begin{align*}
    \widetilde\nu(E)= \int_Y \# (h^{-1}(y)\cap E) \, d\mathcal H^2(y) \geq \mathcal H^2(h(E))
\end{align*}
in place of $\nu$, where $\# (A)$ denotes the cardinality of the set $A$, then the equivalence between \ref{item:qc_equivalence_i} and \ref{item:qc_equivalence_iii} is immediate provided that this measure is $\sigma$-finite (so that one can define its Radon--Nikodym derivative). The latter is guaranteed if $h$ is a homeomorphism, but it is not true in general under merely continuity and monotonicity. Our proof below in fact shows that $\widetilde \nu= \nu$ for weakly quasiconformal mappings, since $\# (h^{-1}(y))=1$ for $\mathcal H^2$-a.e. $y\in Y$; see Lemma \ref{lemma:injective}. 
\end{rem}

\begin{rem}
    The discussion in this section and the equivalence of definitions of weak quasiconformality can be generalized immediately to metric $n$-manifolds, $n\geq 3$, provided that an $n$-dimensional version of Lemma \ref{lemma:curves_continuum} holds. We are not aware of any such result in the literature, so we consider only the case of $2$-manifolds.
\end{rem}

The techniques used in the proof of Theorem \ref{theorem:qc_definitions} allow us to derive the following topological consequence.

\begin{thm}\label{theorem:homeomorphism}
Let $X,Y$ be metric surfaces without boundary and with locally finite Hausdorff $2$-measure and let $h\colon X\to Y$ be a weakly quasiconformal mapping. If the modulus of the family of non-constant curves passing through $y$ is zero for every $y\in Y$, then $h$ is a homeomorphism. Moreover, a sufficient condition for this property is that 
\begin{align*}
    \liminf_{r\to 0} \frac{\mathcal H^2(B(y,r))}{r^2}<\infty
\end{align*}
for every $y\in Y$.
\end{thm}
This result follows immediately from Lemma \ref{lemma:injective} below. 

\subsection{Proof of Theorem \ref{theorem:qc_definitions}}

In what follows, we assume that $X$ and $Y$ are metric surfaces with locally finite Hausdorff $2$-measure. We freely use the property that the Hausdorff $2$-measure is an outer regular Borel measure \cite[Section 2.10, p.~171]{Fed:69}. Moreover, recall that topological surfaces are second countable and separable and admit a compact exhaustion. Thus, the Hausdorff $2$-measure is $\sigma$-finite if it is locally finite. Let $B(Y)$ be the set of points $y\in Y$ such that
\begin{align*}
    \lim_{r\to 0} \frac{\mathcal H^2(B(y,r))}{r^2}=\infty.
\end{align*} 

\begin{lemm}\label{lemma:borel_measure_zero}
The set $B(Y)$ is Borel measurable and has Hausdorff $2$-measure zero. 
\end{lemm} 
\begin{proof}
The fact that $B(Y)$ has measure zero follows from \cite[2.10.19 (5), p.~181]{Fed:69}, which implies that there exists a uniform constant $C>0$ such that 
\begin{align*}
    \limsup_{r\to 0} \frac{\mathcal H^2(B(y,r))}{r^2}\leq C
\end{align*}
for a.e.\ $y\in Y$.

We prove the measurability statement.  For fixed $r>0$ the function $y\mapsto \mathcal H^2(B(y,r))$ is lower semi-continuous, thus Borel measurable. Indeed, by Fatou's lemma, whenever $y_n\to y$, we have
$$\mathcal H^2(B(y,r))= \int \chi_{B(y,r)} \, d\mathcal H^2 \leq \liminf_{n\to\infty} \int \chi_{B(y_n,r)} \, d\mathcal H^2.$$
Moreover, by the monotone convergence theorem we see that for fixed $y\in Y$ the function $r\mapsto\mathcal H^2(B(y,r))$ is left-continuous. We conclude that the function $f(y,r)=r^{-2}\mathcal H(B(y,r))$, $y\in Y$, $r>0$, is Borel measurable in $y$ and left-continuous in $r$. It now follows that the set $B(Y)=\{y\in Y: \lim_{r\to 0} f(y,r)=\infty\}$ is Borel measurable by writing
\begin{align*}
    B(Y)= \bigcap_{k=1}^\infty \bigcup_{n=1}^\infty \bigcap_{r\in (0,1/n)\cap \mathbb Q}\{y\in Y: f(y,r)>k\}.
\end{align*}
This completes the proof.
\end{proof}

We denote by $C(Y)$ the set of points of $y\in Y$ such that the modulus of the family of non-constant curves passing through $y$ is positive.

\begin{lemm}\label{lemma:curves_point}
We have $C(Y)\subset B(Y)$.
\end{lemm}

\begin{proof}
We show that for each $y\in Y\setminus B(Y)$, the modulus of the family of non-constant curves passing through $y$ is zero. {For each $\delta>0$, let $\Gamma_\delta$ denote the family of curves passing through $y$ with diameter bounded below by $\delta$. } By the subadditivity of modulus, it suffices to show the conclusion for {the family $\Gamma_\delta$ for all $\delta>0$.} 

Since $y\in Y\setminus B(Y)$, there exists $k>0$ such that $\mathcal H^2(B(y,r)) \leq kr^2$ for a sequence of arbitrarily small $r>0$. We also fix $N\in \N$.  Let $R_1<\delta/2$ be a radius such that $\mathcal H^2(B(y,R_1))\leq k R_1^2$, and let $r_1\coloneqq R_1/2$. In the annulus $A_1=A(y;r_1,R_1)\coloneqq \{ x\in Y:  r_1<d(x,y)<R_1\}$, we set $\rho =N^{-1}(R_1-r_1)^{-1}$. Now, consider $R_2<r_1$ so small that $\mathcal H^2(B(y,R_2))\leq k R_2^2$, define $r_2\coloneqq R_2/2$, and set $\rho=N^{-1}(R_2-r_2)^{-1}$ in the annulus $A_2=A(y;r_2,R_2)$. We repeat this procedure $N$ times, until we obtain a last annulus $A_N=A(y;r_N,R_N)$. We set $\rho=0$ outside the union of these annuli. 

Note that $\rho$ is an admissible function for $\Gamma_{\delta}$. Indeed, any curve $\gamma\in \Gamma_{\delta}$ connects $y$ to $\partial B(y,R_1)$, since $\diam(|\gamma|)\geq \delta>2R_1$. Thus, $\gamma$ intersects all annuli $A_i$, $i\in \{1,\dots,N\}$, with
$$\int_{\gamma}\chi_{A_i} \, ds \geq R_i-r_i.$$
for each $i\in \{1,\dots,N\}$. This implies admissibility. We now have
	\begin{align*}
		\Mod\Gamma_{\delta} \leq \int \rho^2 \, d\mathcal H^2 = \frac{1}{N^2} \sum_{j=1}^N \frac{\mathcal H^2(A_j)}{(R_j-r_j)^2}\leq \frac{k}{N^2} \sum_{j=1}^N \frac{R_j^2}{R_j^2/4}= \frac{4k}{N}.
	\end{align*}
This converges to $0$ as $N\to \infty$, completing the proof.
\end{proof}

\begin{lemm}\label{lemma:curves_continuum}
Suppose that $E,F\subset X$ are disjoint, non-trivial continua. Then the modulus of {the family of} curves connecting $E$ and $F$ is positive. 
\end{lemm}
This result requires that $X$ is a metric surface and thus has locally Euclidean topology. It can be proved by a slight modification of \cite[Proposition 3.5]{Raj:17}.  The idea is to consider a fixed curve $\gamma$ joining $E$ and $F$ and define the function $u\colon X \to \mathbb{R}$ by $u(x) = d(|\gamma|,x)$. Then there exists $T>0$ such that for almost every $t \in (0,T)$ the level set $u^{-1}(t)$ contains a rectifiable curve joining $E$ and $F$. The family of such curves has positive modulus.

\begin{lemm}\label{lemma:injective}
Suppose that $h\colon X\to Y$ is a continuous, non-constant, and monotone mapping such that 
$$\Mod \Gamma\leq K\Mod h(\Gamma)$$
for each curve family $\Gamma$ in $X$. Then $h$ is injective in $X\setminus h^{-1}(C(Y))$. In particular, $h$ is injective in $X\setminus h^{-1}(B(Y))$.  
\end{lemm}

This lemma proves Theorem \ref{theorem:homeomorphism}.  Indeed, the assumption of the theorem implies that $C(Y)=\emptyset$ or $B(Y)=\emptyset$. It follows from Lemma \ref{lemma:injective} that $h$ is injective on $X$. By the invariance of domain theorem, $h$ is a topological embedding.

\begin{proof}
It suffices to show that $h^{-1}(y)$ is a singleton for each $y\in h(X)\setminus C(Y)$. Suppose that $h^{-1}(y)$ contains more than one point. By the monotonicity of $h$, $h^{-1}(y)$ is a closed connected subset of $X$. By assumption, $h$ is non-constant, so $X\setminus h^{-1}(y)$ is a non-empty open set. Since $X$ is a metric surface, there exists a non-trivial continuum $E\subset X\setminus h^{-1}(y)$. Since $X$ is locally compact, $h^{-1}(y)$ contains a non-trivial continuum $F$. Let $\Gamma$ be the family of curves connecting $E$ and $F$. Then $\Mod \Gamma>0$ by Lemma \ref{lemma:curves_continuum}. By assumption, we have $\Mod h(\Gamma) >0$. Note that each curve of $h(\Gamma)$ is a non-constant curve joining $y$ to $h(E)$, and that $h(E)$ does not contain $y$. By the definition of $C(Y)$, we have $\Mod h(\Gamma)=0$, a contradiction. Thus, $h$ is injective in $X\setminus h^{-1}(C(Y))$. By Lemma \ref{lemma:curves_point}, we conclude that $h$ is injective on $X\setminus h^{-1}(B(Y))$.
\end{proof} 

\begin{cor}\label{corollary:measure}
Suppose that $h$ is as in Lemma \ref{lemma:injective}. If $E\subset X$ is a Borel set, then $h(E)\setminus B(Y)$ is a Borel set. Moreover, the set function $\nu(E)= \mathcal H^2(h(E))$ is an outer regular, locally finite, Borel measure on $X$.
\end{cor}
\begin{proof}
By Lemma \ref{lemma:borel_measure_zero}, $B(Y)$ is a Borel subset of $Y$. Since $h$ is continuous, $h^{-1}(B(Y))$ is a Borel subset of $X$. By Lemma \ref{lemma:injective}, $h$ is injective on $X\setminus h^{-1}(B(Y))$. By the Lusin--Souslin theorem \cite[Theorem 15.1, p.~89]{Ke:95} it follows that if $E$ is a Borel subset of $X \setminus h^{-1}(B(Y))$, then $h(E)$ is a Borel subset of $Y$. Now, if $E$ is any Borel subset of $X$, then 
$$h(E)\setminus B(Y)=h(E\setminus h^{-1}(B(Y)))=h(E\cap (X\setminus h^{-1}(B(Y)))),$$
which implies that $h(E)\setminus B(Y)$ is a Borel set.

Since $\mathcal H^2$ is a Borel measure on $Y$, it is immediate that the set function $\nu(E)= \mathcal H^2(h(E))$, restricted to $X\setminus h^{-1}(B(Y))$, where $h$ is injective, is a Borel measure. Since $\mathcal H^2(B(Y))=0$, it follows that $\nu$ extends to a Borel measure on $X$. 

Since $h$ is continuous and the Hausdorff $2$-measure of $Y$ is locally finite, it follows that $\nu(E)<\infty$ whenever $E$ is compact. For the outer regularity, recall that $\mathcal H^2$ is outer regular on $Y$. Thus, for any set $E\subset X$, there exists an open set $U$ in $Y$ containing $h(E)$ such that $\mathcal H^2(U)$ approximates $\mathcal H^2(h(E))$. The set $h^{-1}(U)$ is open by continuity and contains $E$. Moreover $\nu(h^{-1}(U))= \mathcal H^2(h(h^{-1}(U)))\leq \mathcal H^2(U)$, so the $\nu$-measure of $h^{-1}(U)$ approximates the $\nu$-measure of $E$, as desired.
\end{proof}

\begin{proof}[Proof of Theorem \ref{theorem:qc_definitions}]
We assume that $h\colon X\to Y$ is a non-constant mapping, otherwise the implications are trivial.

Suppose first that \ref{item:qc_equivalence_i} is true.  Lemma \ref{lemma:weak_upper_gradient_path_integral} implies that for each curve family $\Gamma$ in $X$ we have $\Mod \Gamma\leq K\Mod h(\Gamma)$. By Corollary \ref{corollary:measure}, $\nu=\mathcal H^2\circ h$ is an outer regular, locally finite, Borel measure on $X$. Now, by the Lebesgue--Radon--Nikodym decomposition theorem \cite[Theorem 3.8, p.~91]{Fol:99}, we can write $\nu= J_h d\mathcal H^2+ \nu_s$, where $J_h\in L^1_{\textrm{loc}}(X)$ and $\nu_s$ is a measure singular with respect to $\mathcal H^2$. 

By \ref{item:qc_equivalence_i}, we have
$$\int_{h^{-1}(E)}g_h^2\, d\mathcal H^2 \leq K  \mathcal H^2(E)$$
for each Borel set $E\subset Y$. Since $B(Y)$ has measure zero by Lemma \ref{lemma:borel_measure_zero}, we have $g_h=0$ a.e.\ on $h^{-1}(B(Y))$. Let $E\subset X$ be an arbitrary Borel set. By Corollary \ref{corollary:measure}, $h(E)\setminus B(Y)$ is a Borel set. Thus,
\begin{align*}
    \int_{E}g_h^2\, d\mathcal H^2&= \int_{E\setminus h^{-1}(B(Y))}g_h^2 \, d\mathcal H^2 \leq \int_{h^{-1}( h(E) \setminus B(Y))}g_h^2\, d\mathcal H^2\\&\leq K  \mathcal H^2( h(E)\setminus B(Y)) =K\mathcal H^2(h(E)).
\end{align*}
It follows that 
$$\int_E g_h^2\, d\mathcal H^2\leq  K \int_E J_h\, d\mathcal H^2 + K\nu_s(E).$$
The singularity of $\nu_s$ with respect to $\mathcal H^2$ implies that $g_h^2\leq K J_h$ a.e.\ in $X$ with respect to $\mathcal H^2$, as desired.

Conversely, if $g_h^2\leq KJ_h$, then for every Borel set $E\subset Y$ we have 
\begin{align*}
    \int_{h^{-1}(E)}g_h^2 \, d\mathcal H^2 &\leq K\int_{h^{-1}(E)}J_h\, d\mathcal H^2 \\&\leq K \nu( h^{-1}(E)) = K\mathcal H^2(h(h^{-1}(E)))\leq K \mathcal H^2(E).
\end{align*}
This proves \ref{item:qc_equivalence_i}.
\end{proof}

\bigskip

\section{Examples}\label{sec:examples}
In this section, we present concisely several known examples {illustrating some of the possible behavior of} weakly quasiconformal maps. We then give a detailed example in \Cref{example:distinction} showing that in the non-compact case of Theorem \ref{thm:one-sided_qc} there is no clear distinction between the situations where $\Omega=\C$ and $\Omega=\mathbb D$.  This {example verifies} Proposition \ref{prop:example} in the introduction.

\subsection{Example: Failure of the Lusin $(N)$ property}
{Let $(X,\mu)$, $(Y,\nu)$ be measure spaces. A mapping $f\colon X \to Y$ satisfies the \textit{Lusin $(N)$ property} if $\mu(E)=0$ implies $\nu(f(E)) =0$ for every measurable set $E \subset X$. Every metric surface becomes a measure space by giving it the Hausdorff $2$-measure.} Rajala \cite[Section 17]{Raj:17} proves that there exists a quasiconformal homeomorphism $h$ from a planar domain $\Omega$ onto a length surface $X\subset \R^3$ with $\mathcal H^2(X)<\infty$ such that $h$ maps a Cantor set of $2$-measure zero in $\Omega$ onto a Cantor set of positive Hausdorff $2$-measure in $X$. Thus, we cannot guarantee that the weakly quasiconformal mapping $h$ of Theorem \ref{thm:one-sided_qc} has the Lusin $(N)$ property, even if it is quasiconformal.

\subsection{Example: Collapsing a ball to point}
Let $B$ be a closed ball in $\C$ and consider the metric space $X$ obtained from $\C$ by identifying points in $B$, equipped with the quotient metric. Then $X$ is a length space homeomorphic to $\C$ with locally finite Hausdorff $2$-measure, and the natural projection $P\colon \C \to X$ is weakly $1$-quasiconformal. However, $X$ is not quasiconformally equivalent to any planar subset because there exists a point of $X$, namely the point $P(B)$, such that the modulus of the family non-constant curves passing through that point is positive. On the other hand, in the Euclidean plane, the modulus of non-constant curves passing through a given point is always zero. 

This {space} serves as an example where the Lusin $(N^{-1})$ property fails. That is, a set of positive $2$-measure in $\C$ is mapped by $P$ to a set of $2$-measure zero in $X$. Of course, $P$ is not a homeomorphism. A less trivial example is given in \cite{NR:20}. By inspecting the construction there, it can be shown that there exists a length surface $X\subset \R^3$ (with metric induced by the Euclidean metric) with locally finite Hausdorff $2$-measure such that there exists a weakly quasiconformal homeomorphism $h\colon \C \to X$ with the property  that a set of positive $2$-measure of $\C$ is mapped to a set of $2$-measure zero in $X$. In fact, $h$ is {the restriction of a global quasiconformal homeomorphism of $\mathbb{R}^3$.}

\subsection{Example: Traveling for free in a Cantor set}\label{example:cantor}
Let $\Omega\subset \C$ be a domain and $C$ be a totally disconnected, relatively closed subset of $\Omega$.  We consider the density $\chi_{\Omega\setminus C}$ {on $\Omega$}, which gives rise to a pseudometric on $\Omega$. Namely, 
$$d(x,y)= \inf_{\gamma} \int_{\gamma}\chi_{\Omega\setminus C} \, ds$$
where the infimum is taken over all rectifiable curves in $\Omega$ joining $x$ and $y$. 

\begin{prop}\label{prop:cantor}
The function $d\colon \Omega\times\Omega\to [0,\infty)$ is a metric with the following properties.
\begin{enumerate}[label=\normalfont(\roman*)]
    \item $(\Omega,d)$ is a length space with locally finite Hausdorff $2$-measure.
    \item The identity map $id \colon (\Omega, |\cdot |)\to (\Omega, d)$ is a homeomorphism that is locally $1$-Lipschitz and weakly $1$-quasiconformal on $\Omega$, and locally isometric on $\Omega\setminus C$.
    \item\label{prop:cantor:curve} Let $\gamma \colon [a,b]\to \Omega$ be a curve, denote by $[a_i,b_i]$, $i\in I$, the closures of the components of $\gamma^{-1}(\Omega\setminus C)$ and set $\gamma_i=\gamma|_{[a_i,b_i]}$, $i\in I$. If $\ell_{|\cdot |}(\gamma)<\infty$, then
\begin{align*}
    \ell_d(\gamma) =\sum_{i\in I}\ell_d(\gamma_i)=\sum_{i\in I}\ell_{|\cdot |}(\gamma_i)=\int_{\gamma}\chi_{\Omega\setminus C}\, ds 
\end{align*}
    and $\mathcal H^1(|\gamma|\cap C)=0$. Conversely, if $\mathcal H^1(|\gamma|\cap C)=0$, then 
    \begin{align*}
    \ell_d(\gamma) =\sum_{i\in I}\ell_d(\gamma_i)=\sum_{i\in I}\ell_{|\cdot |}(\gamma_i).
\end{align*}
\end{enumerate}
\end{prop}

If $C$ has positive area, then by following the argument of \cite[Example 2.1]{Raj:17},
it can be shown that $(\Omega,d)$ is not quasiconformally equivalent to any planar domain. We do not provide the details of the general case here, but in Example \ref{example:distinction} we give a specific example $(\Omega,d)$ with the same property. {The general question of when constructions of this type yield a surface quasiconformally equivalent to a planar domain has been investigated in \cite{IR:20}.}

\begin{proof}
We first show that $d$ is a metric that is topologically equivalent to the Euclidean metric. Let $x,y\in \Omega$. We trivially have $d(x,y)\leq |x-y|$ if the line segment between $x$ and $y$ is contained in $\Omega$. On the other hand, if $x$ and $y$ are distinct points, and $x\notin C $ or $y\notin C$, then $d(x,y)\geq \max\{\dist_{|\cdot |}(x,C), \dist_{|\cdot |}(y,C)\}>0$. If $x,y\in C$, then there exists a topological annulus $A \subset \Omega\setminus C$ separating $x$ from $y$ \cite[Corollary 3.11, p.~35]{Wh:64}. Hence $d(x,y)$ is bounded below by the distance of the boundary components of the annulus $A$. Thus, $d(x,y)>0$. If $\{x_n\}_{n\in \N}$ is a sequence in $\Omega$ with $d(x_n,x)\to 0$ for some $x\in \Omega$, then for any given annulus $A\subset \Omega\setminus C$ and all sufficiently large $n\in \N$, $x_n$ cannot be separated from $x$ by $A$. In fact, by the result referenced above, we can consider arbitrarily small such annuli surrounding $x$. We conclude that $x_n$ converges to $x$ in the Euclidean metric. This completes the proof of the topological equivalence of $d$ with the Euclidean metric. To summarize, $(\Omega , d)$ is a metric space such that the identity map $id \colon (\Omega, |\cdot |)\to (\Omega, d)$ is a locally $1$-Lipschitz homeomorphism. This also implies that $(\Omega, d)$ has locally finite Hausdorff $2$-measure. Moreover, $d$ is by definition locally isometric on $\Omega\setminus C$ to the Euclidean metric. 

Next, we show that the identity map $id \colon (\Omega, |\cdot |)\to (\Omega, d)$ is weakly $1$-quasiconformal. The function  $g=\chi_{\Omega\setminus C}$ is trivially an upper gradient of $id$ and for any Borel set $E\subset \Omega$ we have 
$$\int_E g^2 \, d\mathcal H^2_{d} =\mathcal H^2_{d}(E\setminus C) \leq \mathcal H^2_{d}(E)  \leq \mathcal H^2_{|\cdot |}(E),$$
since $id$ is locally $1$-Lipschitz. We now employ Lemma \ref{lemma:weak_upper_gradient_path_integral} \ref{lemma:weak_upper_gradient_measure_ineq}, which implies that $id$ is weakly $1$-quasiconformal.

It remains to show \ref{prop:cantor:curve}, which also implies that $d$ is a length metric. First, suppose that $\gamma$ is rectifiable with respect to the Euclidean metric. Then
\begin{align*}
    \ell_d(\gamma) \geq \sum_{i\in I}\ell_d(\gamma_i)=\sum_{i\in I}\ell_{|\cdot |}(\gamma_i)= \int_{\gamma}\chi_{\Omega\setminus C}\, ds.
\end{align*}
For the reverse inequality, let $a=t_0<t_1<\dots<t_n=b$ be a partition of $[a,b]$, and note that by the definition of $d$ we have
\begin{align*}
    \sum_{j=1}^n d(\gamma(t_{j-1}),\gamma(t_j)) \leq \sum_{j=1}^n \int_{\gamma|_{[t_{j-1}, t_j]}} \chi_{\Omega\setminus C}\, ds = \int_{\gamma} \chi_{\Omega\setminus C}\, ds.
\end{align*}
This shows the first part of \ref{prop:cantor:curve}.  

Next, we recall the area formula
\begin{align*}
    \ell_d(\gamma)= \int_{\Omega} \mathcal \# ( \gamma^{-1}(x)) \, d\mathcal H^1_d(x),
\end{align*}
where $\# (A)$ denotes the cardinality of the set $A$. See \cite[Theorem 2.10.13, p.~177]{Fed:69} for a proof. In general, we have
\begin{align*}
    &\int_{\Omega} \mathcal \# ( \gamma^{-1}(x)) \, d\mathcal H^1_d(x)\\
    &\qquad\qquad= \sum_{i\in I }\int_{\Omega\setminus C} \mathcal \# ( \gamma^{-1}(x) \cap (a_i,b_i)) \, d\mathcal H^1_d(x) +\int_{C} \mathcal \# ( \gamma^{-1}(x)) \, d\mathcal H^1_d(x)\\
    &\qquad\qquad= \sum_{i\in I } \ell_{d}(\gamma_i) +\int_{C} \mathcal \# ( \gamma^{-1}(x)) \, d\mathcal H^1_d(x).
\end{align*}
If $\ell_{|\cdot|}(\gamma)<\infty$, then the left-hand side is finite and equal to $ \sum_{i\in I } \ell_{d}(\gamma_i)$ by the previous, so $\int_{C} \mathcal \# ( \gamma^{-1}(x)) \, d\mathcal H^1_d(x)=0$, which is equivalent to $\mathcal H^1(|\gamma|\cap C)=0$. Conversely, if $\mathcal H^1(|\gamma|\cap C)=0$, then by the area formula we obtain that $\ell_d(\gamma)= \sum_{i\in I} \ell_d(\gamma_i)=\sum_{i\in I} \ell_{|\cdot |}(\gamma_i)$.
\end{proof}

\subsection{Example: No distinction between the plane and the disk}\label{example:distinction}
We show that there exists a length surface $X$ homeomorphic to $\C$ with locally finite Hausdorff $2$-measure such that there exist two weakly $1$-quasiconformal maps, one from $\mathbb D$ onto $X$ and one from $\C$ onto $X$. This proves \Cref{prop:example}. {See Example 6.2 in \cite{CR:20} for a similar construction in the context of Plateau's problem for metric spaces.}

In fact, such a space $X$ cannot be quasiconformally mapped to a planar domain. Indeed, suppose there were a quasiconformal map from $X$ onto a planar domain $\Omega$. By postcomposing with a conformal map, we may assume that $\Omega=\mathbb D$ or $\Omega=\C$. Since there exist weakly quasiconformal maps from $\mathbb D$ onto $X$ and from $\C$ onto $X$, we obtain a weakly quasiconformal map $f$ from $\mathbb D$ onto $\C$ or from $\C$ onto $\mathbb D$. In fact, Theorem \ref{theorem:homeomorphism} implies that $f$ is a homeomorphism. It is {well-known} that a weakly quasiconformal homeomorphism between planar domains is quasiconformal. Specifically, by definition, a quasiconformal homeomorphism between Euclidean domains is a priori required to satisfy only one modulus inequality; see \cite[Section 3]{LV:73}. Thus, we obtain a contradiction by Liouville's theorem {for quasiconformal mappings \cite[Theorem 17.4]{Vai:71}}.

Next, we describe the construction of the space $X$, which relies on the next lemma.

\begin{lemm}\label{lemma:nonremovable}
There exists a totally disconnected, closed set $C\subset \mathbb C$ that is contained in the the union of countably many rectifiable curves, a domain $V\subset \mathbb D$ one of whose boundary components is $\partial \mathbb D$, and a conformal map $g$ from $V$ onto the domain $U=\mathbb C\setminus C$ such that $g(z) \to \infty $ as $z\to \partial \mathbb D$.
\end{lemm}

For the proof we recall the general fact that a homeomorphism $f\colon U\to V$ between domains $U,V\subset \widehat{\C}$ extends to a bijection $f^*$ between the boundary components of $U$ and the boundary components of $V$. Namely, if $B$ is a boundary component of $U$, then $f^*(B)$ is precisely the boundary component $B^*$ of $V$ with the property that $f(z_n)$ accumulates at $B^*$ whenever $\{z_n\}_{n\in \N}$ is a sequence in $U$ accumulating at $B$. Moreover, $f^{-1}$ extends analogously to a bijection $(f^{-1})^*$ between the boundary components of $V$ and $U$ and has the property $(f^{-1})^*=(f^*)^{-1}$. See \cite[Proposition 3.1]{NY:20} for a proof of these standard facts.

\begin{proof}
Let $C_0\subset \R$ be a linear Cantor set that is not removable for conformal maps. Such sets have been studied by Ahlfors and Beurling \cite{AB:50}. By definition, there exists a non-M\"obius conformal homeomorphism $f_0$ from $U_0={\C}\setminus C_0$ onto a domain $V_0$ in ${\C}$. 

If all boundary components of $V_0$ are points, then in fact $f_0$ extends to a homeomorphism from ${\C}$ onto ${\C}$ that is conformal on ${\C}\setminus C_0$. Since $C_0$ has finite length, it follows that $C_0$ is removable for continuous analytic functions \cite[Theorem 2]{Bes:31} and thus $f_0$ extends conformally to ${\C}$; alternatively one can argue using the fact that $C_0$ is removable for conformal homeomorphisms \cite[Theorem 35.1]{Vai:71}. Hence $f_0$ is a M\"obius transformation, a contradiction.

Therefore, there exists a boundary component of $V_0$ that is a non-de\-generate continuum $E$. We consider a conformal map $\psi$ from the simply connected domain $\widehat{\C}\setminus E$ onto the unit disk $\mathbb D$. We set $V=\psi (V_0)$. Note that $\partial \mathbb D$ is a boundary component of $V$ that corresponds to $E$.

Next, consider a M\"obius transformation $\varphi$ of $\widehat\C$ that maps the boundary point $(f^*_0)^{-1}(E)$ of $U_0$ to $\infty$. Thus, $\varphi$ maps a point of $C_0=\partial U_0$ to $\infty$ and $\varphi(C_0)$ is contained in a great circle through $\infty$. We set $U=\varphi(U_0)$ and note that the set $C=\partial U\cap \C$ is totally disconnected and is contained in two locally rectifiable curves passing through $\infty$, and thus in the union of countably many rectifiable curves in the plane. Then the map $g=\varphi \circ f_0^{-1} \circ \psi^{-1} \colon V\to U$ has the desired properties.
\end{proof}

Consider the set $C$, the domains $U=\C\setminus C$, $V\subset \mathbb D$, and the map $g\colon V \to U$ as in Lemma \ref{lemma:nonremovable}. By Example \ref{example:cantor}, there exists a length space $(\C, d)$ arising from the density $\chi_U$ such that $(\C,d)$ has locally finite Hausdorff $2$-measure and the identity map $id\colon (\C,|\cdot |)\to (\C, d) $ is a weakly $1$-quasiconformal homeomorphism. The fact that $\partial U$ is contained in countably many  rectifiable paths (in the Euclidean metric), together with Proposition \ref{prop:cantor} \ref{prop:cantor:curve}, imply that $\mathcal H^1_d(\partial U)=0$.

Since all boundary components of $U$ are points and $g^*$ is a bijection between the boundary components of $V$ and $U$, it follows that $g$ extends continuously to a map from $\mathbb D$ onto $\C$. We denote the extension by $g$. Moreover, $g$ is a monotone map. Indeed, the preimage of each point of $\C$ under $g$ is either a point or connected component of $\partial V\cap \mathbb D$.

We claim that $g\colon  (\mathbb D, |\cdot |) \to (\C, d)$ is a weakly $1$-quasiconformal map. Suppose that this is the case and define $X$ to be the space $(\mathbb C, d)$. Then the maps 
\begin{align*}
    id\colon (\mathbb C,|\cdot |)\to X \quad \textrm{and}\quad g\colon (\mathbb D, |\cdot |)\to X
\end{align*}
are both weakly $1$-quasiconformal. This concludes the construction.

Now we prove the claim that $g$ is weakly $1$-quasiconformal. This follows from Lemma \ref{lemma:weak_upper_gradient_path_integral} \ref{lemma:weak_upper_gradient_measure_ineq} upon verifying that  $|g'|\chi_V$ is an upper gradient of $g$ and for any Borel set $E\subset \mathbb C$ we have
\begin{align}\label{example:measure}
    \int_{g^{-1}(E)} |g'|^2 \chi_V \, d\mathcal H^2_{|\cdot|}\leq  \mathcal H^2_{d}(E). 
\end{align}

First, we verify \eqref{example:measure}. Since $g^{-1}(\partial U)=\partial V\cap \mathbb D$, it suffices to verify the inequality for Borel sets $E\subset U$. Since the Hausdorff $2$-measure agrees there with the Lebesgue measure, the desired inequality follows from the conformality of $g$ on $V$ and the fact that $d$ is locally isometric to the Euclidean metric in $U$. 

Next, we prove that $|g'|\chi_V$ is an upper gradient of $g$. It is crucial here that  $\mathcal H^1_{d}(\partial U)=0$.  It suffices to prove that for every path $\gamma\colon [a,b]\to \mathbb D$ that is rectifiable with respect to the Euclidean metric we have
\begin{align*}
    d( g(\gamma(a)) ,g(\gamma(b))) \leq \int_{\gamma} |g'|\chi_V \, ds.
\end{align*}
Let $[a_i,b_i]$, $i\in I$, be the closures of the components of $\gamma^{-1}(V)$ and consider the subpaths $\gamma_i =\gamma|_{[a_i,b_i]}$, $i\in I$, of $\gamma$. Since $\mathcal H^1_{d}(\partial U)=0$, by Proposition \ref{prop:cantor} \ref{prop:cantor:curve}, we have
\begin{align*}
    d( g(\gamma(a)) ,g(\gamma(b))) &\leq \ell_{d}( g\circ \gamma) = \sum_{i\in I }\ell_{|\cdot |}( g\circ \gamma_i)\\
    &=\sum_{i\in I} \int_{\gamma_i} |g'|\, ds=\sum_{i\in I} \int_{\gamma_i} |g'| \chi_{V}\, ds\leq \int_{\gamma} |g'|\chi_V \, ds.
\end{align*}
This completes the proof.\qed

\bigskip

\bibliographystyle{abbrv}
\bibliography{bibliography}

\end{document}